\documentclass[11pt]{article}
\usepackage[top=1.1in, left=1in, right=1in, bottom=1.1in]{geometry}
\usepackage{amsmath,amssymb,amsthm,amsfonts,verbatim}
\usepackage{enumerate}
\usepackage{microtype}
\usepackage{url}
\usepackage{algpseudocode}
\usepackage{tikz}
\usetikzlibrary{arrows}
\usepackage{amsmath,amssymb}
\usepackage[justification=centering]{caption}
\usepackage[justification=centering]{subcaption}
\usepackage[english]{babel}
\usepackage{array}
\usepackage{psfrag}
\usepackage{graphicx}
\usepackage{color}
\usepackage{setspace}
\usepackage{hyperref}
\usepackage[utf8]{inputenc}
\usepackage{cleveref}
\usepackage{mathrsfs} 
\usepackage{authblk}
\usepackage{threeparttable}
\usepackage[titletoc,title]{appendix}
\usepackage{xcolor}

\newtheorem{lemma}{Lemma}[section]

\newtheorem{corollary}{Corollary}[section]
\newtheorem{assumption}{Assumption}
\newtheorem{theorem}{Theorem}[section]

\newtheorem{definition}[theorem]{\bf Definition}

\newtheorem{remark}{Remark}[section]
\newtheorem{example}{Example}[section]
\definecolor{red}{rgb}{1,0.2,0.2}

\newcounter{algorithmcounter}

\newenvironment{algorithm}{\begin{quote}%
    \refstepcounter{algorithmcounter}%
  \textbf{Algorithm \arabic{algorithmcounter}}%
  \quad
}{%
\end{quote}%
}

\makeatletter
\newcommand*{\rom}[1]{\expandafter\@slowromancap\romannumeral #1@} 
\makeatother

\hypersetup{colorlinks=true,linkbordercolor=red,linkcolor=blue,pdfborderstyle={/S/U/W 1}}

\title{A Graph-Algorithmic Approach\\ for the Study of Metastability 
in Markov Chains}

\author[1]{Tingyue Gan \thanks{tgan@math.umd.edu}}
\author[1]{Maria Cameron \thanks{cameron@math.umd.edu}}
\affil[1]{Department of Mathematics, University of Maryland, College Park, MD 20742}
 
 \date{\today}   

\begin{document}

\maketitle

\begin{abstract}
Large continuous-time Markov chains with exponentially small 
transition rates arise in modeling  complex systems in physics, chemistry and 
biology. We propose a constructive graph-algorithmic approach to determine the sequence of critical timescales
 at which the qualitative behavior of a given Markov chain changes, and give an effective description of the dynamics on each of them.
This approach is valid for both time-reversible and time-irreversible Markov processes, with or without symmetry.
Central to this approach are two graph algorithms, Algorithm 1 and Algorithm 2, 
for obtaining the sequences of the critical timescales 
and the hierarchies of Typical Transition Graphs or T-graphs { indicating the most likely transitions} in the system
 {without and with} symmetry respectively. 
The sequence of {critical} timescales includes the subsequence of the reciprocals { of the real parts } of eigenvalues.
{Under a certain assumption}, we prove sharp asymptotic estimates for eigenvalues 
(including prefactors) and show how one can extract them 
from the output of Algorithm 1. We discuss the relationship between Algorithms 1 and 2, and explain how one needs to interpret the 
output of Algorithm 1 if it is applied in the case with symmetry { instead of Algorithm 2}.
Finally, we analyze an example motivated  by R. D. Astumian's model of the dynamics of kinesin, a molecular motor, 
{ by means of   Algorithm 2}. 
\end{abstract}

\section{Introduction}

Phase transitions in non-equilibrium systems, conformational changes in molecules or atomic clusters, 
financial crises, global climate changes, and  genetic mutations exemplify the phenomenon 
where a seemingly stable behavior of 
a system in hand, persisting for a long time, undergoes a sudden qualitative change.
Such systems are often referred to as metastable.
A popular choice of mathematical model for investigating such systems is  a Markov jump process with  a finite number of states 
and exponentially distributed holding times. The dynamics of the process is described by the generator matrix $L$. 
Each off-diagonal entry of $L$ is  the transition rate from state $i$ to state $j$.
Often, it takes the form  \cite{wales0,FW}
\begin{equation}
\label{Lij}
L_{ij} = \kappa_{ij}\exp(-U_{ij}/\varepsilon),
\end{equation}
where { $\kappa_{ij} > 0$ is the pre-factor, the number $U_{ij} > 0$} is the exponential factor or order, 
and  $\varepsilon > 0$ is a small parameter. In many cases, the pre-factors $\kappa_{ij}$ are not available, and the rates are 
determined only up to the exponential order:
\begin{equation}
\label{Lij1}
L_{ij}\asymp\exp(-U_{ij}/\varepsilon),~~{\rm i.e.}~~{ \lim_{\varepsilon\rightarrow 0}\varepsilon\log L_{ij}(\varepsilon) = -U_{ij}}.
\end{equation}

The reciprocal $L_{ij}^{-1}$ is the expected waiting time for a jump from state $i$ to state $j$.
The diagonal entries $L_{i i}$ are defined so that the row sums of $L$ are zeros, i.e.,  $L_{ii} = - \sum_{j \neq i} L_{ij}$. 
Hence, $-L_{ii}$ is the  escape rate from state $i$. 

A Markov chain with pairwise rates of the form \eqref{Lij} or \eqref{Lij1} can be associated with a weighted directed graph $G(\mathcal{S},\mathcal{A},\mathcal{U})$
where the set of vertices $\mathcal{S}$ is the set of states, the set of arcs $\mathcal{A}$ includes only such arcs $i\rightarrow j$ that $L_{ij}>0$,
and  $\mathcal{U}$ is the set of arc weights. An arc $i\rightarrow j$ has weight $U_{ij}$ if { $L_{ij}\asymp\exp(-U_{ij}/\varepsilon)$}.
We set $U_{ij}=+\infty$ if $L_{ij} = 0$, i.e., if there is no arc $i\rightarrow j$ in the graph.

In this work, we will consider continuous-time Markov chains with finite numbers of states  
and pairwise transition rates of the form \eqref{Lij} or \eqref{Lij1}. 
{ Under this framework, the timescales on which various transition processes take 
place in the system are well-separated as $\varepsilon$ tends to zero. }
Our goal is to find a constructive way to calculate the sequence of
critical timescales at which the  behavior of such a Markov chain undergoes qualitative changes, and give effective descriptions of
its behavior on the whole range of timescales from zero to infinity. 
Imagine that the time evolution of 
a given Markov chain is observed for some not very large fixed number of time units. We want to be able to predict 
what transitions the observer will see 
depending on the initial state and the size of the time unit. Such a prediction is easy if the time unit tends to zero.
Then { it is extremely unlikely to observe any transitions}. If the time unit tends to infinity, 
then an equilibrium probability distribution will be observed.  
However, even in this case, the determination of arcs along which the transitions will be most likely observed
remains a nontrivial problem. On timescales between those two extremes, the
problem of giving an effective description of the 
observable transitions is { difficult}.  
Prior to presenting our solution to it, we give an account of works that provided us with necessary background 
{and/or} inspiration.

\subsection{Background}
Long-time behavior of stochastic systems with rare transitions has been attracting attention of mathematicians for the last fifty years. 
Freidlin and Wentzell 
developed the Large Deviation Theory \cite{FW} in 1970s. They 
showed that the long-time behavior of a system 
evolving according to the SDE
$$
dX_{t}=b(X_t)dt+\sqrt{\varepsilon}dW_t
$$
 can be modeled by means of continuous-time 
Markov chains. The  states of the Markov chain correspond to attractors of  the corresponding ODE $\dot{x}=b(x)$.
The pairwise rates are  logarithmically equivalent to $\exp(-U_{ij}/\varepsilon)$, where $U_{ij}$ is the quasi-potential, 
the quantity characterizing the difficulty of the passage from attractor $i$ to attractor $j$.
Recently, Buchet and Reygner \cite{bouchet} calculated the pre-factor $\kappa_{ij}$ in the case where 
state $i$ represents a stable equilibrium point separated from the attractor corresponding to state $j$ by a Morse index one saddle.

In early 1970s, Freidlin proposed to describe the long-time behavior of the system via a hierarchy of cycles (we refer to them as Freidlin's cycles)
\cite{freidlin-cycles,freidlin-physicad,FW} in the case where { each cycle has a 
unique main state, a unique exit arc, and the exit rates for all cycles are distinct.}
This hierarchy can be mapped onto a tree. In the time-reversible case, this tree is a complete binary tree \cite{cam1}. 
Later, in 2014, Freidlin extended this approach to the case with symmetry \cite{freidlin_symmetry}
replacing the hierarchy of cycles with the hierarchy of Markov chains.
Each cycle{/}Markov chain in  Freidlin's hierarchy  is born at a specific critical  timescale, {which is the reciprocal of its rotation rate}. 
The corresponding hierarchy of timescales has only partial but not complete order:
cycles{/Markov chains} of the same order typically have different  timescales. 
 {The birth timescales of Freidlin's cycles/Markov chains constitute an important subset of critical timescales of the system.}

{The other important subset of critical timescales is given by the reciprocals of the absolute values of the real parts} of
nonzero eigenvalues of the generator matrix $L$. 
Using the classic potential theory as a tool, Bovier and collaborators \cite{Bovier_1, Bovier_2, Bovier_3,Bovier2016} 
derived sharp estimates for low lying spectra of time-reversible Markov chains (all their eigenvalues are real) with pairwise rates not necessarily of the form \eqref{Lij},
{ defined a hierarchy of metastable sets}, and  
 identified the link between eigenvalues and expected exit times.  
A more general study, utilizing almost degenerate perturbation theory, 
was  conducted by Gaveau and Schulman \cite{GS}, 
in which a spectral definition of metastability was given for a broad class of Markov chains. 
The Transition Path Theory proposed by E and Vanden-Eijnden and extended to Markov chains by Metzner et al \cite{dept}, 
can be viewed as an extension of the potential-theoretic approach 
 exercised by Bovier et al in the time-irreversible case. It is focused on the study of transition pathways 
 between any particular pair of metastable sets.

Asymptotic estimates of the exponential orders of {the real parts} of eigenvalues of the generator matrix  of any Markov chain
with pairwise rates of the form \eqref{Lij1} 
were developed by Wentzell in early 1970s \cite{Wentzell}. 
 { These estimates are given in terms of the optimal W-graphs, i.e.,
solutions of certain combinatoric optimization problems on so-called W-graphs.
Assuming the more concrete  form \eqref{Lij} of the pairwise transition rates, we have derived asymptotic estimates for
eigenvalues including pre-factors  for the case  where all optimal W-graphs are unique.
Greedy graph algorithms 
to solve these optimization problems in case where  all optimal W-graphs are unique were introduced in \cite{Cameron} and \cite{CG}. 
The one in \cite{Cameron} assumes time-reversibility and finds the sequence of asymptotic estimates for eigenvalues starting from the smallest ones.
The greedy/dynamical programing ``single-sweep algorithm" in \cite{CG} does not require time-reversibility and computes the sequence of 
asymptotic estimates for eigenvalues starting from the largest ones. }
Sharp estimates for eigenvalues for the time-reversible case with symmetry were obtained by Berglund and Dutercq 
using tools from the group representation theory \cite{berglund}.

\subsection{Summary of main results}
The starting point of this work is the single-sweep algorithm. { 
In \cite{CG}, we introduced } it in the form convenient for programming and used it
only for the purpose of finding asymptotic estimates for eigenvalues and eigenvectors of the large time-reversible 
Markov chain with 169523 states representing the energy landscape of the Lennard-Jones-75 cluster. 

 In this work, we extend the single-sweep algorithm \cite{CG}
{ for finding the whole sequence of critical timescales corresponding to both,  births of Freidlin's cycles and reciprocals of 
the absolute values of  the real parts of eigenvalues.
Besides the set of critical timescales, the output will include  
 the hierarchy of \emph{Typical Transition Graphs} (\emph{T-graphs}) introduced in this work, 
 that mark the transitions most likely to observe up to certain timescales.
Wentzell's optimal W-graphs are readily extracted from the T-graphs.}
We will refer to this extension of the single-sweep algorithm as Algorithm 1.
The presentation of Algorithm 1 is significantly different from the one in \cite{CG}.
Each step of Algorithm 1 is motivated by the consideration of 
the dynamics of the system at an appropriate timescale.

Algorithm 1 offers a constructive way to simultaneously solve both Freidlin's problem of building the hierarchy of cycles and
Wentzell's problem of finding asymptotic estimates for eigenvalues. 
Contrary to 
\cite{freidlin-cycles,freidlin-physicad,FW},
Freidlin's cycles are found in the decreasing order of their { rotation rates}.
Asymptotic estimates for eigenvalues, containing pre-factors (if so do the input data), 
{ are found during the run} of Algorithm 1.
Our proof for sharp asymptotic estimates of eigenvalues of the generator matrix for 
time-irreversible Markov chains for which all optimal W-graphs are unique is provided.

{  From the programming  point of view,
Algorithm 1 and the single-sweep algorithm in \cite{CG} differ in their stopping criteria. 
Algorithm 1 stops as it computes all Freidlin's cycles.
The costs of the single-sweep algorithm and Algorithm 1 and the difference between them  depend on the structure
of the graph. 
For a graph with $n$ vertices and maximal vertex degree $d$, the cost of both algorithms are at worst $O((nd)^2\log(nd))$. 
}

{ Algorithm 1 is designed for the case with \emph{no symmetry, i.e.,
where all critical timescales are distinct and all Freidlin's cycles have unique exit arcs}.
Markov chains arising in modeling complex physical systems might or might not be symmetric.
For example, the networks representing energy landscapes of biomolecules and clusters of particles interacting 
according to a long-range potential \cite{wales_book,wales_landscapes,pathsample,web} are mostly non-symmetric, 
while the networks representing the dynamics of  particles interacting according to
a short-range potential \cite{short1,short2,short3,holmes} are highly symmetric.
}

{ To handle the case with symmetry, we have developed a modification of Algorithm 1 and called it Algorithm 2.
Algorithm 2 computes the sequence of distinct values of critical timescales and the hierarchy of 
T-graphs.  
}

{ The presence of symmetry is not necessarily apparent in a Markov chain.
Algorithm 1 can run in the case with symmetry and produce an output that does not reflect its presence.
However, the output will be inaccurate in some aspects. 
The relationship between the outputs of Algorithms 1 and 2
and a recipe for the interpretation of the output of Algorithm 1 used in a symmetric case 
is summarized in a theorem proven in this work.}

Algorithms 1 and 2 constitute a graph-algorithmic approach to the study of metastability in 
continuous-time Markov chains with exponentially small transition rates. 

Algorithms 1 and 2 are illustrated on examples. 
{The case with symmetry is not as graphic as the one without it, however, 
it is important, as symmetry often occurs in Markov chains modeling natural systems.
One such example,} motivated by Astumian's model \cite{Astumian} of the directed motion of kinesin protein, a molecular motor, is analyzed 
by means of Algorithm 2. The stochastic switch between two chemical states, breaking the detailed balance in this system, 
enables the directed motion (walking). The rate of the chemical switch is treated as a parameter.
The most likely walking style  is indicated for each rate value.

The rest of the paper is organized as follows. 
Some necessary background on continuous-time Markov chains and optimal W-graphs is provided in Section \ref{sec2}.
{The T-graphs are introduced in Section \ref{sec:Tgraphs}.}
Algorithm 1 is presented and discussed in Section \ref{sec:alg1}. 
In Section \ref{sec:sym}, we address the case with symmetry and
introduce Algorithm 2.
The interpretation of the output of Algorithm 1 applied in the case with symmetry is given in Section \ref{sec:sym1}.
The molecular motor example is investigated in Section \ref{sec:MM}. { We summarize our work in Section \ref{sec:conclusion}.}
Appendices \ref{App_A} and \ref{App_B} contain proofs of theorems.

\section{Significance and nested property of optimal W-graphs}
\label{sec2}
In this Section, we provide some necessary background and discuss some of our recent results that are essential for 
the presentation of our new developments.

\subsection{Continuous-time Markov chains}
{ Let $G(\mathcal{S},\mathcal{A},\mathcal{U})$ be a } weighted directed graph  
associated
with a given continuous-time Markov chain with pairwise transition rates of the form \eqref{Lij} or \eqref{Lij1}.
Throughout the rest of the paper we adopt the following assumptions.
\begin{assumption}
\label{A1}
The set of states $\mathcal{S}$ is finite and $|S|=n$.
\end{assumption}
\begin{assumption}
\label{A2}
The graph $G(\mathcal{S},\mathcal{A},\mathcal{U})$ has a unique closed communicating class.
\end{assumption}
We remind (see e.g. \cite{norris}) that  a subset of vertices $C$ is called a \emph{communicating class}, if there is a directed path  leading from any vertex $i\in C$
to any vertex $j\in C$. 
A communicating class $C$ is called \emph{closed} if for any vertex $i\in C$ the following condition holds: 
if there is a directed path from $i\in C$ to $j$, then $j\in C$.
A trivial example of a closed communicating class is an \emph{absorbing state}, i.e., a state with no outgoing arcs. 
All states of an \emph{irreducible} Markov chain belong to the same closed communicating class.

In a weighted directed graph $G(\mathcal{S},\mathcal{A},\mathcal{U})$ with a single closed communicating class $C$,
all states in $C$ are \emph{recurrent}, while the rest of the states $\mathcal{S}\backslash C$ are \emph{transient}.
Assumption \ref{A2} guarantees that the invariant distribution 
 $\pi$ satisfying $\pi L=0$, $\pi_i\ge 0$, $i\in \mathcal{S}$, and $\sum_{i\in\mathcal{S}}\pi_i = 1$,
 is unique and supported only on the recurrent states.
Note that $\pi$ is the left eigenvector of $L$ corresponding to the zero eigenvalue $\lambda_0=0$.
Due to the zero row sum property of $L$,  the corresponding right eigenvector   is  $\phi_0= [1,\ldots,1]^T$.

Due to the weak diagonal dominance of $L$ and 
non-positivity of its diagonal entries, all its  eigenvalues have non-positive real parts. 
 This can be shown, e.g., by applying Gershgorin's circle theorem \cite{Gershgorin}. 
{ Let $z_m = -\lambda_m + i\mu_m$, $m = 1,\ldots, n-1$,  be nonzero eigenvalues of $L$, ordered so that $0<\lambda_1\le\ldots\le\lambda_{n-1}$, 
and column vectors $\phi_m$ and  row vectors $\psi_m$ be the corresponding right and left eigenvectors. }
If $L$ is diagonalizable, one can write the probability distribution $p(t)$, 
governed by the Fokker-Planck equation $d p/ d t = p L$, in the following form:
\begin{align}
p(t) = p(0) e^{t L} =  \pi + \sum_{m = 1}^{n-1} e^{-\lambda_m t}e^{i\mu_mt}(p(0) \phi_m ) \psi_m.  \label{pdf}
\end{align}
Considering the time as a function of $\varepsilon$, one can infer from Eq. \eqref{pdf} that the $k$th eigen-component of $p(t(\varepsilon))$
 is significant only  for  time $ t(\epsilon)$
 of an exponential order not greater than the one of $\lambda_m^{-1}(\epsilon)$. 


\subsection{Optimal W-graphs}
\label{sec:owg}
The W-graphs were introduced by Wentzell \cite{Wentzell} in order to obtain asymptotic estimates for exponential factors of real parts of 
eigenvalues of the generator matrices of Markov chains with pairwise rates of the form \eqref{Lij1}. The W-graphs  generalize  
the $i$-graphs introduced by Freidlin \cite{freidlin-cycles,freidlin-physicad}
for building the hierarchy of cycles describing the long-term behavior of such Markov chains.

\begin{definition} {\bf (W-graph)}
Let $G(\mathcal{S}, \mathcal{A},\mathcal{U}) $ be a weighted directed graph satisfying Assumptions \ref{A1} and \ref{A2}.  
A W-graph with $m$ sinks  is a subgraph $g_m$ of $G$  with the same set of vertices $\mathcal{S}$ such that 
\begin{enumerate}[(i)]
\item  $n-m$ vertices have exactly one outgoing arc in $g_m$, while the other $m$ vertices, called sinks, have no outgoing arcs in $g_m$;
\item  $g_m$ contains no cycles.
\end{enumerate}
\end{definition}

Assumptions \ref{A1} and \ref{A2} guarantee the existence of a W-graph with any $1\le m\le n$ sinks.
If $m = n$, there is the unique W-graph, and it has no arcs.  
A W-graph with $m$ sinks is a forest consisting of $m$ { connected components. Each connected component is a directed in-tree, i.e., a 
rooted tree with arcs pointing towards its root (sink).} 
The collection of all possible W-graphs with $m$ sinks will be denoted by $\mathcal{G}(m)$.

Asymptotic estimates for eigenvalues are defined by a collection of W-graphs with minimal possible sums of weights of their arcs \cite{Wentzell}.
We call such graphs optimal W-graphs and denote by $g^{\ast}_m$. Precisely,
\begin{definition}{\bf(Optimal W-graph)}  
\label{defW}
$g^{\ast}_m$ is an optimal W-graph  with $m$ sinks if and only if
\begin{equation}
\label{optwg}
g^{\ast}_m = \arg \min_{ g \in \mathcal{G}(m) }  \mathcal{V} (g),~~\quad{\rm where}
~~\mathcal{V}(g):=\sum_{(i\rightarrow j)\in g}U_{ij}.
\end{equation} 
\end{definition} 
{ For brevity, we write $(i\rightarrow j)\in g$ meaning $(i\rightarrow j)\in\mathcal{A}_g$, where $\mathcal{A}_g$ is the set of arcs of $g$.}
We emphasize that, in the optimal W-graph, the sum of weights $\mathcal{V}(g)$ needs to be minimized simultaneously
with respect to the choices of sinks and arcs.


\subsection{Asymptotic estimates for eigenvalues}
 Wentzell established the following asymptotic estimates for eigenvalues in terms of the optimal W-graphs \cite{Wentzell}:
 \begin{theorem}{ (Wentzell, 1972)}\footnotemark[1] 
 \label{W_asymeigval}
  Let $ ~0, -\lambda_1+i\mu_1,~ -\lambda_2+i\mu_2,~ \cdots,~ -\lambda_{n-1}+i\mu_{n-1}$ be the eigenvalues of the generator matrix with 
  off-diagonal entries of the order of $\exp ( - U_{ij} / \varepsilon )$, ordered so that $0 < \lambda_1 \leq \lambda_2 \leq \cdots \leq \lambda_{n-1}$. 
  Then as $\varepsilon \rightarrow 0$,  
\begin{equation}
{\lambda_m} \asymp \exp(-\Delta_m / \varepsilon),\quad{\rm where}\quad
\Delta_m  := \mathcal{V} (g^{\ast}_m) - \mathcal{V} ( g^{\ast}_{m+1} ),\quad m = 1, 2, \cdots, n-1.
\end{equation}  
 \end{theorem}
 \footnotetext[1]{No proof of Theorem \ref{W_asymeigval} was provided in \cite{Wentzell}. We are not aware of any published proof of this result.}

Theorem \ref{W_asymeigval} implies \cite{Wentzell} that 
{
\begin{equation}
\label{ww}
\Delta_1 \geq \Delta_2 \geq \cdots \geq \Delta_{n-1}>0.
\end{equation}
}

{
Assumption \ref{A3} below allows us to derive a number of important facts.
\begin{assumption} 
\label{A3}
For all $1\le m\le n$, the optimal W-graphs $g^{\ast}_m$  are unique.
\end{assumption} 
It is, in essence, a \emph{genericness assumption}, because 
if the weights $U_{ij}$ are random real numbers in an interval $(0,U_{\max}]$, then it holds with probability one. 
}

{ Under Assumption \ref{A3}, the strict inequalities take place { in Eq. \eqref{ww}}. 
Hence, if $\varepsilon$ is sufficiently small, all eigenvalues are real and distinct \cite{Wentzell}. 
In this case, } for a generator matrix $L$ with off-diagonal entries of the form \eqref{Lij},  
the estimates given by Theorem \ref{W_asymeigval} can be made more precise.
 \begin{theorem}
  \label{GC_asymeigval} 
Suppose a Markov chain with pairwise rates of the form  $\kappa_{ij}\exp(-U_{ij}/\varepsilon)$
satisfies Assumptions \ref{A1}, \ref{A2}, and \ref{A3}. Let $0 < \lambda_1 \le \lambda_2\le\cdots \le \lambda_{n-1} $ 
be eigenvalues of the corresponding matrix $-L$. Then as $\varepsilon \rightarrow 0$, for $1\le m\le n-1$ we have
 \begin{align}
 \lambda_m &= \alpha_m \exp ( - \Delta_m / \varepsilon ) (1 + o(1) ),\quad {\rm where} \label{GC_lambda}\\
\Delta_m  &:= \mathcal{V} (g^{\ast}_m) - \mathcal{V} ( g^{\ast}_{m+1} ),\quad
 \alpha_m : = \frac{\prod_{(i\rightarrow j)\in g^{\ast}_m}\kappa_{ij}}{\prod_{(i\rightarrow j)\in g^{\ast}_{m+1}} \kappa_{ij}}.\notag
 \end{align}
 \end{theorem}
 Originally, we formulated Theorem \ref{GC_asymeigval} in \cite{CG} but did not provide a detailed proof in order to stay focused on practical matters.
Here, we provide a proof of Theorem \ref{GC_asymeigval} in Appendix \ref{App_A} as promised in \cite{CG}. 
The key point in this proof is the relationship between the coefficients of the characteristic polynomial of  $-L$ 
and the W-graphs with the corresponding numbers of sinks.  
%
%

\subsection{Weak nested property of optimal W-graphs} 
\label{sec:nested}
In \cite{CG}, under Assumption \ref{A3}, we proved a weak nested property of optimal W-graphs,
a pivotal property that  serves as a stepping stone in the design of the single sweep algorithm \cite{CG} and Algorithm 1, and 
facilitates a deeper understanding of the metastable behavior of continuous-time Markov chains.
The weak nested property can be compared to a stronger nested property of optimal W-graphs
taking place if the underlying Markov chain is time-reversible \cite{Cameron}.

{A connected component $S$ of a W-graph $g$ will be identified with its set of vertices.
$\mathcal{A}(g; S)$ will denote the set of arcs of  $g$ with tails in $S$.
}
\begin{theorem} 
\label{nested_thm}
{\bf (The weak nested property of optimal W-graphs \cite{CG})} 
Suppose that Assumptions \ref{A1}, \ref{A2}, and \ref{A3} hold.
Then the following statements are true.
\begin{enumerate} [(i)]
\item
There exists a unique connected component  $S_m$ of $g^{\ast}_{m+1}$, whose sink $s^{\ast}_m$ is not a sink of $g^{\ast}_m$.
\item
$\mathcal{A} (g^{\ast}_m; \mathcal{S} \setminus S_m ) = \mathcal{A} (g^{\ast}_{m+1}; \mathcal{S} \setminus S_m )$, 
i.e. the arcs in $g^{\ast}_{m}$ with tails in $\mathcal{S} \setminus  S_m$ coincide with those in $g^{\ast}_{m+1}$.
However, $\mathcal{A} ( g^{\ast}_m; S_m \setminus \{s^{\ast}_m \} )$ and $\mathcal{A} (g^{\ast}_{m+1}; S_m \setminus \{s^{\ast}_m\} )$ do not necessarily coincide.
\item
There exists a single arc $(p^{\ast}_m \rightarrow q^{\ast}_m) \in g^{\ast}_m$ with tail  $p^{\ast}_m$ in $S_m$ 
and head $q^{\ast}_m$ in another connected component  $Z_m$ with sink $z^{\ast}_m$ of $g^{\ast}_{m+1}$. 
\end{enumerate}
\end{theorem}
Theorem \ref{nested_thm} is illustrated in Fig. \ref{nested_fig}. 
The optimal W-graph $g^{\ast}_m$ is obtained from $g^{\ast}_{m+1}$ by $(i)$ adding an arc $p^{\ast}_m\rightarrow q^{\ast}_m$
with tail $p^{\ast}_m$ and head $q^{\ast}_m$ lying in different connected components of $g^{\ast}_{m+1}$, denoted by  $S_m$ and $Z_m$ respectively,  and $(ii)$
possibly rearranging some arcs with tails and heads in $S_m$.
We say that $S_m$ is \emph{absorbed} by $Z_m$ to form $g^{\ast}_m$. 
All arcs of $g^{\ast}_{m+1}$ with tails not in $S_m$
 are inherited by $g^{\ast}_m$.

\begin{figure}[htbp]
\begin{center}
\includegraphics[width=\textwidth]{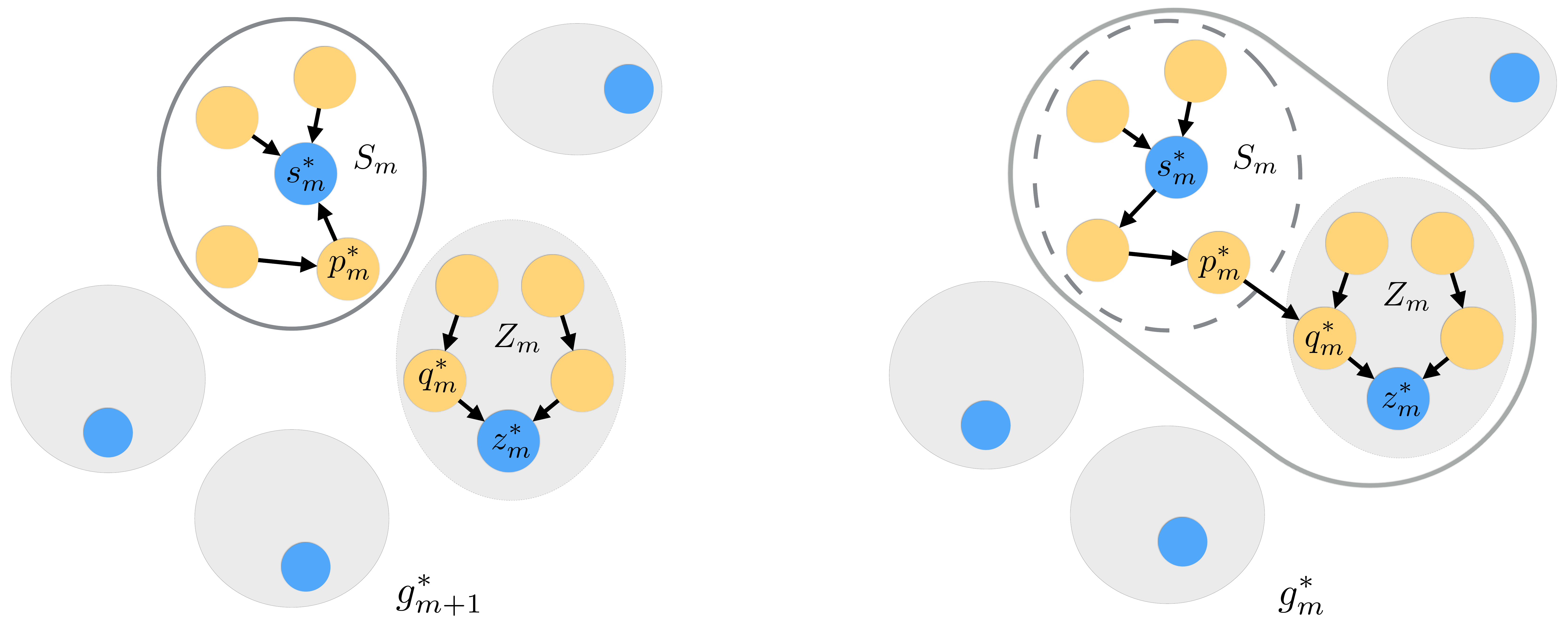}
\caption{An illustration  for the weak nested property of optimal W-graphs (Theorem \ref{nested_thm}).
The ovals symbolize connected components of the optimal W-graph $g^{\ast}_{m+1}$.
Blue  and yellow vertices represent sinks and non-sinks respectively. 
}
\label{nested_fig}
\end{center}
\end{figure}


{
\section{Freidlin's hierarchies, the critical timescales, and the T-graphs}
\label{sec:Tgraphs}
In this Section, we  introduce the \emph{Typical Transition Graphs} or \emph{T-graphs} for 
continuous-time Markov chains with pairwise rates of the form of Eq. \eqref{Lij} or \eqref{Lij1}. 
Originally, the T-graphs arose as a byproduct of the single-sweep algorithms \cite{CG}.
In this work, we connect the T-graphs 
to Freidlin's construction of the hierarchy of cycles/Markov chains 
\cite{freidlin-cycles,freidlin-physicad,FW,freidlin_symmetry}. 

A timescale is a function of $\epsilon$ defined up to the exponential order: we  say that $t(\varepsilon)$
is logarithmically equivalent to $e^{\theta/\varepsilon}$ and write 
$$
t(\varepsilon)\asymp e^{\theta/\varepsilon}~~~{\rm if}~~~
\lim_{\varepsilon\rightarrow 0}\varepsilon\log(t(\varepsilon)) = \theta.
$$
For brevity, we write 
$$
t > e^{\theta_1/\varepsilon}~~~\text{ implying that }~~~
\lim_{\varepsilon\rightarrow 0}\varepsilon\log(t(\varepsilon)) > \theta_1,
$$
and adopt analogous meanings for all other inequality signs for comparing timescales.

For every vertex $i\in\mathcal{S}$, let  $U_{\min}(i):=\min_{j\in\mathcal{S}}U_{ij}$.
The  outgoing arcs from $i$ of the minimal weight $U_{\min}(i)$ will be called \emph{min-arcs} from $i$,
and the set of all min-arcs from $i$ will be denoted by $\mathcal{A}_{\min}(i)$.
If a Markov jump process starts at state $i$, then the probability of the first jump along an arc $i\rightarrow j$
is given by
\begin{equation}
\label{t1}
\mathbb{P}_i(i\rightarrow j) = \frac{L_{ij}}{\sum_{l\neq i}L_{il}} =  \frac{\kappa_{ij}e^{-U_{ij}/\varepsilon}}{\sum_{l\neq i}\kappa_{il}e^{-U_{il}/\varepsilon}}
= \left[\frac{\kappa_{ij}}{\sum_{(i\rightarrow l)\in\mathcal{A}_{\min}(i)}\kappa_{il}}\right]
e^{-(U_{ij}-U_{\min}(i))/\varepsilon}(1 + o(1)).
\end{equation}
Eq. \eqref{t1} shows that $\mathbb{P}_i(i\rightarrow j)$ remains positive as $\varepsilon\rightarrow 0$ if an only if the arc $i\rightarrow j$ is a min-arc from $i$.
The expected exit time from $i$ is logarithmically equivalent to $e^{U_{\min}(i)/\varepsilon}$.

For simplicity of the presentation below, we adopt
\begin{assumption}
\label{A4}
The Markov chain associated with the graph $G(\mathcal{S},\mathcal{A},\mathcal{U})$ is irreducible.
\end{assumption}

Freidlin's construction of the hierarchy of cycles/Markov chains \cite{freidlin-cycles,freidlin-physicad,freidlin_symmetry}  
can be outlined as follows (with minor terminological modifications). 
We begin with the graph $G_0(\mathcal{S},\emptyset,\emptyset)$
with the set of states $\mathcal{S}$ and no arcs. All vertices are called the zeroth order Markov chains.
Then all min-arcs from all vertices are added to $G_0$ resulting in the graph 
$$
G_1(\mathcal{S},\mathcal{A}_1,\mathcal{U}_1),~~{\rm where}~~ 
\mathcal{A}_1:=\bigcup_{i\in\mathcal{S}}\mathcal{A}_{\min}(i),\quad  \mathcal{U}_1 = \{U_{ij}~|~(i\rightarrow j)\in\mathcal{A}_1\}.
$$ 
All nontrivial closed communicating classes in $G_1$
(i.e., consisting of more than one state), are called the first  order cycles/Markov chains.
The birth timescale of each first order cycle/Markov chain (the reciprocal of its rotations rate)
is the maximal expected holding time among its vertices. 
If there is only one first order cycle/Markov chain, and it contains the whole set of states $\mathcal{S}$, the construction is fulfilled.
Otherwise, each first order cycle/Markov chain is treated as a macro-state. For each macro-state, 
the exit rate and the set of exit arcs along which the process
escapes from it with probability that remains positive as $\varepsilon\rightarrow 0$ are found. 
These exit arcs are the min-arcs from the macro-states. Their weights are modified to set them equal to
the exponential factors of their exit rates.  
Contracting the macro-states into single super-vertices,
one obtains the graph $G_1^{(1)}(\mathcal{S}^{(1)},\mathcal{A}^{(1)}_1,\mathcal{U}^{(1)}_1)$ which is a directed forest.
The graph $G_2^{(1)}(\mathcal{S}^{(1)},\mathcal{A}^{(1)}_2,\mathcal{U}^{(1)}_2)$ is obtained from $G_1^{(1)}$ 
by adding all min-arcs from the macro-states with their modified weights. 
Then  one checks for nontrivial closed communicating classes in $G^{(1)}_2$. 
If there is a unique closed communicating class comprising the whole set of states $\mathcal{S}^{(1)}$, the process terminates.
Otherwise, it is continued in the same manner as it was done for the graph $G_1(\mathcal{S},\mathcal{A}_1,\mathcal{U}_1)$.
This recursive procedure will terminate in a finite number of steps $Q$ with a graph $G_Q(\mathcal{S}^{(Q-1)},\mathcal{A}^{(Q-1)}_Q,\mathcal{U}^{(Q-1)}_Q)$. 
It  produces  the hierarchy of Freidlin's cycles/Markov chains of orders $q=0, 1, 2, \ldots, Q$.
The one of order $Q$ comprises the whole set of states $\mathcal{S}^{(Q-1)}$. Recursively restoring all contracted 
closed communicating classes while keeping the  modified arc weights, one obtains the graph 
$G^{\ast}(\mathcal{S},\mathcal{A}^{\ast},\mathcal{U}^{\ast}):=G_Q(\mathcal{S},\mathcal{A}_Q,\mathcal{U}_Q)$.
We will call the transitions corresponding to arcs $(i\rightarrow j)$ of the graph $G^{\ast}(\mathcal{S},\mathcal{A}^{\ast},\mathcal{U}^{\ast})$ \emph{typical transitions} 
and associate them with the  corresponding associated timescales $e^{U_{ij}^{\ast}/\varepsilon}$.

We remark that the birth timescales and exit timescales of the cycles/Markov chains 
 are only partially ordered in the sense that if $C_b$ is a closed communicating class in some $G_r(\mathcal{S}^{(r-1)},\mathcal{A}_r^{(r-1)})$
 containing a super-vertex comprising a closed communicating class $C_a$ that the birth timescale of $C_b$ is greater than the one of $C_a$, and the holding time 
 in $C_b$ is greater than the one in $C_a$. However, if $C_x$ and $C_y$ are closed communicating classes in $G_q(\mathcal{S}^{(q-1)},\mathcal{A}_q^{(q-1)})$
 and $G_r(\mathcal{S}^{(r-1)},\mathcal{A}_r^{(r-1)})$ respectively, where $q < r$, it is possible that the birth and exit timescales of $C_x$ are greater than those of $C_y$.
 Furthermore, the order in which arcs are added to $G_0$ to form  eventually $G_Q$ is, in general, neither increasing nor decreasing order 
 of their modified or original weights.

In this work, we propose an alternative construction that builds the graph $G^{\ast}(\mathcal{S},\mathcal{A}^{\ast},\mathcal{U}^{\ast})$ 
from $G_0(\mathcal{S},\emptyset,\emptyset)$
by adding arcs to it \emph{ in the increasing order of their modified weights}. 
In our construction,  the procedure of finding the exit rates from the closed communicating classes is simpler than that in Freidlin's one 
due to the order in which the arcs are added.

\begin{definition}{\bf (Critical exponents and T-graphs)}
\label{def:Tgraph}
Let $G^{\ast}(\mathcal{S},\mathcal{A}^{\ast},\mathcal{U}^{\ast})$ be the graph obtained as a result of the construction of the hierarchy of Freidlin's cycles/Markov chains
as described above. The ordered subset of $\mathcal{U}^{\ast}$ defined by
\begin{equation}
\label{critical}
\gamma_1\le\ldots\le\gamma_K := \bigcup_{q=0}^{Q-1}\{ U_{\min}^{(q)}(i)~|~ i\in \mathcal{S}^{(q)} \}
\end{equation}
and is the set of critical exponents.  
The corresponding timescales $e^{\gamma_1/\varepsilon}$, ..., $e^{\gamma_K/\varepsilon}$ are the critical timescales.

Let $\theta_1<\ldots<\theta_P$ be the ordered set of distinct numbers in $\{\gamma_k\}_{k=1}^{K}$. 
The Typical Transition Graph or T-graph $T_p$ is the subgraph of 
$G^{\ast}(\mathcal{S},\mathcal{A}^{\ast},\mathcal{U}^{\ast})$ containing all arcs of 
weights up to $\theta_p$, i.e., $T_p=T_p(\mathcal{S},\mathcal{A}(T_p),\mathcal{U}(T_p))$ where 
\begin{equation}
\label{tg1}
\mathcal{A}(T_p) = \{i\rightarrow j\in \mathcal{A}^{\ast}~|~U^{\ast}_{ij}\le\theta_p\}, \quad
\mathcal{U}(T_p) = \{U^{\ast}_{ij}\in \mathcal{U}^{\ast}~|~U^{\ast}_{ij}\le\theta_p\},\quad p = 0,1,\ldots,P.
\end{equation}
The T-graph $T_p$ is associated with the range of timescales $t\in[e^{\gamma_p/\varepsilon},e^{\theta_{p+1}/\varepsilon})$ for $p = 1,\ldots,P-1$.
The T-graphs $T_0$ and $T_P$ are associated with the ranges of timescales $t\in[0,e^{\theta_1/\varepsilon})$ and $t\in[e^{\theta_P/\varepsilon},\infty)$.
\end{definition}

Algorithm 1 in Section \ref{sec:alg1} builds the hierarchy of the critical exponents and the T-graphs in the case where 
all critical exponents $\gamma_k$, $k=1,\ldots,K$, are
distinct and the min-arcs from each vertex and each closed communicating class are unique. In this case, $K=P$ and $\gamma_k=\theta_k$, $1\le k\le K$,
and all closed communicating classes in all graphs $G^{(q-1)}_q(\mathcal{S}^{(q-1)},\mathcal{A}^{(q-1)}_q,\mathcal{U}^{(q-1)}_q)$, $1\le q\le Q$,
are simple directed cycles. 
Each of these cycles has the unique main state, where the system is found with probability close to one for small enough $\varepsilon$ provided that it is in the cycle,
and the unique exit arc, via which the system escapes from the cycle with probability close to one for small enough $\varepsilon$. Furthermore, the exit rates from
the all cycles of all orders are distinct.
We refer to such a case as \emph{a case with no symmetry}.
In this case, one can subdivide the set of the critical exponents $\gamma_k$, $1\le k\le K$, into two subsets associated
with the set of nonzero eigenvalues and the births of cycles respectively. 
Besides, if the pairwise rates $L_{ij}$ are of the form $L_{ij}=\kappa_{ij}e^{-U_{ij}/\varepsilon}$,
one can obtain the sharp estimates for the eigenvalues from the output of Algorithm 1. 
Finally, the unique hierarchy of the optimal W-graphs can be easily extracted from the found T-graphs.

Algorithm 2 presented in Section \ref{sec:alg2} is designed to handle the case with symmetry, 
i.e., the one where at least two critical exponents $\gamma_k$
coincide or at least one vertex or a closed communicating class has a non-unique min-arc. 
It computes the set of numbers $\theta_p$, $1\le p\le P$, and the hierarchy of the T-graphs $T_p$. 

Algorithm 1 is design to handle the case with no symmetry. However, it might be impossible to determine the presence of symmetry
by the input data and output of Algorithm 1. This can happen if one of the closed communicating classes has more that one min-arc.
Hence the symmetry test should be made at each step of the while-cycle in the code.
Suppose the code implementing Algorithm 1 detects some symmetry while running, but continues running and terminates normally.
In this case, its output will contain a correct set of distinct values of the critical exponents, while the set of graphs will not be 
the correct set of the T-graphs. Nevertheless, one can still extract some important facts about the  true T-graphs from its output.
Section \ref{sec:sym1} contains a discussion and a theorem on this subject. 
}

\section{Algorithm 1 for the study of the metastable behavior}  
\label{sec:alg1}
{
In this Section, we design Algorithm 1
that (a) finds the set of critical exponents and (b) builds the hierarchy of T-graphs.
Throughout this Section, we adopt Assumption \ref{A1}-\ref{A4} and 
\begin{assumption}
\label{A5}
All min-arcs are unique at every stage of Algorithm 1 presented below, and all critical exponents are distinct.
\end{assumption}
The while-loop in Algorithm 1 is essentially the same as the one in the ``single-sweep algorithm"
introduced in \cite{CG} for finding asymptotic estimates for eigenvalues and the hierarchy of optimal W-graphs.  
However,  Algorithm 1 has a different stopping criterion and used for a broader, 
more ambitious purpose. More output data are extracted. 
Furthermore, the presentations of these algorithms are very different.
Here, the description of Algorithm 1 serves the purpose of understanding the metastable behavior of the Markov process.
Its recursive structure and {\tt Contract}
and {\tt Expand} operators reflect the processes taking place on various timescales. 
On the contrary, the single-sweep algorithm was presented in \cite{CG} as a recipe for programming,
where
the key procedures were manipulations
with various sets of outgoing arcs. It was not recursive and 
involved no {\tt Contract} and {\tt Expand} operators.
}


\subsection{The recursive structure of the algorithm} 
\label{sec:code1}
{
Assumption \ref{A5} implies that a Markov process starting at a vertex $i$ 
leaves $i$ via the unique min-arc from $i$ denoted by $\mu(i)$ with probability approaching one as $\varepsilon\rightarrow 0$. 
For every vertex $i\in\mathcal{S}$, we select the unique min-arc $\mu(i)$,
sort the set of min-arcs $\{\mu(i)~|~i\in\mathcal{S}\}$
according to their weights $U_{\mu(i)}\equiv U_{\min}(i)$ in the increasing order,
and call the sorted set the bucket $\mathcal{B}$.

We start with the bucket $\mathcal{B}$
containing $n$ arcs and a graph $T =T_0= T(0)$ containing all vertices of $G$ and no arcs.
At each step of the main cycle of Algorithm 1,  
the minimum weight arc in the bucket $\mathcal{B}$  is transferred to the graph $T$.

At step one, the arc $\mu_1$ of the minimum weight in the whole 
graph $G$ is transferred to $T$. 
Set $\gamma_1 = U_{\mu_1}$ and $T_1=T$.
The graph $T_1$  coincides with 
the optimal W-graph $g^{\ast}_{n-1}$ with $n-1$ sinks.
Theorem \ref{GC_asymeigval} implies that the  eigenvalue $\lambda_{n-1}$ is approximated by
$
\lambda_{n-1}\approx\kappa_{\mu_1}e^{-U_{\mu_1}/\varepsilon} = \kappa_{\mu_1}e^{-\gamma_{1}/\varepsilon}.
$
Hence, $\alpha_{n-1} = \kappa_{\mu_1}$ and $\Delta_{n-1} = \gamma_1=U_{\mu_1}$.

At step two, we transfer the next arc $\mu_2$ from the bucket $\mathcal{B}$ to the graph $T$.
Set  $U_{\mu_2} = :\gamma_2$ and $T_2 = T$.
If  $T_2$ contains no cycles, 
the optimal W-graph $g^{\ast}_{n-2}$ with $n-2$ sinks coincides with $T_2$,  
the eigenvalue $\lambda_{n-2}\approx\kappa_{\mu_2}\exp(-U_{\mu_2}/\varepsilon)$, and $\Delta_{n-2}=\gamma_2=U_{\mu_2}$,  $\alpha_{n-2}=\kappa_{\mu_2}$. 

We continue transferring arcs from $\mathcal{B}$ to $T$ in this manner 
until the addition of a new arc creates a cycle in the graph $T$.
This will happen sooner or later because, if all $n$ arcs in $\mathcal{B}$ are transferred to the graph $T$ with $n$ vertices, 
a cycle must appear in $T$ (see Lemma \ref{lemma:m2} in Section \ref{sec:sym1}).

Suppose a cycle $c$ is formed after the arc $\mu_{last}$ of weight $\gamma_{last}$ was added to the graph $T$.  
Abusing notations, we treat $c$ as a graph or its set of vertices depending on the context. 
We need to identify the exit arc  from $c$, i.e., the arc $i\rightarrow j$ with $i\in c$ and  $j\notin c$
such that the probability to exit  $c$ via $i\rightarrow j$ tends to one as $\varepsilon\rightarrow 0$. 
We discard all arcs with tails and heads both in  $c$ and 
modify the weights and pre-factors of all remaining outgoing arcs from $c$ according to the following rule:
\begin{equation}
\label{update_rules}
\boxed{
U_{ij}^{new} = U_{ij}  + \gamma_{last} - U_{\mu(i)} ,\quad \kappa_{ij}^{new}=\frac{\kappa_{ij}\kappa_{\mu_{last}}}{\kappa_{\mu(i)}}.
}
\end{equation}
Note that none of the arcs with modified weight is in the bucket $\mathcal{B}$ at this moment, 
and all min-arcs with tails in $c$ have been already 
removed from $\mathcal{B}$ and added to the graph $T$.
The update rule is of crucial importance for Algorithm 1. It is consistent with the one in  \cite{freidlin-physicad}.
A  simple explanation for it is provided in Section \ref{sec:update_rule} below. 
The number $U^{new}$ shows what would be the total weight of the graph obtained from the last found
optimal W-graph by replacing the arc $\mu(i)$ with $i\rightarrow j$ and adding the arc $\mu_{last}$ (see \cite{CG} for more details). 
Then we contract the cycle $c$ to a single vertex  (a super-vertex)  $v_{c}$.
Finally,  among the arcs with tails in $v_{c}$, we find an arc with minimal modified weight, denote it by $\mu(v_{c})$, and add it to 
the bucket $\mathcal{B}$. By Assumption \ref{A5}, such an arc is unique. 
The idea of such a contraction is borrowed from the Chu-Liu/Edmonds algorithm \cite{chu-liu,Edmonds} 
for finding the optimal branching when the root is given a priori.
}

We continue adding arcs and contacting cycles in this manner.
Note that the indices $k$ of the numbers $\gamma_k$ are equal to the number of steps or, equivalently,
to the arcs removed from the bucket $\mathcal{B}$ and added to the graph T-graph  $T^{(r)}$ where $r$ is the current recursion level, 
i.e., the number of cycles contracted into super-vertices.
{ All arc additions not leading to cycles are associated with eigenvalues.}
The indices of the numbers $\Delta_m$ and $\alpha_m$ at step $k$ are $m=n-k+r$.

The main cycle terminates as the bucket $\mathcal{B}$ becomes empty.
This stopping criterion allows us to obtain the whole collection of the critical timescales and the whole hierarchy of T-graphs.
Suppose the main cycle terminates at the recursion level $R>1$. Then Algorithm 1 returns to the previous recursion level $r-1$
and expands super-vertex $v_{c_{R}}$ back into the cycle $c_{R}$. Then, if $R-1>1$, Algorithm 1 returns to the
recursion level $R-2$ 
and expands super-vertex $v_{c_{R-1}}$ back into the cycle $c_{R-1}$.
And so on, until recursion level zero is reached.
After that, one can extract the optimal W-graphs out of the appropriate T-graphs..
 
Below, Algorithm 1 is given in the form of a pseudocode. The operators {\tt Contract} and {\tt Expand} 
are described in Section \ref{sec:contract_expand} below.

\begin{algorithm}\\
\label{algorithm1}
{\bf Initialization:} Set step counter $k=0$, cycle or recursion depth counter $r=0$, and eigenvalue index $m=n$. Prepare the bucket 
$\mathcal{B}$, i.e., for every state $i\in \mathcal{S}$, find the min-arc $\mu(i)$, and then sort the set $\{\mu(i),i\in\mathcal{S}\}$
according to the arc weights $U_{\mu(i)}$ in the increasing order:
$$
U_{\mu_1}<U_{\mu_2}<\ldots<U_{\mu_n}.
$$
The graph $G^{(0)}(\mathcal{S}^{(0)},\mathcal{A}^{(0)},\mathcal{U}^{(0)})$ is original graph $G(\mathcal{S},\mathcal{A},\mathcal{U})$.\\
Set the graph $T=T(\mathcal{S}^{(0)},\emptyset,\emptyset)$. Set $T_0=T$.\\

{\bf The main body of the Algorithm:} Call the function {\tt FindTgraphs} with arguments  $k = 0$, $r = 0$, $G^{(0)}(\mathcal{S}^{(0)},\mathcal{A}^{(0)},\mathcal{U}^{(0)})$,
$T(\mathcal{S}^{(0)},\emptyset,\emptyset)$, and $\mathcal{B}$.

{\bf Function} {\tt FindTgraphs}$\left(r,k,G^{(r)},T^{(r)},\mathcal{B}\right)$\\
{\tt while} $\{$ $\mathcal{B}$ is not empty {\tt and} $T^{(r)}$ has no cycles $\}$\\
\hspace*{6 mm}{\tt (1)} Increase the step counter: $k=k+1$;\\
\hspace*{6 mm}{\tt (2)} Transfer the minimum weight arc $\mu^{(r)}_{k}$ from the bucket $\mathcal{B}$ to the graph $T^{(r)}$;\\
\hspace*{6 mm}{\tt (3)} Set $T^{(r)}_k = T^{(r)}$;\\
\hspace*{6 mm}{\tt (4)} Set $\gamma_{k} = U_{\mu^{(r)}_k}$;\\
\hspace*{6 mm}{\tt (5)} Check whether $T^{(r)}$ has a cycle; if so, denote the cycle by $c_{r+1}$;\\
\hspace*{12 mm} {\tt if} $\{$ $T^{(r)}$ contains no cycle $\}$\\
\hspace*{12 mm} {\tt (6)} Decrease the eigenvalue index $m = m-1$ and set $\Delta_{m} = \gamma_k$; $\alpha_{m} = \kappa_{\mu_k}$; \\
\hspace*{12 mm} {\tt (7)} Set $k(m) = k$ and denote by $(s^{\ast})^{(r)}_m$ the sink of\\
\centerline{
 the connected component of $T^{(r)}_{k-1}$ containing the tail of $\mu^{(r)}_k$;\\ 
}
\hspace*{12 mm} {\tt end if}\\
{\tt end while}\\
{\tt if} $\{$ a cycle  was detected at step {\tt (5)} $\}$\\
\hspace*{6 mm}{\tt (8)} Save the index $k$ at which the cycle $c_{r+1}$ arose: set $k_{r+1} = k$;\\
\hspace*{6 mm}{\tt (9)} {Remove the arcs with both  tails and heads  in $c_{r+1}$ (if any)};\\
\hspace*{12 mm}{\tt if} $\{$ the set of  arcs of $G^{(r)}$ with tails in $c_{r+1}$ and heads not in $c_{r+1}$ is not empty $\}$\\
\hspace*{12 mm}{\tt (10)} Modify weights and pre-factors of all arcs\\
\centerline{
 with tails in $c_{r+1}$  and heads not in $c_{r+1}$ according to Eq. \eqref{update_rules};\\
 }
\hspace*{12 mm}{\tt (11)} Contract $c_{r+1}$ into a super-vertex $v_{c_{r+1}}$: \\
\centerline{
$G^{(r+1)}$ = {\tt Contract}$(G^{(r)},c_{r+1})$; $T^{(r+1)}$ = {\tt Contract}$(T^{(r )},c_{r+1})$;\\
}
\hspace*{12 mm}{\tt (12)} Denote  the min-arc from $v_{c_{r+1}}$ by $\mu(v_{c_{r+1}})$ and add it to $\mathcal{B}$;\\
\hspace*{12 mm}{\tt (13)} Call the function {\tt FindTgraphs}$\left(r+1,k,G^{(r+1)},T^{(r+1)},\mathcal{B}\right)$;\\
\hspace*{12 mm}{\tt (14)} Expand the super-vertex $v_{c_{r+1}}$ back into the cycle $c_{r+1}$: \\
 \centerline{
 {\tt for} $j \ge k_r$   $T^{(r)}_j$ = {\tt Expand}$(T^{(r+1)}_j,c_{r+1})$;   {\tt end for} 
}
\hspace*{12 mm}{\tt end if}\\
{\tt end if}\\
\noindent
{\tt end }
\end{algorithm}

\begin{figure}[htbp]
\begin{center}
\includegraphics[width = \textwidth]{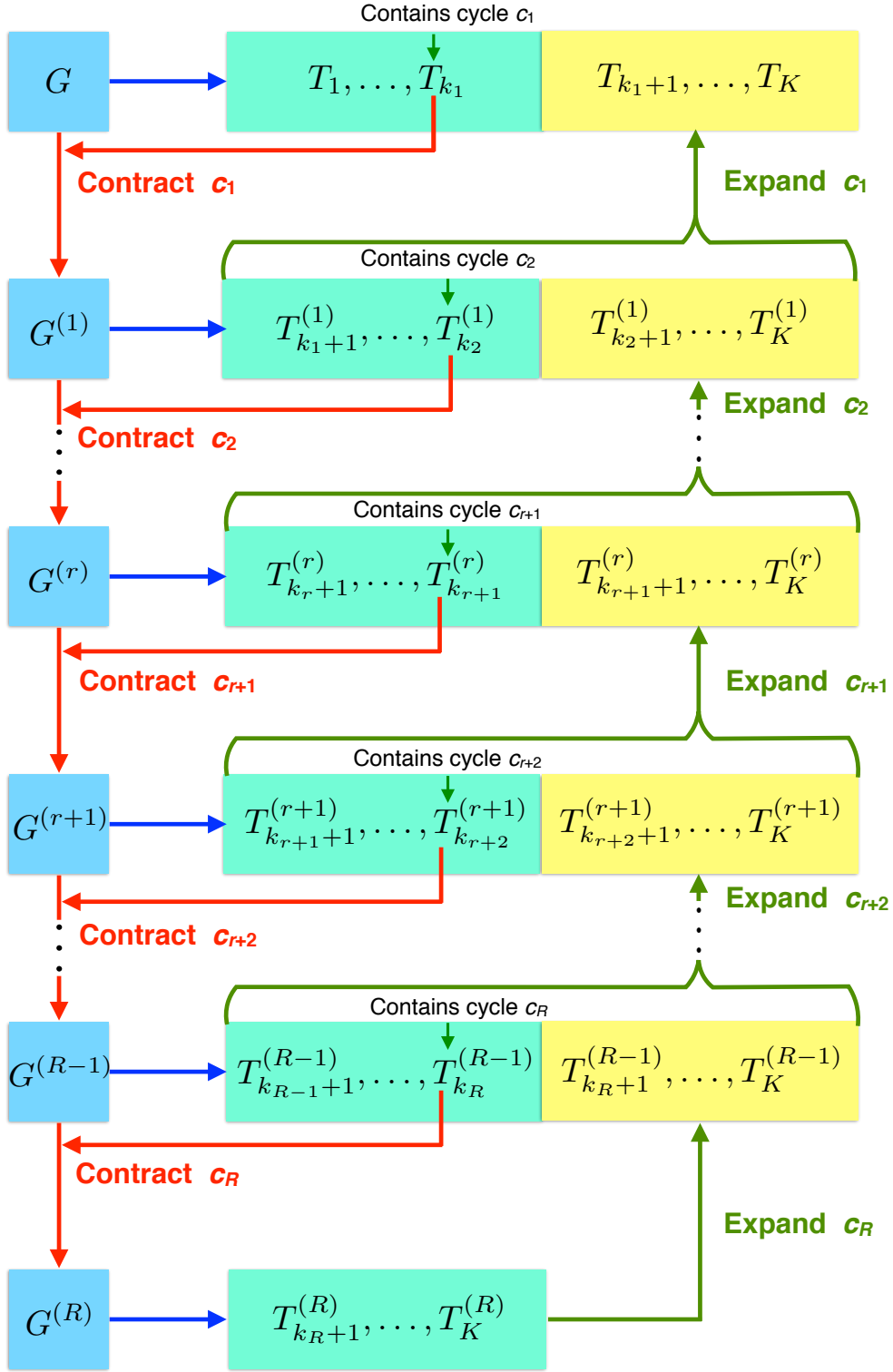}
\caption{The flow chart of Algorithm 1. The superscripts indicate the recursion levels. }
\label{fig:flowchart}
\end{center}
\end{figure}

The flow chart of Algorithm 1 is shown in Fig. \ref{fig:flowchart}.
At each recursion level, the T-graphs obtained in the while-cycle  are placed into the green boxes, while the ones
 obtained by the {\tt Expand} operator are placed into the yellow boxes. 
The number $K$ is the index of the last step of Algorithm 1.
At $k=K$, cycle $c_{R+1}$ is detected in $T^{(R)}$. { Hence steps {\tt (8)} and {\tt (9)} are performed.} 
However, since it is the last step, the set of arcs in $G^{(R)}$
with tails in $c_{R+1}$ and heads not in $c_{R+1}$ is empty. Therefore, steps {\tt (10) - (14)} are not executed. Instead,
the algorithm returns to the previous recursion level and lands at step {\tt (14)}  where the function {\tt Expand} is called.

The output of Algorithm 1 consists of several datasets.
\begin{itemize}
\item 
The set of critical exponents $\gamma_1<\ldots<\gamma_K$ defining the critical timescales 
$e^{\gamma_1/\varepsilon}<\ldots<e^{\gamma_K/\varepsilon}$.
\item 
The hierarchy of T-graphs $T_0\subset T_1 \subset \ldots\subset T_K$, indicating the most likely transitions to observe up to timescales $t$, where
$0\le t<e^{\gamma_1/\varepsilon}$, $e^{\gamma_1/\varepsilon}\le t< e^{\gamma_2/\varepsilon}$, ..., 
$e^{\gamma_K/\varepsilon}\le t<\infty$ respectively.
\item
The set of exponents $\Delta_{n-1}<\Delta_{n-1}<\ldots<\Delta_1$ and pre-factors $\alpha_{n-1},\alpha_{n-2},\ldots,\alpha_1$
determining sharp estimates of eigenvalues according to $\lambda_{m}=\alpha_me^{-\Delta_m/\varepsilon}$.
Note that  $m = n-k+r$.
\item
The set of sinks $s^{\ast}_m$ and $z^{\ast}_m$. 
The optimal W-graph $g^{\ast}_m$ has the sinks $\{s_j^{\ast}\}_{j=0}^{m-1}$ where the sink $s^{\ast}_0$ is 
defined as the sink $z^{\ast}_1$. The sinks $z^{\ast}_m$ are the sinks of those connected components of 
the optimal W-graphs $g^{\ast}_{m+1}$ that absorb the connected components with sinks $s^{\ast}_m$ as the graphs $g^{\ast}_m$
are created.
The sets of vertices in connected components
of the optimal W-graph $g^{\ast}_m$ coincides with those of the T-graph $T_{n-m+r}$.
An explanation how one can extract the optimal W-graphs $g^{\ast}_m$ from the graphs $T_{n-m+r}$ is given in Section \ref{sec:extractW}.
\item
The set of step numbers $k_1,\ldots,k_R$ at which the cycles $c_1,\ldots,c_R$ are created.
\end{itemize}

\begin{remark}
\normalfont
The number of arcs in the bucket $\mathcal{B}$ is reduced by one after each step when no cycle is created, and 
remains that same otherwise.
Let $k^{\ast}$ be the smallest step index such that the number of arcs in the bucket  $\mathcal{B}$ at the end of step $k^{\ast}$
reaches one. 
The eigenvalue index $m$ switches to $m=1$ during step $k^{\ast}$. 
 A new cycle will be obtained at every step $k$ such that $k^{\ast}<k\le K$. 
 For all $k\ge k^{\ast}$, the T-graphs $T_k$
are connected, while for all $k<k^{\ast}$, $T_k$ are not connected. 
\end{remark}
\begin{remark}
\normalfont
The total number of steps $K$ and the total number of cycles $N_c$ satisfy the relationship
\begin{equation}
\label{ncycles}
K - N_c = n -1.
\end{equation}
The maximal recursion level $R= N_c - 1$. Hence $K-R=n$.
\end{remark}

\begin{remark}
\normalfont
In the case of a time-reversible Markov chain, the total number of Freidlin's cycles of nonzero orders is $n-1$ \cite{cam1}.
Therefore, the total number of steps made by Algorithm 1 in this case  will be
$K = 2n - 2$.
In the general case, the total number of cycles $N_c$ satisfies
\begin{equation}
\label{ncycles2}
1\le N_c\le n-1.
\end{equation} 
Indeed, let us construct a directed rooted tree of cycles, whose leaves are the vertices, the root is the cycle $c_{N_c}$, the 
other nodes correspond to the cycles $c_r$, $1\le r\le R$, and whose arcs are directed toward the root.
Each node of this tree except for the root has exactly one outgoing arc. Hence, the total number of arcs is $n + N_c-1$.
 On the other hand, the nodes corresponding to the cycles $c_r$, $1\le r\le N_c$ have at least two incoming arcs.
 Therefore, the number of arcs $N_a$ must satisfy
 $$N_a = n+N_c-1 \ge 2N_c.$$ 
This inequality implies that $N_c\le n-1$. Besides,  Assumption \ref{A4}, implies that there must be at least one cycle.
This proves Eq. \eqref{ncycles2}. 
Hence, the maximal possible recursion depth  $R=n-2$ is achieved 
if and only if the Markov chain is time-reversible.
\end{remark}
\begin{remark}
\normalfont
The stopping criterion of Algorithm 1 can be modified depending on the goal.
If one would like to use Algorithm 1 { only} for finding estimates for eigenvalues and optimal W-graphs, it suffices to stop 
the while-cycle as soon as 
$|\mathcal{B}|=1$.
If one would like to find all T-graphs  for the timescales $0\le e^{\theta/\varepsilon}<\Delta$, one needs to stop the while-cycle as soon as
$\gamma_k \ge \Delta$.
\end{remark}
\begin{remark}
\label{rem:1}
\normalfont
For programming purposes,
if a cycle $c$ is encountered, 
it is more convenient to merge the sets of outgoing arcs from the vertices in $c$  instead of contracting $c$ into a single super-vertex
as it is done in the single-sweep algorithm in \cite{CG}.
\end{remark}

\subsection{The functions {\tt Contract} and {\tt Expand}}
\label{sec:contract_expand}
The function $G'$ = {\tt Contract}$(G,c)$ maps the graph $G$  onto the graph $G'$ 
as follows. All vertices of $G$  belonging to the cycle $c$  are mapped onto a single vertex $v_c$ of $G'$.
More formally, let $G=G(\mathcal{S},\mathcal{A},\mathcal{U})$ and $G'=G'(\mathcal{S'},\mathcal{A'},\mathcal{U'})$.
Let $\mathcal{S}_c$ be the subset of vertices lying in the cycle $c$, and $\mathcal{A}_c$ be the subset of arcs of $G$ defined by
$$
\mathcal{A}_c:=\{(i\rightarrow j)~|~i,j\in \mathcal{S}_c\}.
$$
Then
$\mathcal{S}' = (\mathcal{S}\backslash \mathcal{S}_c)\cup v_c$, $\mathcal{A}' = \mathcal{A}\backslash \mathcal{A}_c$,
$U'_{ij} = U_{ij}$ if $(i\rightarrow j)\in \mathcal{A}$ and $i\notin \mathcal{S}_c$, and  $U'_{v_cj} = \mathcal{U}_{ij}$,
if $(i\rightarrow j)\in \mathcal{A}$ and $i\in S_c$.

The function $G$ = {\tt Expand}$(G',c)$ is the inverse function of $G'$ = {\tt Contract}$(G,c)$.
If $G'=G'(\mathcal{S'},\mathcal{A'},\mathcal{U'})$, $G=G(\mathcal{S},\mathcal{A},\mathcal{U})$, 
then
$\mathcal{S} = (\mathcal{S'}\backslash \{v_c\})\cup \mathcal{S}_c$, $\mathcal{A} = \mathcal{A}\cup \mathcal{A}_c$,
$U_{ij} = U'_{ij}$ if $(i\rightarrow j)\in \mathcal{A'}$ and $i\neq v_c$, and  $U_{ij} ={U'}_{v_cj}$,
if $(i\rightarrow j)\in \mathcal{A}'$ and $i\in \mathcal{S}_c$. 
%

\subsection{An explanation of the update rules for arc weights and pre-factors}
\label{sec:update_rule}
In this Section, we explain the update rules for arc weights and pre-factors given by Eq. \eqref{update_rules} (also see \cite{freidlin-cycles,freidlin-physicad, FW}). 
A cycle 
$$
c = \{i_1\rightarrow i_2\rightarrow\ldots\rightarrow i_q\rightarrow i_1\}
$$ 
appears in Algorithm 1 if and only if the min-arcs from the vertices $i_l$, $l=1,\ldots,q-1$ are $\mu(i_l) = i_{l}\rightarrow i_{l+1}$, and 
the min-arc from $i_q$ is $\mu(i_q)=i_q\rightarrow i_1$. Let us restrict the  dynamics of the Markov chain to the cycle $c$ and 
  for each state $i_l\in c$ neglect the
transition rates of smaller exponential orders than $e^{-U_{\mu(i_l)}/\varepsilon}$ which is $\max_{(i_l\rightarrow j)\in\mathcal{A}}e^{-U_{i_lj}/\varepsilon}$
for sufficiently small $\varepsilon$. 
Then we obtain the following
generator matrix $L^c$ approximately describing the dynamics within the cycle $c$:
\begin{equation}
\label{Lapprox}
L^c=\left[\begin{array}{cccc}
-L^c_{\mu(i_1)}&L^c_{\mu(i_1)}&&\\
&-L^c_{\mu(i_2)}&L^c_{\mu(i_2)}&\\
&\ddots&\ddots&\\
&&-L_{\mu(i_{q-1})}^c&L_{\mu(i_{q-1})}^c\\
L^c_{\mu(i_q)}&&&-L^c_{\mu(i_q)}\end{array}\right],\quad {\rm where} ~~L^c_{\mu(i_l)} = \kappa_{\mu(i_l)}e^{-U_{\mu(i_l)}/\varepsilon}.
\end{equation}
Solving $\pi_c^TL^c=0$, $\sum \pi_c(i) = 1$, we find the approximation to the invariant distribution in $c$:
\begin{equation}
\label{pic}
\pi_c(i_l) \approx \frac{\frac{1}{\kappa_{\mu(i_l)}}e^{U_{\mu(i_l)}/\varepsilon}}{\sum_{j=1}^q\frac{1}{\kappa_{\mu(i_j)}}e^{U_{\mu(i_j)}/\varepsilon}},\quad l=1,\ldots,q.
\end{equation}
Without the loss of generality we assume that the last added arc to the cycle $c$ is $\mu(i_q)$. 
Then $U_{\mu(i_q)} > U_{\mu(i_l)}$ for $l=1,2,\ldots,q-1$.
Multiplying the enumerator and the denominator in Eq. \eqref{pic} by $\kappa_{\mu(i_q)}e^{-U_{\mu(i_q)}/\varepsilon}$ and neglecting small summands
in the denominator we obtain
\begin{equation}
\label{pic1}
\pi_c \approx\left[ \frac{\kappa_{\mu(i_q)}}{\kappa_{\mu(i_1)}}e^{-(U_{\mu(i_q)}-U_{\mu(i_1)})/\varepsilon},\ldots,
\frac{\kappa_{\mu(i_q)}}{\kappa_{\mu(i_{q-1})}}e^{-(U_{\mu(i_q)}-U_{\mu(i_{q-1})})/\varepsilon},1\right].
\end{equation}
The quasi-invariant distribution $\pi_c$ {(i.e., the approximation to the invariant distribution for dynamics restricted to the subset of states $\{i_1,\ldots,i_q\}$)}
allows us to obtain sharp estimates for the exit rates from the cycle $c$ via arcs with tails in $c$.
For any arc $i_l\rightarrow j$ where $i_l\in c$ and $j\notin c$, the exit rate  from $c$ through the arc $i_l\rightarrow j$ is approximated by
\begin{equation}
\label{ur}
\pi_c(i_l)L_{i_lj} = \frac{\kappa_{\mu(i_q)}\kappa_{i_lj}}{\kappa_{\mu(i_l)}}e^{-(U_{i_lj} + U_{\mu(i_q)}-U_{\mu(i_l)})/\varepsilon}.
\end{equation}
If one treats the cycle $c$ as a macro-state, i.e., contracts it into a single vertex, then the effective exit rate from it is given by Eq. \eqref{ur}.
Recalling that  $\mu(i_q)$ is the last added arc, 
one readily reads off the update rule for the arc weights and the pre-factors given by Eq. \eqref{update_rules}.

\subsection{Extraction of optimal W-graphs}
\label{sec:extractW}
{ 
Let us imagine that we have abolished the {\tt Contract} and {\tt Expand} operators in Algorithm 1 and manipulate the arc sets instead as it is done in
the single-sweep algorithm \cite{CG}, i.e., as it is suggested in Remark \ref{rem:1} in Section \ref{sec:code1}. I.e., instead of contracting a cycle $c$ into a super-vertex, we 
$(i)$ discard all arcs with both tails and heads in $c$ that are not in the current graph $T$,
$(ii)$ modify the weights and the pre-factors of all outgoing arcs with tails in $c$ and heads not in $c$ according to Eq. \eqref{update_rules} 
and denote the set of such arcs by $\mathcal{B}_c$, and
$(iii)$ find the arc of minimal weight in $\mathcal{B}_c$ and add it to the bucket $\mathcal{B}$; this arc becomes the min-arc for all vertices in $c$.
The weight of any arc $(i\rightarrow j)$ modified  according to the update rule Eq. \eqref{update_rules}  
is equal to the increment of the total weight of the graph
obtained from the current optimal W-graph by replacing the arc $\mu(i)$ with $(i\rightarrow j)$ and adding the arc that led to the current cycle.
This fact and the weak nested property 
of the optimal W-graphs (Theorem \ref{nested_thm}) guarantee that the whole hierarchy of the optimal W-graphs $g^{\ast}_{n-1}$, ..., $g^{\ast}_1$ 
can be extracted from the T-graphs $T_{k(n-1)}$, ..., $T_{k(1)}$ built by 
Algorithm 1. 
}

Recall that $k(m)$ is the step number at which  the eigenvalue counter switches to $m$:
 $k(m) = n-m+r$, where $r$ is the recursion depth at step $k(m)$ in Algorithm 1. 
For convenience, we denote the unique absorbing state of the T-graph $T_{k(1)}$ by $s^{\ast}_0$, i.e., $z^{\ast}_1\equiv s^{\ast}_0$.
The optimal W-graphs $g^{\ast}_m$ can be extracted from the corresponding T-graphs $T_{k(m)}$ for $1\le m\le n-1$. 
We emphasize that $T_{k(m)}$ is fully expanded graph, i.e., its set of vertices is $\mathcal{S}$.
The set of sinks of $g^{\ast}_m$ is $\{s^{\ast}_j\}_{j=0}^{m-1}$.
In order to obtain the set of arcs of $g^{\ast}_m$, take the set of arcs of $T_{k(m)}$, then,
starting from every sink of $g^{\ast}_m$, trace the incoming arcs backwards and 
make sure that every vertex is visited at most once in the process. 
This procedure can be programmed using the recursive function {\tt AddArc2Wgraph} as follows. Mark all vertices in $\mathcal{S}$ as {\tt NotVisited}.
Set up a graph $g$ with the set of vertices $\mathcal{S}$ and no arcs. Then \\
{\tt for} $j = 0 : m - 1$\\
\hspace*{6 mm} Call {\tt AddArc2Wgraph}$(s^{\ast}_j)$;\\
{\tt end for}\\
{\bf Function} {\tt AddArc2Wgraph}$(s)$ \\
Mark the vertex $s$ as {\tt Visited}; \\
Let $\mathcal{A}_s$ be the set of arcs with heads at $s$;\\
{\tt while} $\{$ $\mathcal{A}_s$ is not empty $\}$\\
\hspace*{6mm} Remove an arc $i\rightarrow s$ from $\mathcal{A}_s$;\\
\hspace*{6mm} {\tt if } $\{$ $i$ is {\tt NotVisited}  $\}$\\
\hspace*{12 mm} Add $i\rightarrow s$ to $g$;\\
\hspace*{12mm} Call {\tt AddArc2Wgraph}$(i)$;\\
\hspace*{6mm} {\tt end if } \\
{\tt end while}\\
{\tt end}\\
The resulting graph $g$ is the desired optimal W-graph $g^{\ast}_m$.

\subsection{An illustrative example}
\label{sec:Texample}

In this Section, we demonstrate how Algorithm 1 works on the Markov chain corresponding to the graph $G$ shown in Fig. \ref{fig:Tex1} (Left).
\begin{figure}[htbp]
\begin{center}
\includegraphics[width=0.75\textwidth]{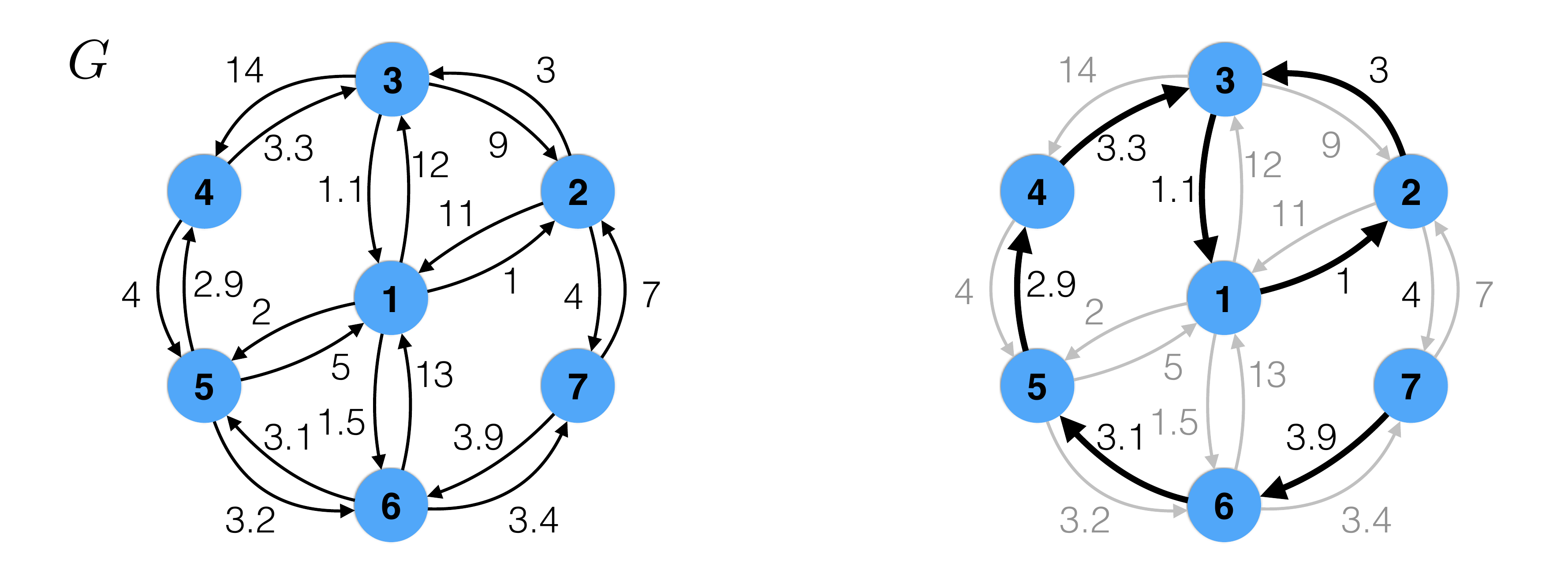}
\caption{An example in Section \ref{sec:Texample}. Left: A graph $G$ representing a continuous-time Markov chain. 
Right: Min-arcs from each vertex are  highlighted.}
\label{fig:Tex1}
\end{center}
\end{figure}
During the initialization, the min-arcs from each vertex are found and moved to the bucket $\mathcal{B}$ (see Fig. \ref{fig:Tex1} (Right)).
The initial graph $T$ is $T(\{1,2,3,4,5,6,7\},\emptyset,\emptyset)$.
The function {\tt FindTgraphs}$(r=0,k=0,G,T,\mathcal{B})$ finds the T-graphs $T_k$ and 
the numbers $\gamma_k$, $k=1,2,3,4$, as shown in Fig. \ref{fig:Tex2}, and $k(6) = 1$, $k(5)=2$, $k(4)=3$. 
The optimal W-graphs $g_6^{\ast}\equiv T_1$, $g_5^{\ast}\equiv T_2$, $g_4^{\ast}\equiv T_3$ and the numbers
$\Delta_6\equiv \gamma_1$, $\Delta_5\equiv \gamma_2$, and  $\Delta_4\equiv \gamma_3$ are found immediately. 
The pre-factors $\alpha_6$, $\alpha_5$, and $\alpha_4$
also can be found immediately if the pre-factors $\kappa$ are available. 
\begin{figure}[htbp]
\begin{center}
\includegraphics[width=0.75\textwidth]{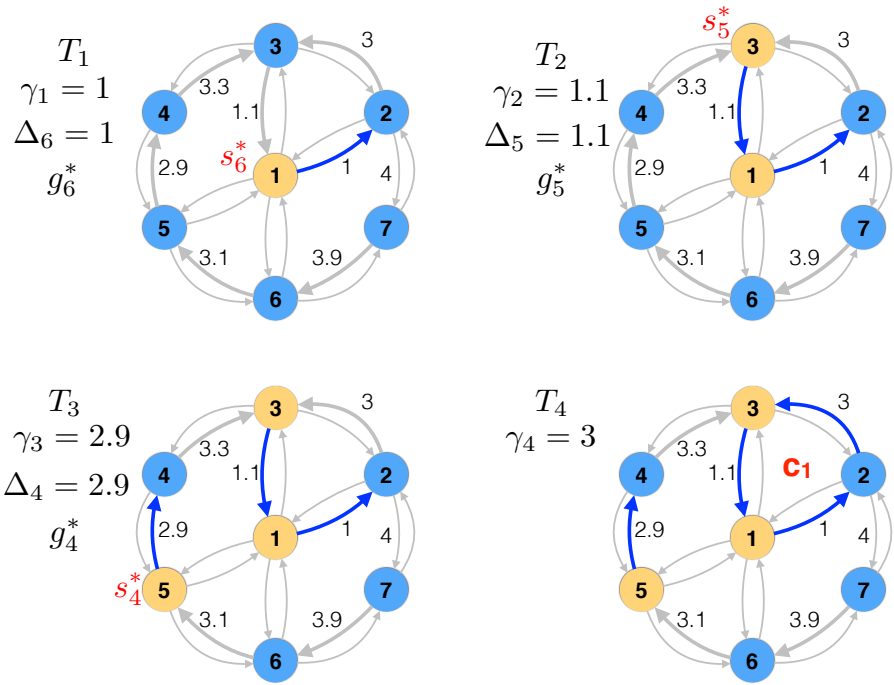}
\caption{An example in Section \ref{sec:Texample}. The T-graphs $T_1$, $T_2$, $T_3$ and $T_4$.
The arcs in $\mathcal{B}$ are shown with thick grey curves. The arcs of $T$ are shown with thick blue curves.
{Sinks of the optimal W-graphs are blue, the other vertices are yellow. }
}
\label{fig:Tex2}
\end{center}
\end{figure}
 The graph $T_4$ contains the cycle $c_1\equiv\{1\rightarrow2\rightarrow3\rightarrow1\}$.
The appropriate arc weights and pre-factors (if available) 
are modified according to Eq. \eqref{update_rules}.
For the arc weights we have
\begin{align*}
 &U_{34}:\quad 14 - 1.1 + 3 = 15.9;\\
 &U_{15}:\quad 2 - 1 + 3 = 4;\\
 &U_{16}:\quad 1.5 - 1 + 3 = 3.5.
 \end{align*}
 Then the cycle $c_1$ is contracted into a single super-vertex { $\{1,2,3\}$. Its min-arc $\{1,2,3\}\rightarrow 6$ of weight 3.5}
 is added to the bucket $\mathcal{B}$,
 and the function {\tt FindTgraphs}$(r=1,k=4,G^{(1)},T^{(1)},\mathcal{B})$ is called.
 The graph $G^{(1)}$ is shown in Fig. \ref{fig:Tex3} (Top Left).
{\tt FindTgraphs}$(r=1,k=4,G^{(1)},T^{(1)},\mathcal{B})$ finds the T-graphs $T^{(1)}_5$, $T^{(1)}_6$, and $T_7^{(1)}$
and the numbers $\gamma_5 = \Delta_3$, $\gamma_6=\Delta_2$, and $\gamma_7$, and $k(3) = 5$, $k(2) = 6$ (Fig. \ref{fig:Tex3}). The graph 
$T^{(1)}_7$ contains a cycle $c_2=\{6\rightarrow5\rightarrow4\rightarrow\{1,2,3\}\rightarrow6\}$. 
\begin{figure}[htbp]
\begin{center}
\includegraphics[width=0.75\textwidth]{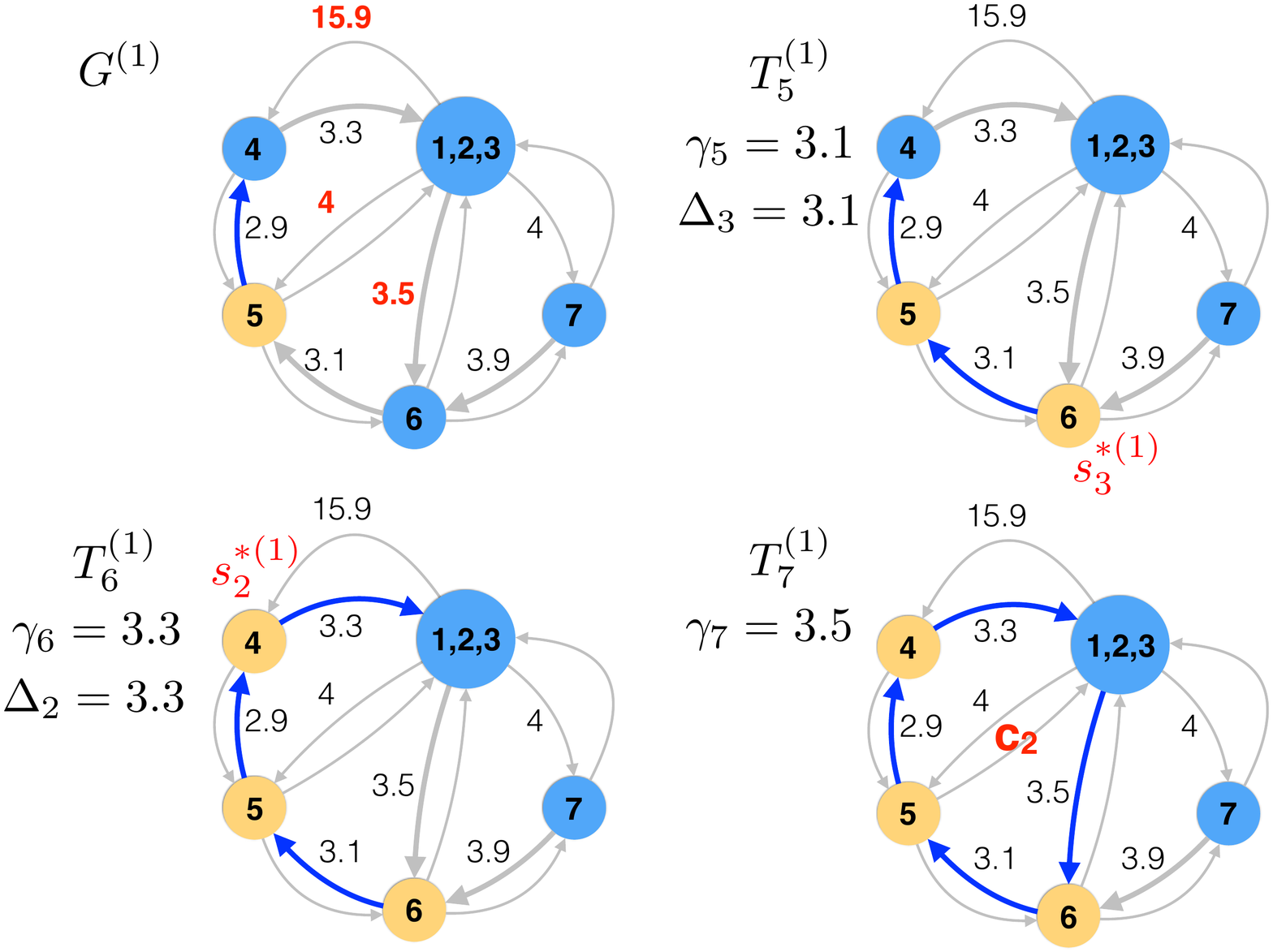}
\caption{An example in Section \ref{sec:Texample}.
The modified arc weights are shown in bold {red} in the graph $G^{(1)}$.  The T-graphs $T^{(1)}_5$, $T^{(1)}_6$, and $T_7^{(1)}$.
The arcs in $\mathcal{B}$ are shown with thick grey curves. The arcs of $T^{(1)}$ are shown with thick blue curves.
{Sinks of the optimal W-graphs or super-vertices containing their sinks are blue, the other vertices are yellow. }
}
\label{fig:Tex3}
\end{center}
\end{figure}
The appropriate arc weights and pre-factors (if available) 
are modified according to Eq. \eqref{update_rules}:
$$
 U_{67}:\quad 3.4 - 3.1 + 3.5 = 3.8.
$$
 Then the cycle $c_2$ is contracted into a single vertex { $\{1,2,3,4,5,6\}$. Its min-arc $\{1,2,3,4,5,6\}\rightarrow 7$}  of 
 weight 3.8 is added to the bucket $\mathcal{B}$,
 and the function {\tt FindTgraphs}$(r=2,k=7,G^{(2)},T^{(2)},\mathcal{B})$ is called.
 The graph $G^{(2)}$ is shown in Fig. \ref{fig:Tex4} (Left).
\begin{figure}[htbp]
\begin{center}
\includegraphics[width=0.75\textwidth]{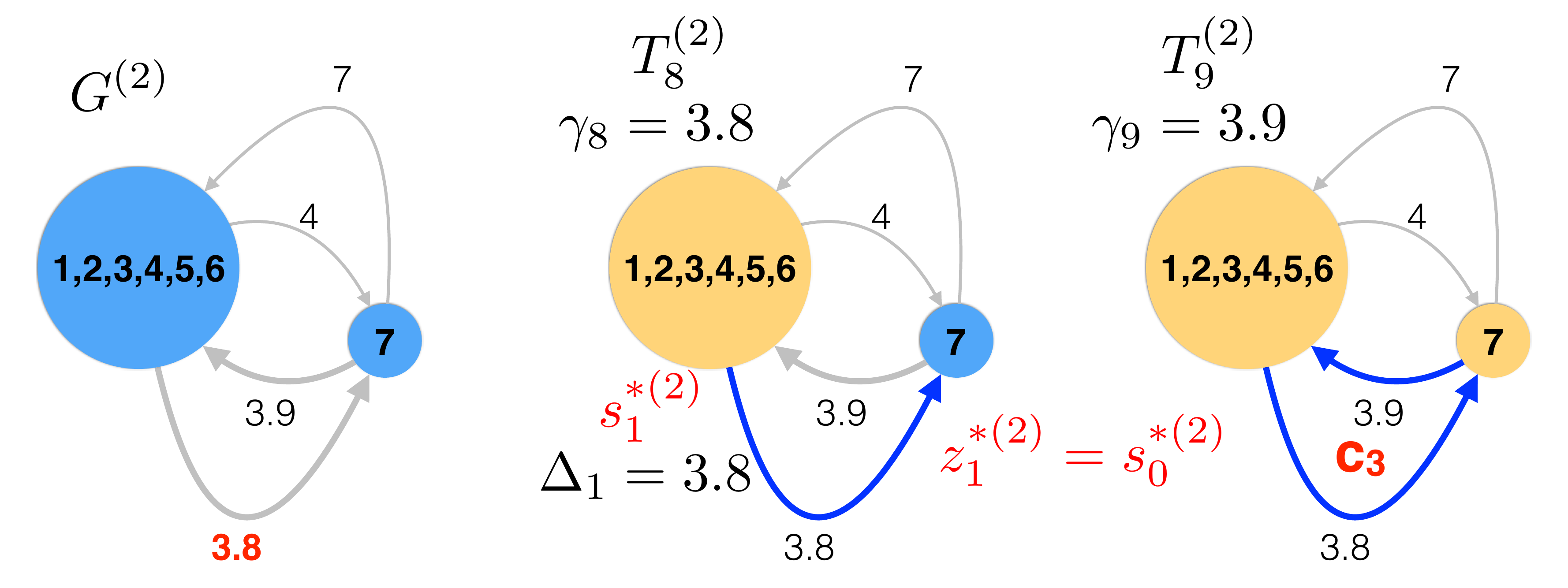}
\caption{An example in Section \ref{sec:Texample}.  
The modified arc weights are shown in bold read in the graph $G^{(2)}$.
The T-graphs $T^{(2)}_8$ and $T^{(2)}_9$.
The arcs in $\mathcal{B}$ are shown with thick grey curves. The arcs of $T^{(1)}$ are shown with thick blue curves.
}
\label{fig:Tex4}
\end{center}
\end{figure}
 {\tt FindTgraphs}$(r=2,k=7,G^{(2)},T^{(2)},\mathcal{B})$ finds $T_8^{(2)}$ and $T_9^{(2)}$,
  $\gamma_8 = \Delta_1$ and $\gamma_9$, and $k(1) = 8$ (see Fig. \ref{fig:Tex4}). 
 The graph $T^{(9)}$
 contains the cycle 
 $c_3=\{\{1,2,3,4,5,6\}\rightarrow 7\rightarrow \{1,2,3,4,5,6\}\}$.
 After the cycle $c_3$ is created,  the set of arcs in $G^{(2)}$ with tails in $c_3$ and heads not in $c_3$ is empty.
 Hence the condition of the {\tt if}-statement following Step {\tt (9)} in Algorithm 1 is {\tt false}. Hence Steps {\tt (10) - (14)}
 are not executed. In particular, the cycle $c_3$ is not contracted, and the function {\tt FindTgraphs} is not called.
Then  {\tt FindTgraphs}$(r=2,k=7,G^{(2)},T^{(2)}_7,\mathcal{B})$ is completed, and 
 the control returns to Step {\tt (14)} of {\tt FindTgraphs}$(r=1,k=4,G^{(1)},T^{(1)},\mathcal{B})$.
After the graphs $T^{(1)}_8$ and $T^{(1)}$ are obtained by expanding the cycle $c_2$,  the control returns to Step {\tt (14)}
of {\tt FindTgraphs}$(r=1,k=4,G^{(1)},T^{(1)},\mathcal{B})$. 
Then the T-graphs $T^{(1)}_k$ are expanded to 
$T_k$ for $k = 5,6,7,8,9$ respectively (Fig. \ref{fig:Tex5}).
\begin{figure}[htbp]
\begin{center}
\includegraphics[width=0.75\textwidth]{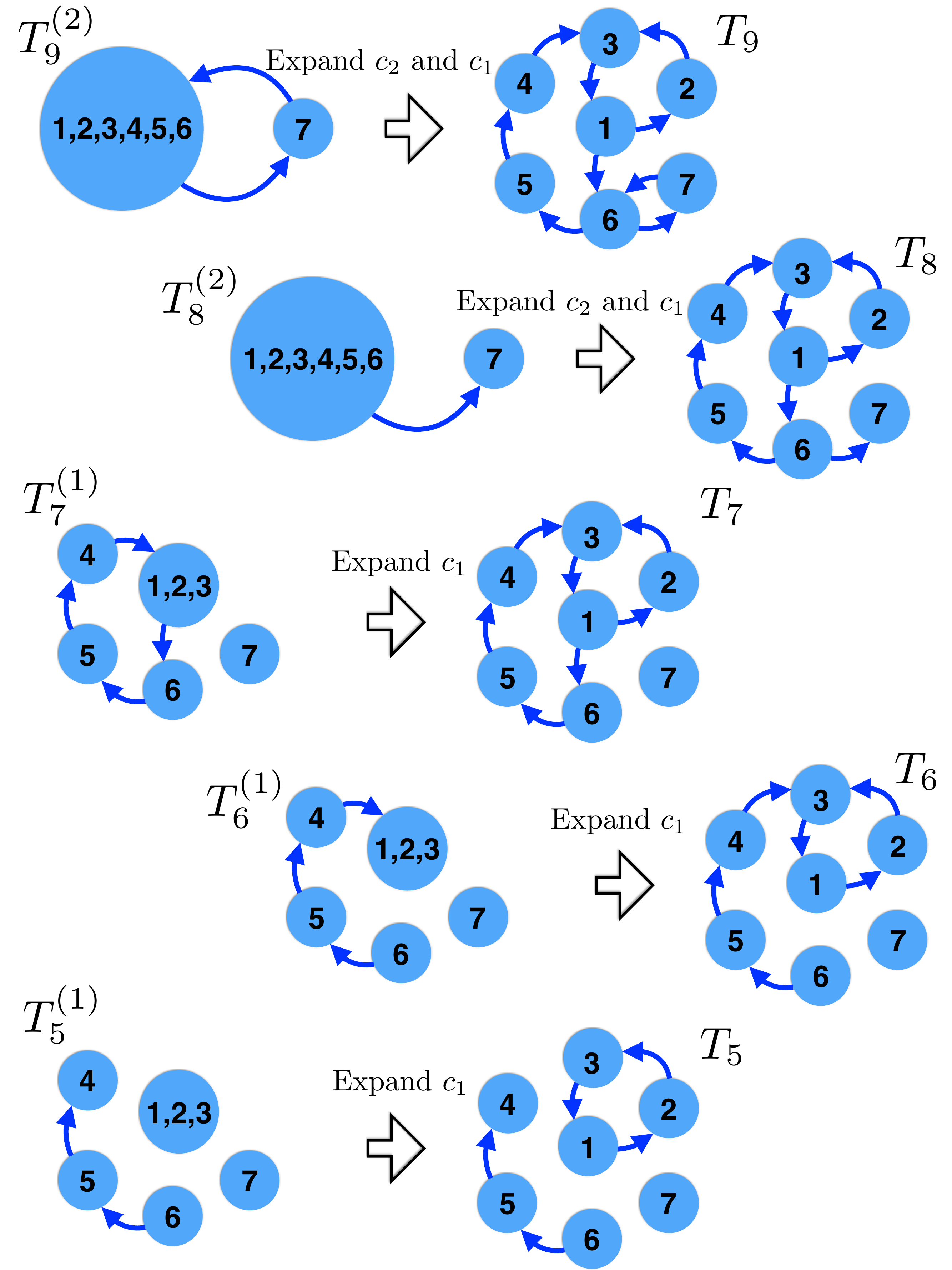}
\caption{An example in Section \ref{sec:Texample}.  The T-graphs $T_9$ and $T_8$
are obtained by expanding the cycles $C_2$ and $c_1$. The graphs $T_7$, $T_6$, and $T_5$ are 
obtained by expanding the cycle $c_1$.
Sinks of the optimal W-graphs or super-vertices containing their sinks are blue, the other vertices are yellow. 
}
\label{fig:Tex5}
\end{center}
\end{figure}
Finally, one can extract the optimal W-graphs $g^{\ast}_m$, $m=1,2,3$, from the T-graphs $T_{k(m)}$,
following the recipe proposed in Section \ref{sec:extractW} (Fig. \ref{fig:Tex6}).
\begin{figure}[htbp]
\begin{center}
\includegraphics[width=0.75\textwidth]{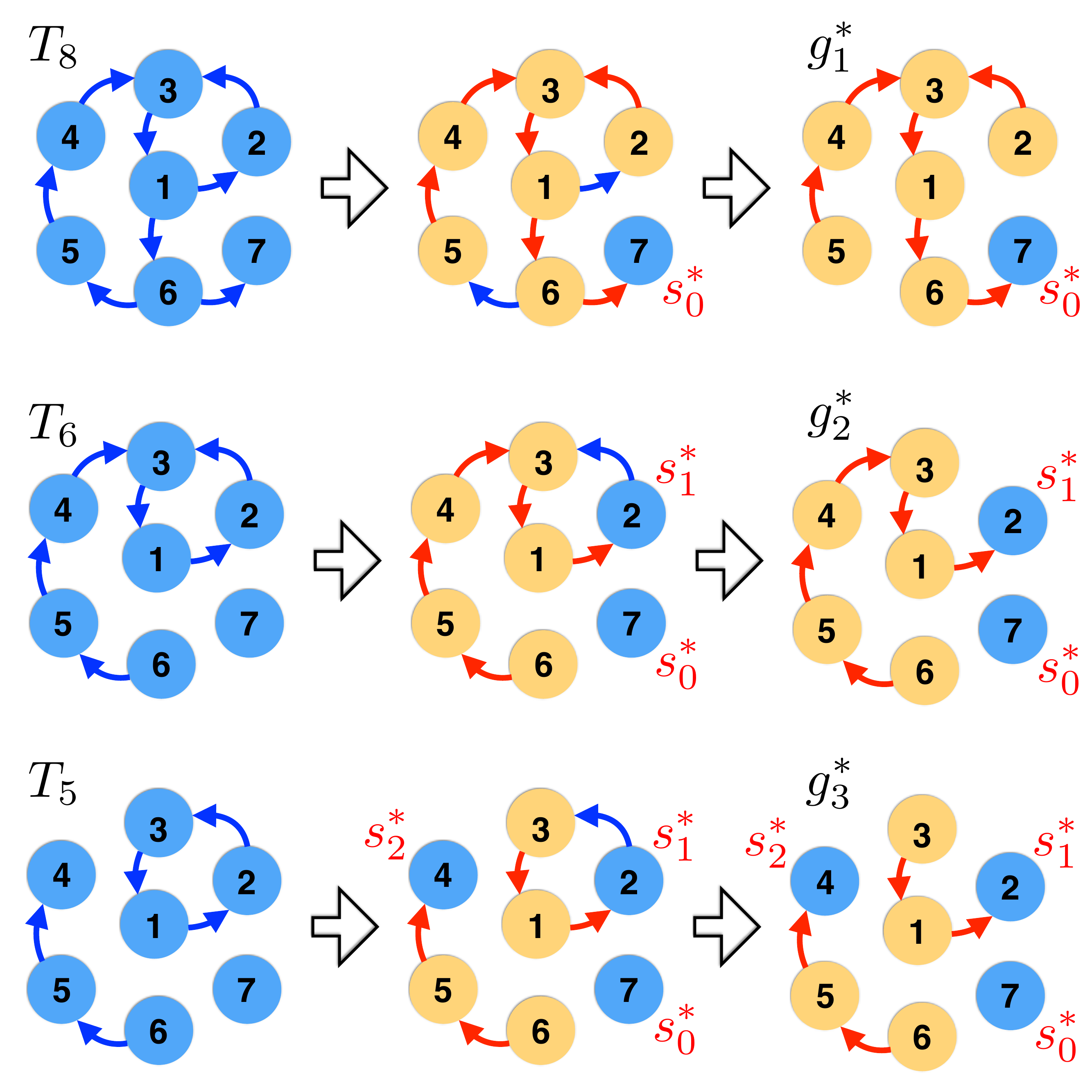}
\caption{An example in Section \ref{sec:Texample}.  Extracting the optimal W-graphs $g_1^{\ast}$, $g^{\ast}_2$ , and $g^{\ast}_3$
from the T-graphs $T_8$, $T_6$, and $T_5$ respectively.
}
\label{fig:Tex6}
\end{center}
\end{figure}
 
 
%
%
%
%

\section{The case with symmetry}
\label{sec:sym}
In this Section, we introduce Algorithm 2 for the study of metastability 
in continuous-time Markov chains with pairwise transition rates of the form $L_{ij}\asymp e^{-U_{ij}/\varepsilon}$
adopting only Assumptions \ref{A1} and \ref{A2} and abandoning Assumptions \ref{A3}, \ref{A4}, and \ref{A5}.

\subsection{Significance of Assumptions \ref{A1} - \ref{A5}}
\label{sec:sign}
We are going to keep Assumption \ref{A1} saying that the number of vertices in $G(\mathcal{S},\mathcal{A},\mathcal{U})$ is finite,
as it guarantees that Algorithm 1  terminates after a finite number of steps. Assumption \ref{A2} saying that $G(\mathcal{S},\mathcal{A},\mathcal{U})$
has a unique closed communicating class, guarantees the uniqueness of the invariant distribution perhaps supported on a subset of $\mathcal{S}$ if 
the corresponding Markov chain is reducible.
If it does not hold, it is natural to consider each closed communicating class of
$G(\mathcal{S},\mathcal{A},\mathcal{U})$ separately. So, we keep it. Note that Assumption \ref{A4}  that the Markov chain is irreducible,
implies Assumption \ref{A2}.

Assumption \ref{A4} is not significant for running Algorithm 1 and for interpreting its output.
However, it guarantees that the last T-graph $T_K$ consists of a single closed communicating class,
that allows us to establish Eqs. \eqref{ncycles} and \eqref{ncycles2}.
Abandoning it means that  $T_K$ might contain transient 
states and might contain no cycles.  

Assumption \ref{A5} saying that all min-arcs are unique at any stage of Algorithm 1, can be split into two conditions:
\begin{enumerate}[$(i)$]
\item
Every vertex of the graphs $G^{(r)}$ has a unique min-arc, $r =  0,\ldots,R$.
\item
The bucket $\mathcal{B}$ has a unique minimum weight arc throughout the whole run of Algorithm 1.
\end{enumerate}
Assumption \ref{A5} implies
 Assumption \ref{A3} that all optimal W-graphs are unique.
The converse is not true. An example, where Condition $(i)$ fails but the optimal W-graphs are unique is shown in Fig. \ref{fig:m1}(a).
\begin{figure}[htbp]
\begin{center}
(a)\includegraphics[width = 0.75\textwidth]{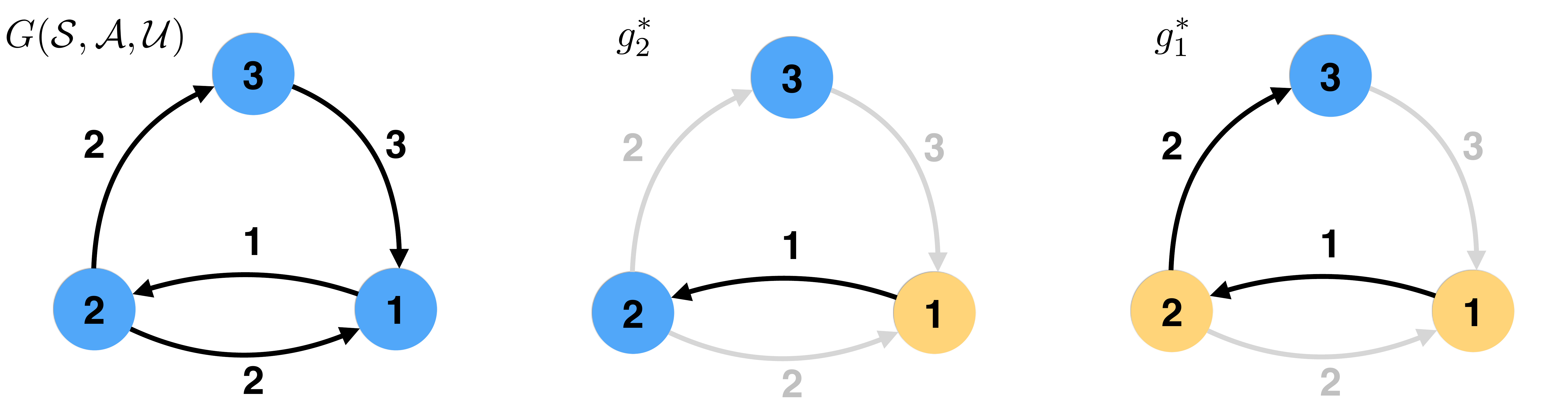}
(b)\includegraphics[width = 0.75\textwidth]{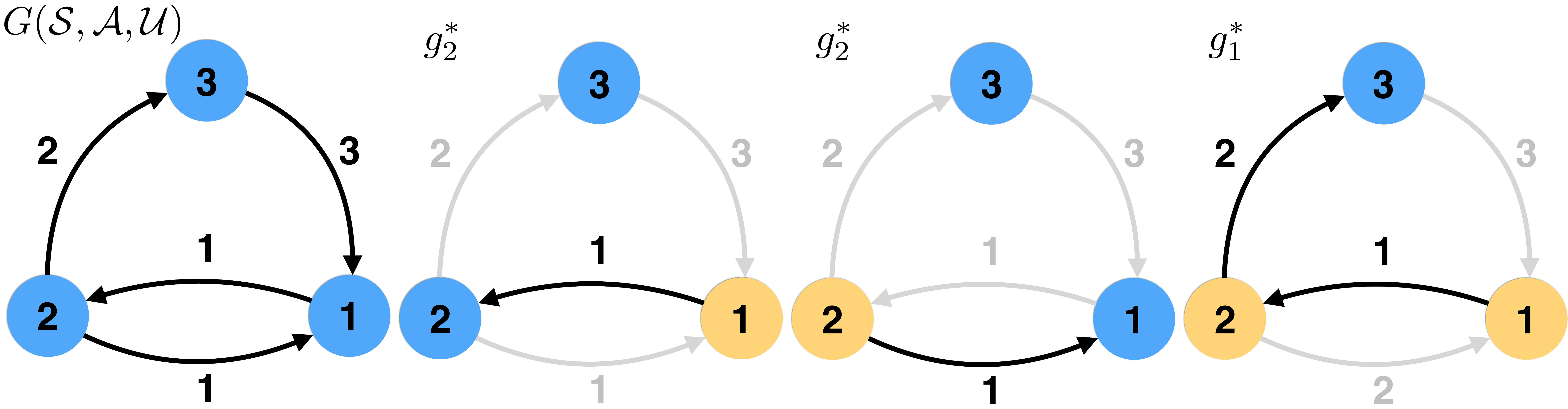}
\caption{
(a): All optimal W-graphs are unique while the min-arc from vertex 2 is not unique.
Furthermore, if the arc $(2\rightarrow 3)$ is chosen as the min-arc from the vertex 2, then the numbers $\gamma_k$ found by Algorithm 1 are 
$\gamma_1 = 1$ and $\gamma_2=2$. If the arc  $(2\rightarrow 1)$ is chosen as the min-arc from the vertex 2, then 
then the numbers $\gamma_k$ found by Algorithm 1 will be $\gamma_1=1$, $\gamma_2=\gamma_3=2$.
(b): All numbers $\Delta$  ($\Delta_2 = 1$ and $\Delta_1=2$) are distinct while the optimal W-graphs are not all unique: 
there are two optimal W-graphs $g^{\ast}_2$ with two sinks.
}
\label{fig:m1}
\end{center}
\end{figure}

The uniqueness of all optimal W-graphs (Assumption \ref{A3}) together with Assumptions \ref{A1} and \ref{A2}
guarantee that the sharp estimates for the eigenvalues given by Theorem \ref{GC_asymeigval} 
are valid. In particular, Assumption \ref{A3}
implies that  all numbers $\Delta_m$, $1\le m\le n-1$, produced by Algorithm 1 are distinct.
The converse is not true.  
An example where all $\Delta_m$'s are distinct but not all optimal W-graphs are unique is shown in Fig. \ref{fig:m1}(b).

If Assumption \ref{A3} is abandoned, { the update rule for the pre-factors \eqref{update_rules}
is no longer justified. A more complicated update rule can be derived instead. We leave it for the future. 
However, as we will show in Section \ref{algorithm2}, Eq. \eqref{update_rules} remains valid for the exponents.  
}


\subsection{Algorithm 2 for the study of metastable behavior}
\label{sec:alg2}
Algorithm 2 is a modification of Algorithm 1 for the case where Conditions $(i)$ and/or $(ii)$ { in Section \ref{sec:sign}}
do not hold. The output of Algorithm 2 is the hierarchy  of the T-graphs $T_p$ 
and the corresponding exponents $\theta_p$, $1\le p\le P$

The structure of Algorithm 2 is similar to the one of Algorithm 1, however, there are important differences.
First, instead of single min-arcs,  the whole sets of min-arcs of the same weight are moved around. Second,
the role of cycles is played by \emph{nontrivial closed communicating classes}.
{ Recall that a closed communication class is a subset of vertices  $C\subset \mathcal{S}$ 
in a directed graph $G(\mathcal{S},\mathcal{A})$ such that
(a) there is a directed path in $G$ leading from any vertex $i\in C$ to any vertex $j\in C$, and 
(b) if there is a directed path from $i\in C$ to $x\in \mathcal{S}$, then $x\in C$.}
The adjective \emph{nontrivial} means \emph{consisting of more than one vertex}. 
Further, we will omit the word \emph{nontrivial} for brevity and refer to them as \emph{closed communicating classes}.
Communicating classes that are not closed will be called \emph{open communicating classes}.


\begin{algorithm}\\
\label{algorithm2}
{\bf Initialization:} Set the step counter $p=0$ and the recursion depth counter $r=0$. Prepare the bucket 
$\mathcal{B}$ as follows. For each vertex $i\in\mathcal{S}$, denote the weight of min-arcs from $i$ by  $U_{\min}(i)$,
find the set of min-arcs 
$$
\mathcal{A}_{\min}(i):=\{(i\rightarrow j )\in\mathcal{A}~|~j\in\mathcal{S},~U_{ij}=U_{\min}(i)\},
$$
and add the whole set $\mathcal{A}_{\min}(i)$ to the bucket $\mathcal{B}$.
Sort the arcs in $\mathcal{B}$
according to their weights in the non-descending order.\\
The graph $G^{(0)}(\mathcal{S}^{(0)},\mathcal{A}^{(0)},\mathcal{U}^{(0)})$ is the original graph $G(\mathcal{S},\mathcal{A},\mathcal{U})$.\\
Initialize the graph $T=T(\mathcal{S}^{(0)},\emptyset,\emptyset)$. Set $T_0=T$.\\

{\bf The main body of the Algorithm:} Call the function {\tt FindSymTgraphs} with arguments  $p = 0$, $r = 0$, $G^{(0)}(\mathcal{S}^{(0)},\mathcal{A}^{(0)},\mathcal{U}^{(0)})$,
$T(\mathcal{S}^{(0)},\emptyset,\emptyset)$, and $\mathcal{B}$.

{\bf Function} {\tt FindSymTgraphs}$\left(r,k,G^{(r)},T^{(r)},\mathcal{B}\right)$\\
{\tt while} $\{$ $\mathcal{B}$ is not empty {\tt and} $T^{(r)}$ has no closed communicating classes $\}$\\
\hspace*{6 mm}{\tt (1)} Increase the step counter: $p=p+1$;\\
\hspace*{6 mm}{\tt (2)} Set $\theta_{p} = \min_{(i\rightarrow j)\in \mathcal{B}}U_{ij}$;\\
\hspace*{6 mm}{\tt (3)} Transfer the set of min-arcs 
$$
\mathcal{A}_p:=\{(i\rightarrow j)\in\mathcal{B}~|~U_{ij}=\theta_p\}
$$
\centerline{
from the bucket $\mathcal{B}$ to the graph $T^{(r)}$;\\
}
\hspace*{6 mm}{\tt (4)} Set $T^{(r)}_p = T^{(r)}$;\\
\hspace*{6 mm}{\tt (5)} Check whether $T^{(r)}$ has a  closed communicating class; \\
{\tt end while}\\
{\tt if} $\{$ $T$ contains $L>0$ closed communicating classes $\}$\\
\hspace*{6 mm}{\tt (6)} Save the index $p$: set $p_{r+1} = p$;\\
\hspace*{6 mm} {\tt for} $\{$ every  closed communicating class $C^l_{r+1}$, $l = 1,\ldots, L$ $\}$\\
\hspace*{12 mm}  {\tt for} $\{$ every vertex $i\in C^l_{r+1}$  $\}$\\
\hspace*{12 mm}  {\tt (7)} { Discard all arcs with tails at $i$  and heads in $C^{l}_{r+1}$;}\\
\hspace*{12 mm}  {\tt (8)} Update the weights of  arcs with tails at $i$  and heads  $j\notin C^{l}_{r+1}$ according to 
\begin{equation}
\label{update_rule1}
U_{ij} = U_{ij} - U_{\min}(i) + \theta_p;
\end{equation}
\hspace*{12 mm}  {\tt end for}\\
\hspace*{6 mm}  {\tt end for}\\
\hspace*{6 mm}{\tt (9)} Contract the closed communicating classes $C^l_{r+1}$ into  super-vertices $v_{C^l_{r+1}}$: \\
\centerline{
$G^{(r+1)}$ = {\tt Contract}$(G^{(r)},\{ C^l_{r+1},~l = 1,\ldots, L\})$; \\
}
\centerline{
$T^{(r+1)}$ = {\tt Contract}$(T^{(r )},\{ C^l_{r+1},~l = 1,\ldots, L\})$;\\
}
\hspace*{6 mm}{\tt for} $\{$ each super -vertex $v_{C^l_{r+1}}$, $l=1,\ldots,L$ $\}$\\
\hspace*{12 mm}{\tt (10)} denote the weight of the min-arc from $v_{C^l_{r+1}}$
by $U_{\min}(v_{C^l_{r+1}})$ \\
\centerline{
and add it to  $\mathcal{B}$;\\
}
\hspace*{6 mm}{\tt end for}\\
\hspace*{6 mm}{\tt (11)} Call the function {\tt FindSymTgraphs}$\left(r+1,p,G^{(r+1)},T^{(r+1)},\mathcal{B}\right)$;\\
\hspace*{6 mm}{\tt (12)} Expand the super-vertices $v_{C^l_{r+1}}$ back into $C^l_{r+1}$, $l = 1,\ldots, L$:\\
 \centerline{
 {\tt for} $j \ge p_r$   $T^{(r)}_j$ = {\tt Expand}$(T^{(r+1)}_j,\{ C^l_{r+1},~l = 1,\ldots, L\})$;   {\tt end for} 
}
{\tt end if}\\
\noindent
{\tt end }
\end{algorithm}

{
\begin{remark}
The set of the critical exponents $\{\gamma_k\}_{k=1}^K$ can be found during the run of Algorithm 2 
by counting the numbers of vertices/super-vertices $m(p)$ from with arcs of weight $\theta_p$ were added at step $p$, 
and then giving each $\theta_p$ the multiplicity $m(p)$.
\end{remark}
 }

The functions {\tt Contract} and {\tt Expand} are defined as described in Section \ref{sec:contract_expand}.
The update rule \eqref{update_rule1} is the same as the rule for the exponential factors in Eq. \eqref{update_rules}.
{It is consistent with the one used in \cite{freidlin_symmetry} for the construction of the hierarchy of Markov chains 
in the case with symmetry. }
Its justification is the following. Suppose a  closed communicating class $C = \{1,\ldots,q\}\subset \mathcal{S}$  in a T-graph $T$
is formed as a result of the addition of a set of arcs of weight $\theta$. 
Let us approximate the dynamics in $C$ by the generator matrix $L^C$
whose off-diagonal entries $L^C_{ij}$ are nonzero if and only if $(i\rightarrow j)\in T$. 
In this case, $L^C_{ij} = L_{ij}$.
The diagonal entries are defined by $L_{ii}^C=-\sum_{j\in C}L^C_{ij}$.
By construction, if a vertex $i$ of $T$ has
more than one outgoing arc, than all outgoing arcs from $i$ have the same weight $U_{\min}(i)$. 
Therefore, the matrix $L^C$
has the following property: each row of $L^C$ has at least one nonzero off-diagonal entry,
and all nonzero entries $L^C_{ij}$ in the row $i$ are of the same exponential order as the diagonal entry $(L^C)_{ii}\asymp\exp(-U_{\min}(i)/\varepsilon)$.
Therefore, $L^C$ can be decomposed into the product
\begin{equation}
\label{S1}
L^C = DM,\quad {\rm where}\quad D ={\rm diag}\{e^{-U_{\min}(1)/\varepsilon},\ldots,e^{ -U_{\min}(q)/\varepsilon}\},
\end{equation}
and all nonzero entries of $M$ are of order one.
Let $\xi$ be the left eigenvector of $M$ corresponding to the eigenvalue zero: $\xi^TM=0$. Then $D^{-1}\xi$ is the left eigenvector of $L^C$.
Normalizing $D^{-1}\xi$, we obtain the quasi-invariant probability distribution in $C$:
\begin{equation}
\label{S2}
\pi_C(i) = \frac{e^{U_{\min}(i)/\varepsilon}\xi(i)}{\sum_{j\in C} e^{U_{\min}(j)/\varepsilon}\xi(j)},\quad1\le i\le q.
\end{equation}
If $\varepsilon$ is sufficiently small, the denominator in Eq. \eqref{S2} is dominated by the term(s) with the largest exponential factor which is 
$$
\max_{j\in C} U_{\min}(j) = \theta .
$$
Hence,
\begin{equation}
\label{S3}
\pi_C(i) \approx \frac{\xi(i)}{\sum_{\{j\in C~|~U_{\min}(j) = \theta\}}\xi(j)}e^{-(\theta - U_{\min}(i))/\varepsilon},\quad1\le i\le q.
\end{equation}
The escape rate from $C$ along an arc $(i\rightarrow x)$, $i\in C$, $x\notin C$, is { approximated by}
\begin{equation}
\label{S4}
\pi_C(i)L_{ix} \approx \frac{\xi(i)}{\sum_{\{j\in C~|~U_{\min}(j) = \theta\}}\xi(j)}e^{-(\theta - U_{\min}(i))/\varepsilon}\kappa_{ix}e^{-U_{ix}/\varepsilon}
\asymp e^{-(U_{ix} + \theta -U_{\min}(i))/\varepsilon}.
\end{equation}
This validates the update rule \eqref{update_rule1}.

{
\begin{remark}
Applied in the case with no symmetry, Algorithm 2  produces the same set of critical exponents and the same T-graphs as Algorithm 1,
because any cycle formed in the graph $T^{(\cdot)}$ is a closed communicating class in this case.
\end{remark}
}

\subsection{An illustrative example for Algorithm 2}
\label{sec:TexampleSym}
Now let us illustrate Algorithm 2 on an example similar to the one in Section \ref{sec:Texample} 
except for the arc weights are rounded to the nearest integers as shown in Fig. \ref{fig:symex}(Left).
The sets of the min-arcs from every vertex are shown thick black in Fig. \ref{fig:symex}(Right).
All of these arcs form the bucket $\mathcal{B}$.
\begin{figure}[htbp]
\begin{center}
\includegraphics[width=0.75\textwidth]{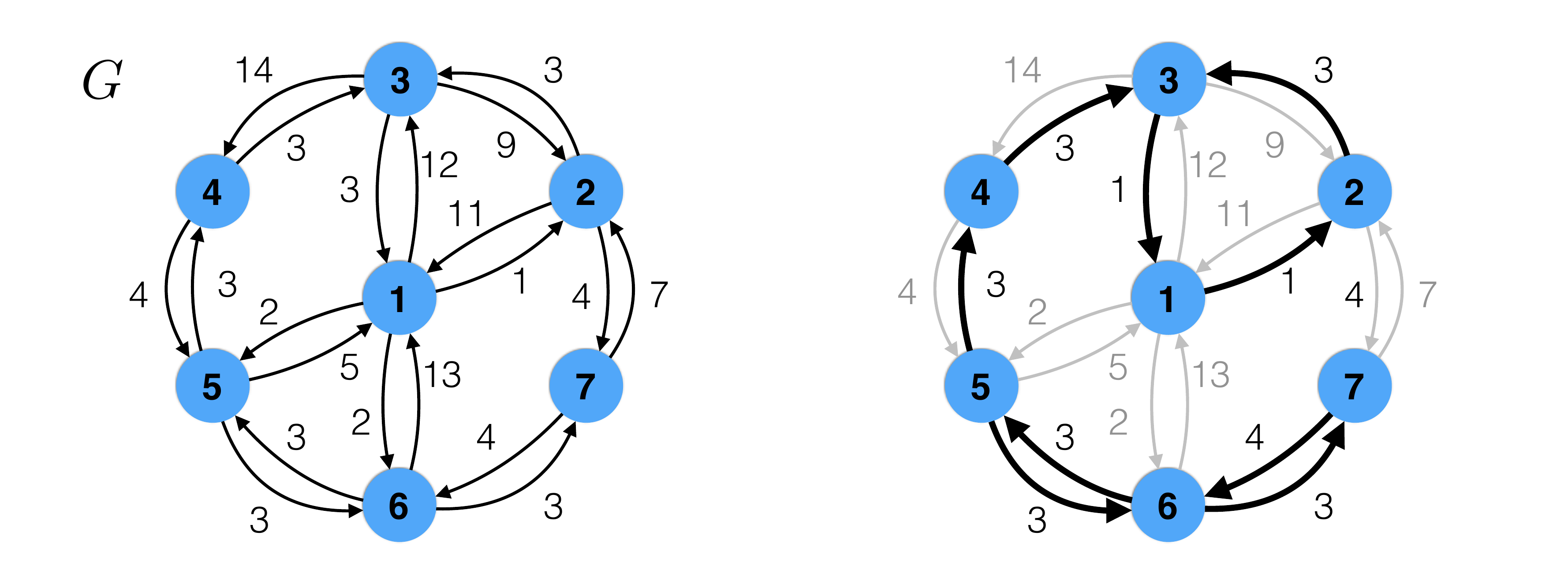}
\caption{An illustrative example for Algorithm 2. Left: The input graph $G(\mathcal{S},\mathcal{A},\mathcal{U})$.
Right: The sets of min-arcs from every vertex are shown thick black. }
\label{fig:symex}
\end{center}
\end{figure}
Then the while-cycle starts. Step $p=1$:  the set of min-arcs of weight 1 is removed from $\mathcal{B}$ 
and added to the graph $T$ forming the T-graph $T_1$ in Fig. \ref{fig:symex2} (Top Left). 
Step $p=2$:  the set of min-arcs of weight 3 is 
removed from $\mathcal{B}$ and added to $T$ resulting in the T-graph $T_2$ in Fig. \ref{fig:symex2} (Top Right). 
The vertices 5 and 6 constitute an open communicating class. Algorithm 2 does not do anything special about it.
The vertices 1, 2, and 3 constitute a closed communicating class which is contracted to a super-vertex { $\{1,2,3\}$} as shown in Fig. \ref{fig:symex2} (Bottom Left).
The modified arc weights are highlighted with bold red. Three min-arcs of weight 4 from the new super-vertex  { $\{1,2,3\}$}  are added to the bucket $\mathcal{B}$. 
Step $p=3$: the set of min-arcs of weight 4 is removed from $\mathcal{B}$ and added to $T^{(1)}$ resulting in 
the graph $T^{(1)}_3$ in Fig. \ref{fig:symex2} (Bottom, Middle). 
 It consists of a single closed communicating class that includes all vertices of $T^{(1)}_3$. 
The bucket $\mathcal{B}$ becomes empty.
The fully expanded T-graph $T_3$ is shown in Fig. \ref{fig:symex2} (Bottom, Right).
 \begin{figure}[htbp]
\begin{center}
\includegraphics[width=0.75\textwidth]{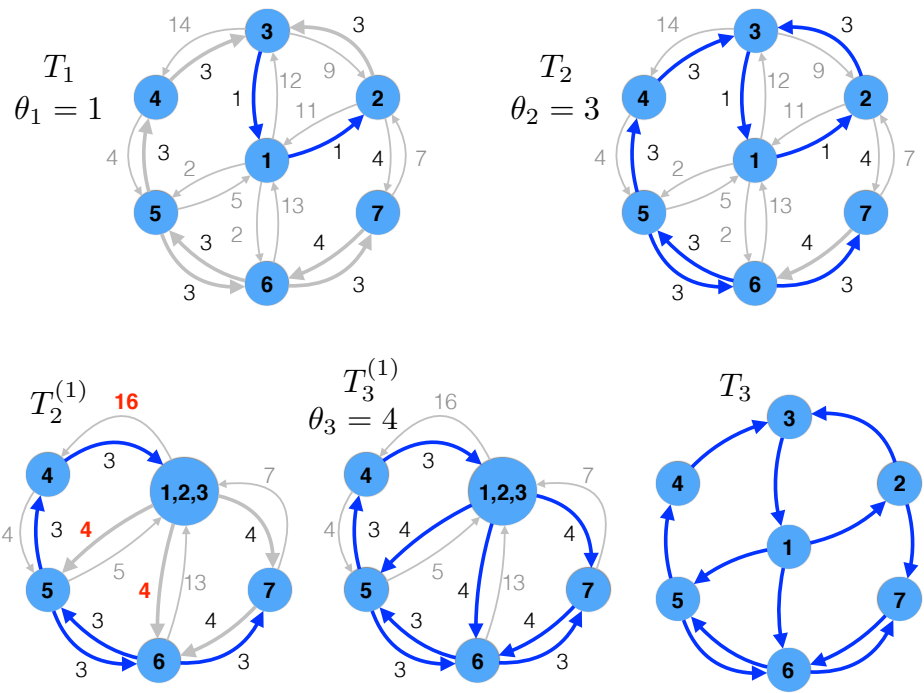}
\caption{
An illustrative example for Algorithm 2, continued. 
The arcs in $\mathcal{B}$ are shown thick grey. The modified arc weights are shown bold red.
}
\label{fig:symex2}
\end{center}
\end{figure}

\section{Interpretation of the output of Algorithm 1 in the case with symmetry}
\label{sec:sym1}
In this Section, we address the question of validity of Algorithm 1 in the case with symmetry.
Since the arc weights are modified during the run of Algorithm 1, it might be impossible to claim that there is no symmetry
before the run is complete.  
Algorithm 1 always picks a single min-arc in the case of multiple min-arcs of the same weight.
The choice of the min-arc is determined by the code 
and can seem random to a user who treats the code as a black-box. 
Using only the output of Algorithm 1, one cannot verify Assumption \ref{A5}: even if 
all numbers $\gamma_k$, $1\le k\le K$, are distinct, Assumption \ref{A5} can still fail as shown in Fig. \ref{fig:m1}(a).
Hence the verification of Assumption 5 must be embedded in the code of Algorithm 1 in order to
make sure that the found graphs are the true T-graphs.

Suppose we a running both Algorithms 1 and 2 in the case where Assumption 5, 
{ which is crucial for the validation of Algorithm 1 but irrelevant to Algorithm 2}, does not hold.  
There are two important differences between them.
\begin{itemize}
\item Algorithm 2 moves around the whole sets of min-arcs of the same weight, while Algorithm 1 moves around only one min-arc in a time.
\item Algorithm 1  contracts cycles into super-vertices 
independent of whether the cycles are closed or open communicating classes, 
while Algorithm 2 contracts only closed communicating classes.
\end{itemize}  
How should the output of Algorithm 1 be interpreted in this case? 
This question is answered in Theorem \ref{theorem:m1}
below. 
To distinguish the ``T-graphs" produced by Algorithm 1 (not necessarily satisfying Definition \ref{def:Tgraph}) from the { true } T-graphs produced by Algorithm 2, 
we will denote the former ones by $\Gamma$. { The set of numbers $\gamma_k$ produced by Algorithm 1 
is not necessarily the true set of critical exponents, but we keep the notation for simplicity.}
To distinguish the buckets $\mathcal{B}$ in Algorithms 1 and 2, we will denote them by $\mathcal{B}'$ and $\mathcal{B}$ respectively.
The graphs  $\Gamma_k$ and $T_p$  are assumed to be fully expanded. 
The recursion levels of Algorithms 1 and 2 will be denoted by $r'$ and $r$ respectively. 

\begin{theorem}
\label{theorem:m1}
\begin{enumerate}
\item 
The set of distinct numbers $\gamma_k$ produced by Algorithm 1 coincides with the set $\{\theta_p\}_{p=1}^P$ produced by Algorithm 2.
\item 
Let $K_p$ be the largest $k$ such that $\gamma _k=\theta_p$. Set $K_0 = 0$.
The graphs $\Gamma_{k}$ are subgraphs  of $T_p$ for all $K_{p-1}<k\le K_p$, $1\le p\le P$.
\item
$C$ is a closed communicating class of $T_p$ if and only if $C$ is a 
closed communicating class of $\Gamma_{K_p}$.
\item
State $i$ is an absorbing state of $T_p$ if and only if it is an absorbing state of $\Gamma_{K_p}$.
\end{enumerate}
\end{theorem}

The proof of Theorem \ref{theorem:m1} is conducted by induction in the recursion level $r$ in Algorithm 2.
It is found in Appendix \ref{App_B}. 

Theorem \ref{theorem:m1} is illustrated in Figs. \ref{fig:Alg2Sym} and \ref{fig:Alg1Sym}.
Algorithms 1 and 2 are applied to the same Markov chain with symmetry. 
Let us compare the T-graphs $T_1$, $T_2$, $T_3$, and $T_4$
 in Fig. \ref{fig:Alg2Sym}
with the graphs $\Gamma_{K_1}=\Gamma_3$, $\Gamma_{K_2} = \Gamma_5$, $\Gamma_{K_3} = \Gamma_7$ and $\Gamma_{K_4}=\Gamma_8$
in Fig. \ref{fig:Alg1Sym} 
respectively. 
We observe that the absorbing states highlighted with lime-green and closed communicating classes  highlighted
with turquoise-blue coincide 
in the corresponding graphs. However, some arcs might be missing in the $\Gamma_{K_p}$ graphs in comparison with the 
corresponding $T_p$ graphs. As a result, $(i)$  the $\Gamma_{K_p}$ graphs might describe the dynamics in the closed communicating classes incompletely,
and $(ii)$ the $\Gamma_{K_p}$  might fail to predict accurately to which recurrent states the process goes if it starts at a transient state.
For example, the arc $3\rightarrow 4$ is missing in the closed communicating classes in $\Gamma_7$ and $\Gamma_8$.
State 4, a transient state on the timescale $e^{1/\varepsilon}$, is not connected to  the absorbing state 3 in $\Gamma_3$ as it is in the T-graph $T_1$.
Furthermore, some arcs might acquire unphysical weights due to the contraction of cycles by Algorithm 1 that are not closed communicating classes.
For example, the arcs $4\rightarrow 5$ and $4\rightarrow 3$ are both min-arcs from state 4 and their weights are 1 in $T_p$, $p=1,2,3,4$. This means that a Markov 
process starting at state 4 proceeds to states 3 or 5 with equal probabilities. However, 
the arc $4\rightarrow 3$ acquires weight 3 while the arc $4\rightarrow 5$ keeps its weight 1 in $\Gamma_7^{(1)}$. 
\begin{figure}[htbp]
\begin{center}
\includegraphics[width=0.75\textwidth]{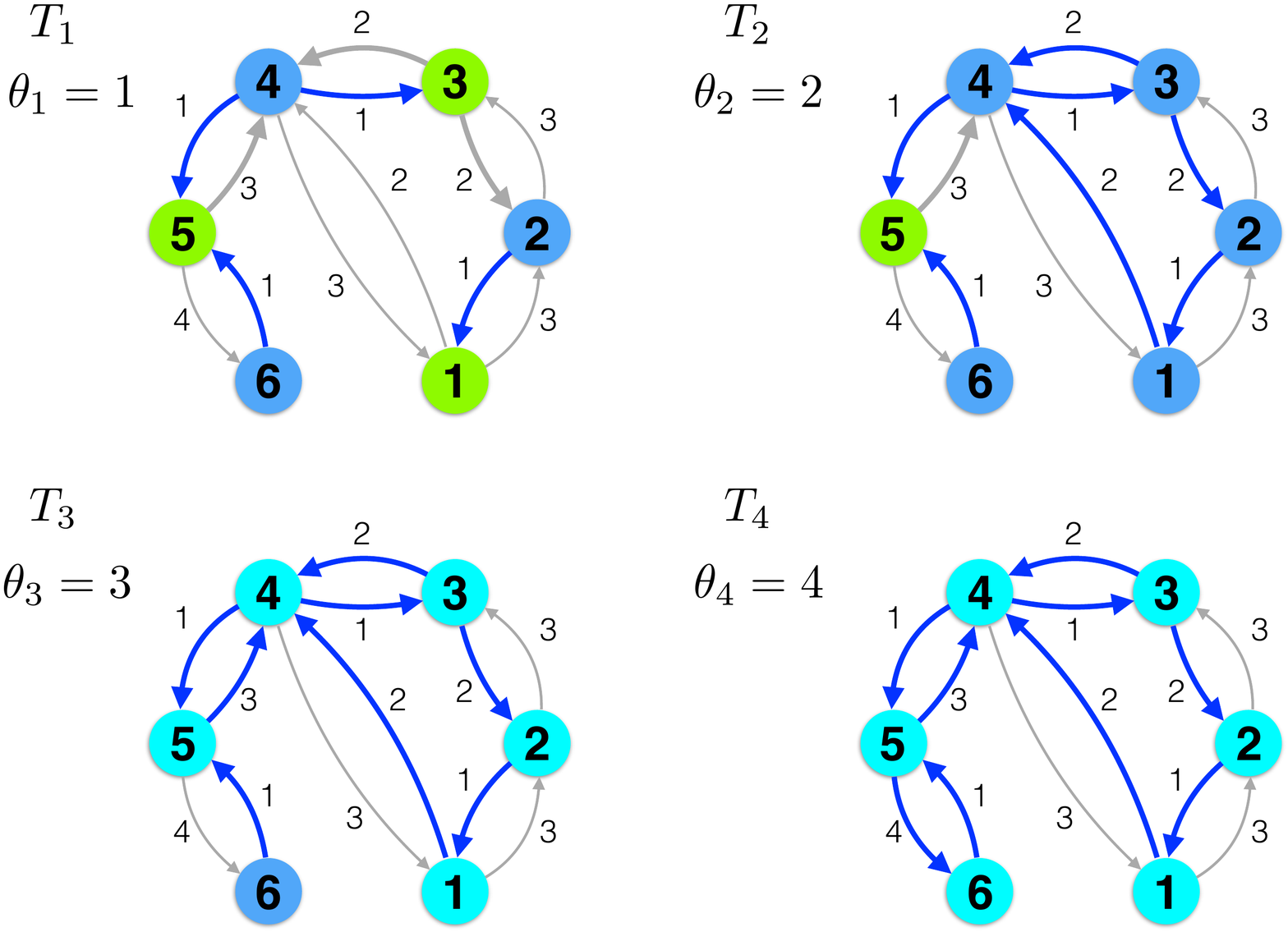}
\caption{An illustration to Theorem \ref{theorem:m1}. The T-graphs and the numbers $\theta$ produced by Algorithm 2.
Absorbing states are highlighted with lime-green, while closed communicating classes are highlighted with turquoise-blue.}
\label{fig:Alg2Sym}
\end{center}
\end{figure}
\begin{figure}[htbp]
\begin{center}
\includegraphics[width=0.75\textwidth]{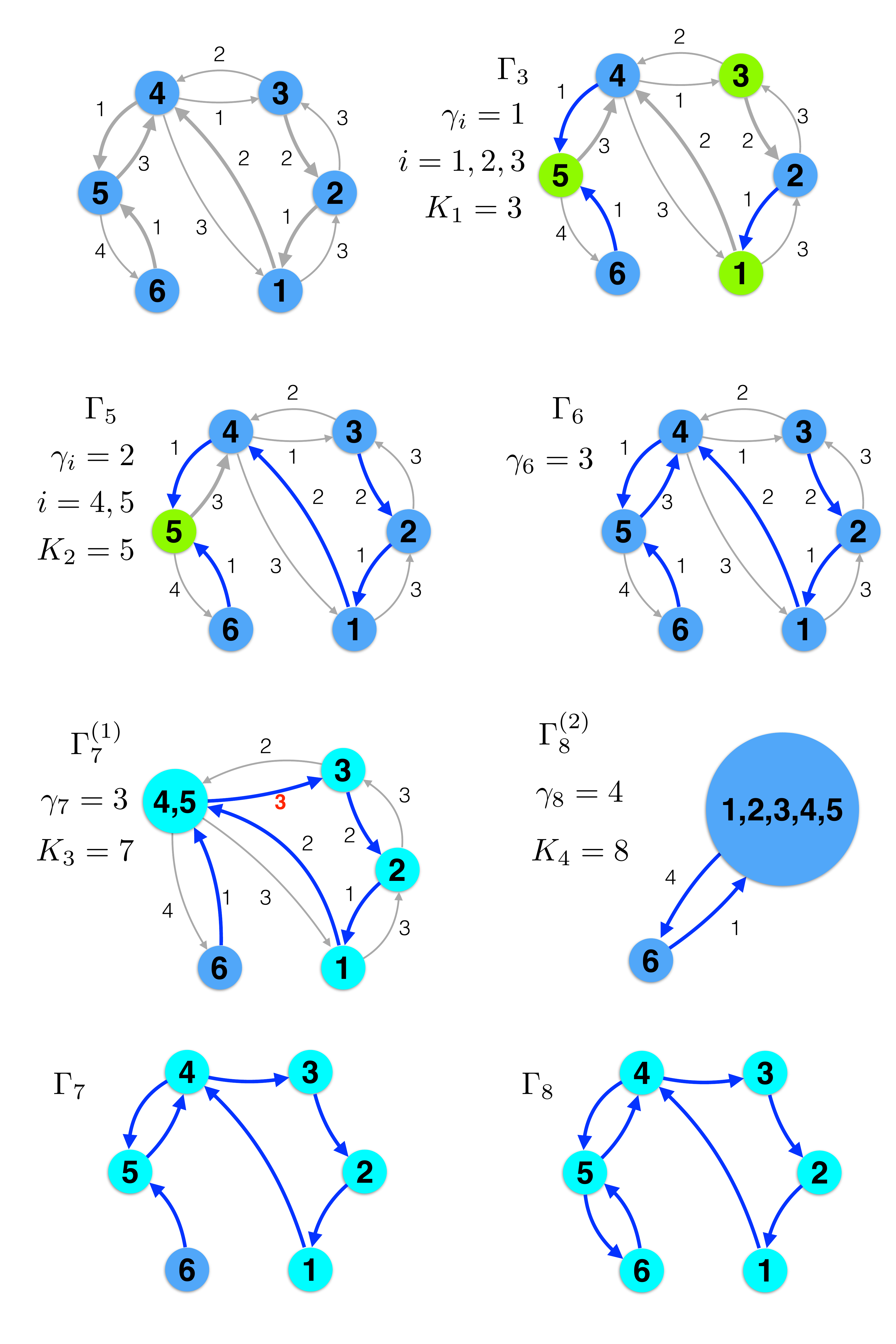}
\caption{ An illustration to Theorem \ref{theorem:m1}. The graphs $\Gamma$ and the numbers $\gamma$ produced by Algorithm 1.
Absorbing states are highlighted with lime-green, while closed communicating classes are highlighted with turquoise-blue.}
\label{fig:Alg1Sym}
\end{center}
\end{figure}


\section{A real world inspired example:  walks of molecular motors}
\label{sec:MM}
In this Section, we will demonstrate the relevance of the time-irreversible and  symmetric Markov chains with 
pairwise rates of the order of $\exp(-U_{ij}/\varepsilon)$ to the real world. The example considered is based on
Astumian's work \cite{Astumian} on molecular motors. 

Molecular motors are molecules that are capable of ``walking" on a substrate 
by converting chemical free energy (often provided by the ATP hydrolysis)
into work. 
The sequence of conformational changes of a molecular  motor can be described as a random walk.
At chemical and thermal equilibrium, transition rates from a conformation $i$ to another conformation $j$ would be 
of the form
$ k_{ij} = A \exp \left(  -(F_{ij} - F_{i})/ (k_B T)  \right) $,
where $F_i$ and $F_{ij}$ are the free energies at state $i$ and the barrier separating $i$ and $j$ { respectively}, 
$T$ is the absolute temperature, and $k_B$ 
is the Boltzmann constant. In this case, the corresponding Markov chain is time-reversible, and biased motion is impossible. 
However, when the chemical reaction (ATP hydrolysis) is no longer at chemical equilibrium due to the excess of ATP,  
the detailed balance (the time-reversibility of the Markov chain) is lost, and biased motion { may arise.}

\begin{figure}[htbp]
\begin{center}
\includegraphics[width=0.5\textwidth]{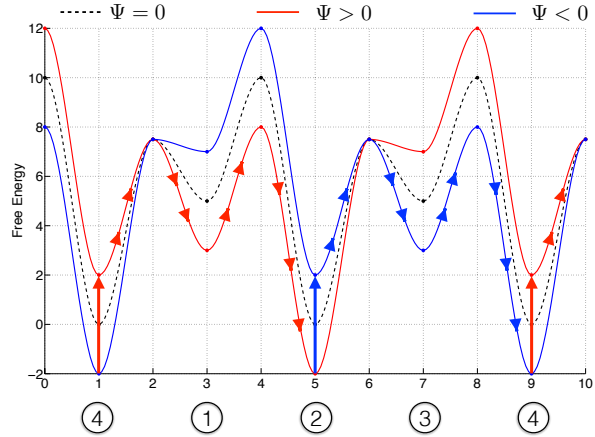}
\caption{The time-dependent free energy landscape of a two-headed molecular motor.
The forward motion, shown with the arrows, occurs a certain range of values of the rate of switching between $\Psi$ and $-\Psi$. }
\label{MMotor1}
\end{center}
\end{figure}

We set up an example based on the dynamics of kinesin, a biomolecular motor, moving on a microtubule, described in \cite{Astumian},
and analyze it by means of  Algorithm 2. Kinesin is a polar protein in the sense that it has distinguishable ends, ``front" and ``back",  that allows us
to identify its forward and backward motion.
Kinesin has two heads, left and right, and moves on its track in a walking manner, that can be viewed as a random walk in a four-state space \cite{Astumian}: 
$1 = $ \{right head front, left head back\}; $2 = $ \{right head attached, left head free\}; $3 =$ \{right head back, left head front\}; $4 =$ \{right head free, left head attached\}.
The possible transitions are $i\rightarrow i+1$, $i=1,2,3$, and $4\rightarrow 1$,  as well as $i\rightarrow i-1$, $i=2,3,4$, and $1\rightarrow 4$.
Cycling through the states in the order $1 \rightarrow 2 \rightarrow 3 \rightarrow 4 \rightarrow 1$ 
leads to a step forward, 
while cycling in the order $1 \rightarrow 4 \rightarrow 3 \rightarrow 2 \rightarrow 1$ leads to a step backward.
Of course, forward or backward motion can occur only if the Markov chain is time-irreversible, that is achieved owing to
the ATP hydrolysis. In the model proposed in \cite{Astumian}, 
its effect boils down to the introduction of the function $\psi(t)$ switching stochastically at rate $e^{-\zeta/\varepsilon}$ between two values, $+\Psi$ and $-\Psi$, 
corresponding to two possible chemical states.
As a result, the free energy  landscape becomes time-dependent as shown in Fig. \ref{MMotor1}.
The dashed black curve corresponds to $\Psi =0$. The red and blue curves correspond to $\psi(t) = \Psi$ and $\psi(t)=-\Psi$ respectively.
For a certain range of  $\zeta$, the forward motion, shown with arrows in Fig. \ref{MMotor1}, will occur. 
We made up the free energy landscapes in Fig. \ref{MMotor1} to mimic the shapes of the graphs in  Fig. 4 in \cite{Astumian}
and to provide us with the input data for Algorithm 2. We picked  $\Psi=2$ and the following values of free energies:
$$
F_1=F_3 = 5,\quad F_2=F_4 = 0,\quad F_{12} = F_{21}=F_{34} =F_{43}= 10,\quad F_{41} =F_{14} = F_{32}= F_{23} = 7.5, 
$$
As in \cite{Astumian}, the transition rates between states 1 and 4  and states 2 and 3 are affected by the chemical state while
the ones between state 3 and 4  and states 1 and 2 are not: 
\begin{align*}
L_{41}& = e^{-(F_{41}-F_4-\psi(t))/\varepsilon},\quad L_{14} = e^{-(F_{41}-F_1+\psi(t))/\varepsilon},\\
L_{23}& = e^{-(F_{23}-F_2+\psi(t))/\varepsilon},\quad L_{32} = e^{-(F_{23}-F_3-\psi(t))/\varepsilon},\\
L_{34}& = e^{-(F_{34}-F_3)/\varepsilon},\quad L_{43} = e^{-(F_{34}-F_4)/\varepsilon},\\
L_{12}& = e^{-(F_{12}-F_1)/\varepsilon},\quad L_{21} = e^{-(F_{21}-F_2)/\varepsilon}.
\end{align*}

To distinguish between two chemical states corresponding to $+\Psi$ and $-\Psi$, we double the state-space to
 $\{1_+, 2_+, 3_+, 4_+, 1_-, 2_-, 3_-, 4_- \}$. The subscripts $'+'$ and $'-'$ correspond to $+\Psi$ and $-\Psi$ respectively.
We assume that the switch between the two chemical states $+\Psi$ and $-\Psi$ is a Poisson process with rate
$e^{-\zeta/\varepsilon}$. The resulting Markov chain is depicted in Fig. \ref{Mchain} (a).

Using Algorithm 2, we analyze the dynamics of the molecular {motor} in the limit $\varepsilon\rightarrow 0$ for $0<\zeta<\infty$.
The stopping criterion in Algorithm 2 is chosen to be ``stop once there is a closed communicating class 
containing at least one state out of  $\{1_{+},1_{-}\}$ and at least one state out of $\{3_{+},3_{-}\}$".
The physical meaning of this criterion is ``stop as soon as the most likely way to switch from 
\{right head front, left head back\} to \{right head back, left head front\} and back is found.
The T-graphs corresponding to the largest timescale achieved by Algorithm 2 before its termination will
be referred to as the ``final T-graphs" for brevity.

It is evident from our analysis that the most likely switching process between the states 
\{right head front, left head back\} and \{right head back, left head front\} and back, described by the final T-graphs,
undergoes qualitative changes at the following set of critical values of $\zeta$: 0.5, 4.5, 5.0, 5.5, 6.0, 9.5, and 10.0.
The application of Algorithm 2 for $ 6<\zeta<9.5$ is sketched in Figs. \ref{Mchain}(b-i).
The final T-graphs for all intervals 
of $\zeta$ bounded by its critical values are shown in Fig. \ref{fig:step}.

If $\zeta<0.5$, the molecular motor randomly steps forward and backward on the timescale $e^{10/\varepsilon}$. Its expected displacement is zero. 
If $0.5<\zeta<4.5$, the  motor walks  forward, and the steps occur on the timescale $e^{(10.5-\zeta)/\varepsilon}$ which is less that $e^{10/\varepsilon}$
but greater than $e^{6/\varepsilon}$.
The interval $4.5<\zeta<6.0$ is the sweet spot: the  motor walks forward on  the minimal possible timescale $e^{6/\varepsilon}$.
Some qualitative changes in the walking style, { indicated by} the corresponding  
final T-graphs (see  Fig. \ref{fig:step}), occur at $\zeta=5.0$ and $\zeta = 5.5$, but they do not affect the stepping rate.
If $6.0<\zeta<10.0$, the motor walks forward at the increasing timescale  $e^{\zeta/\varepsilon}$. A qualitative change in the walking style occurs at $\zeta=9.5$.
If $\zeta>10.0$, each step forward is followed by a step backward, and each step backward is followed by a step forward.
Steps occur on the timescale  $e^{10/\varepsilon}$. The transitions between the chemical states occur on the larger timescale $e^{\zeta/\varepsilon}$, 
and they do not help the motor to walk.
In summary, the motor walks forward if $0.5<\zeta<10$, and it does so at the minimal possible timescale $e^{6/\varepsilon}$ 
{ if} $4.5<\zeta<6.0$.

 \begin{figure}[htbp]
\begin{center}
\centerline{
(a)\includegraphics[width = 0.33\textwidth]{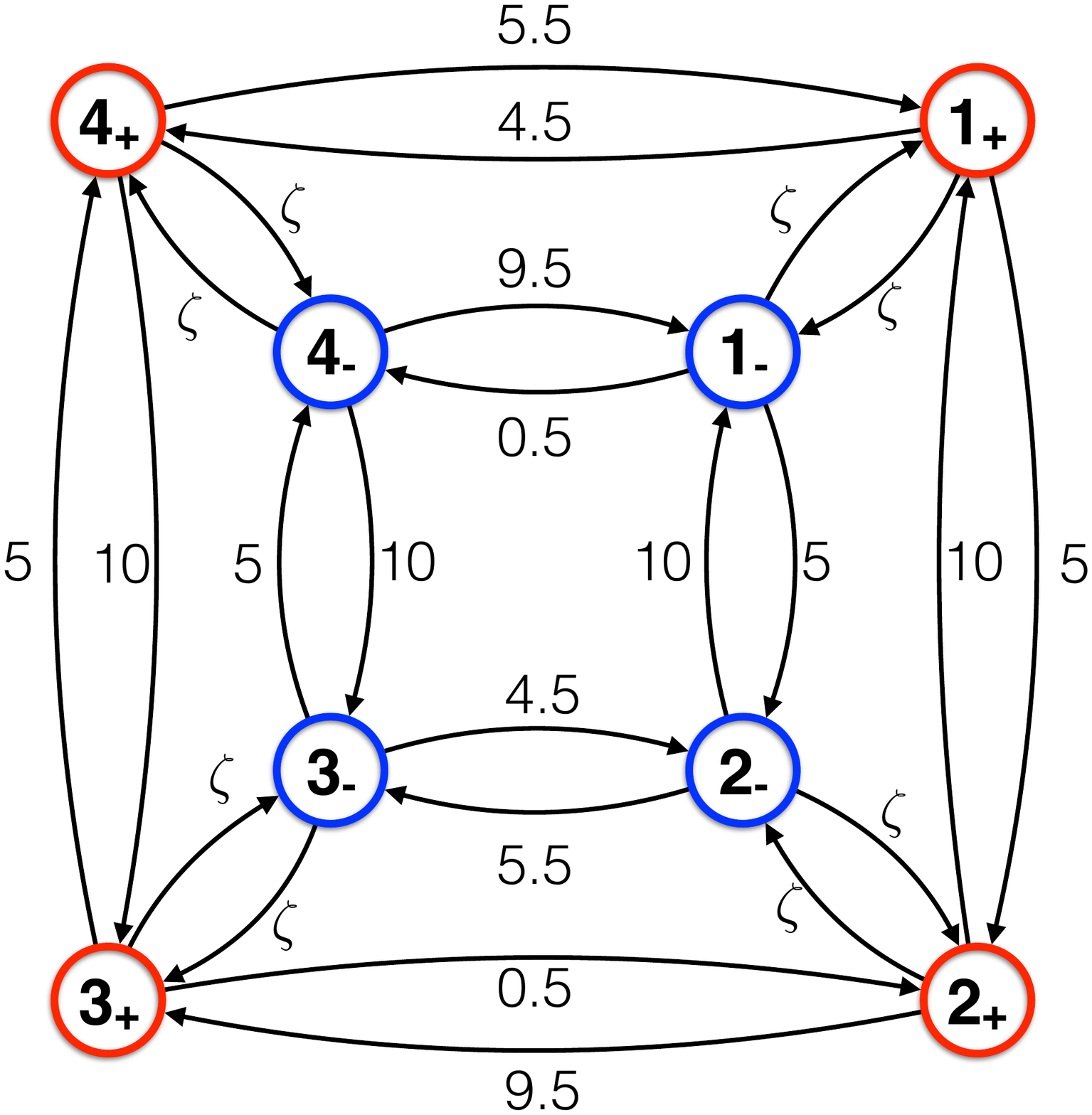}
(b)\includegraphics[width = 0.33\textwidth]{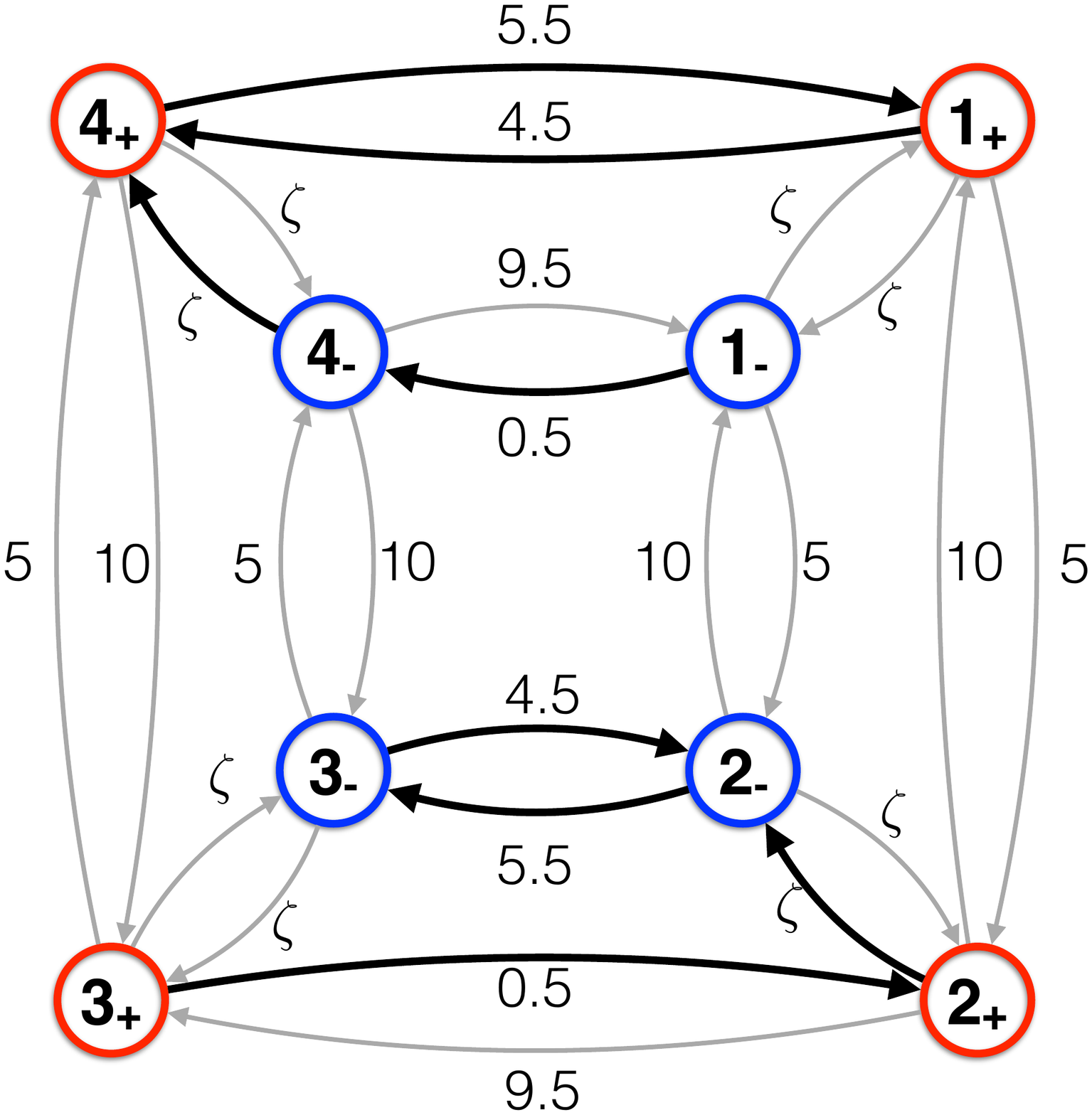}
(c)\includegraphics[width = 0.33\textwidth]{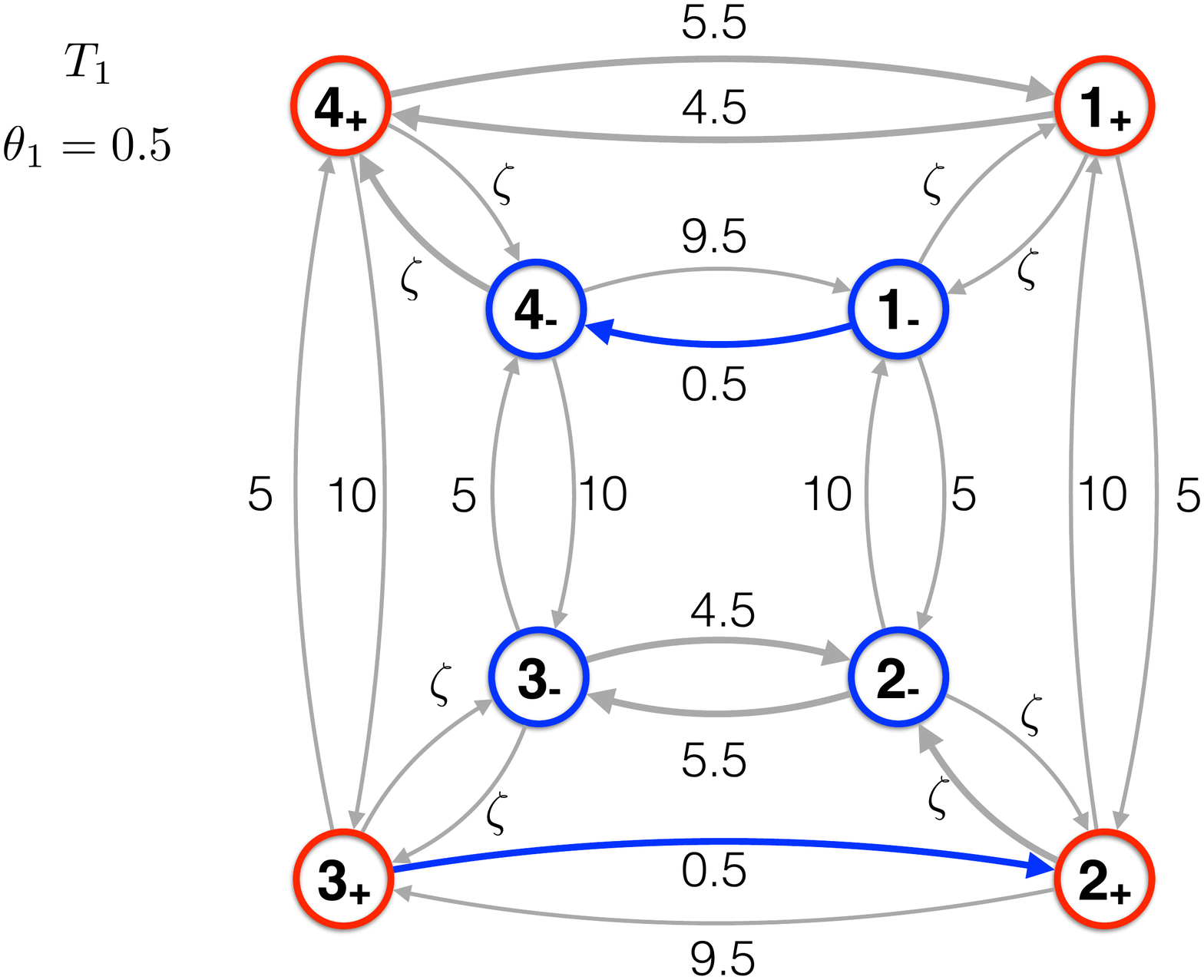}
}
\centerline{
(d)\includegraphics[width = 0.33\textwidth]{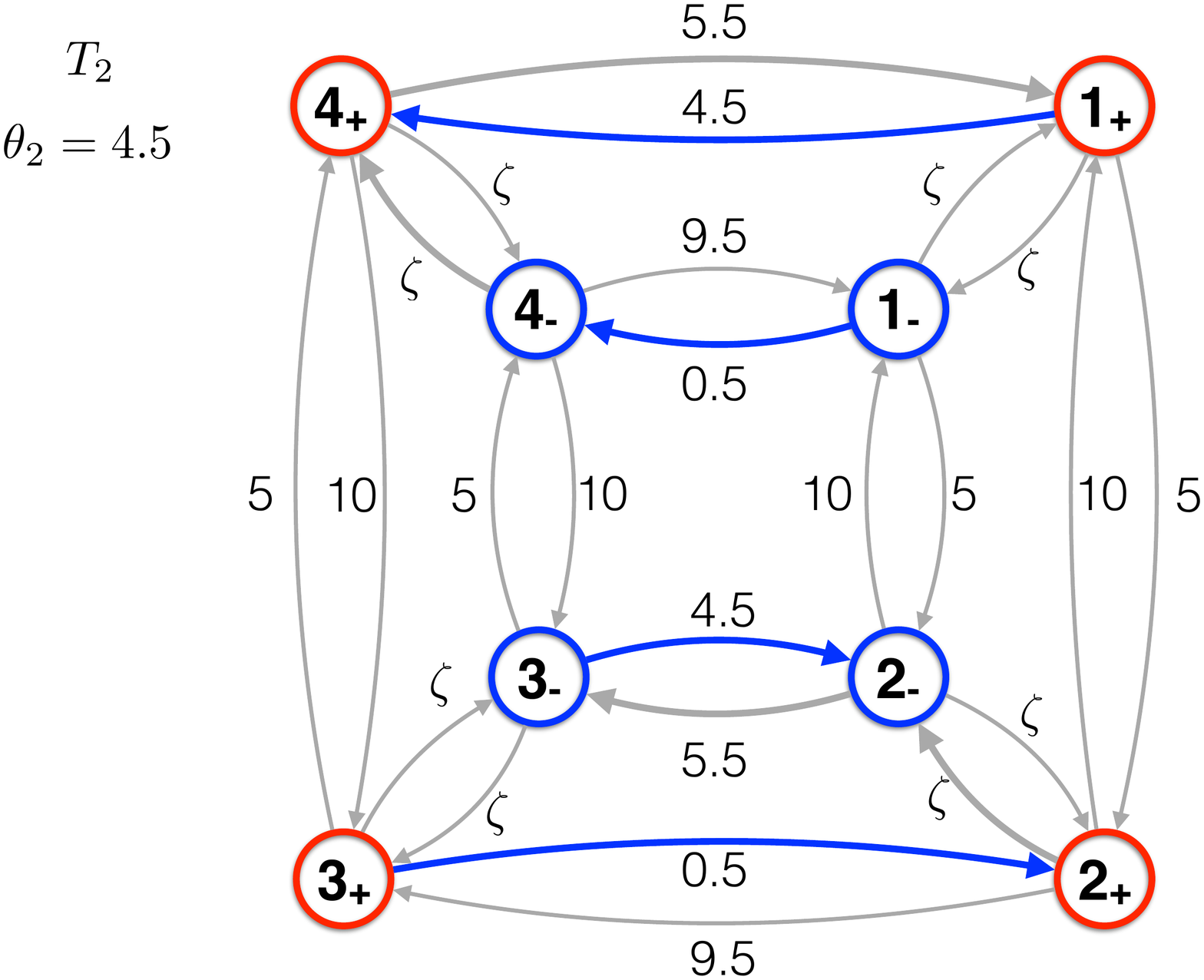}
(e)\includegraphics[width = 0.33\textwidth]{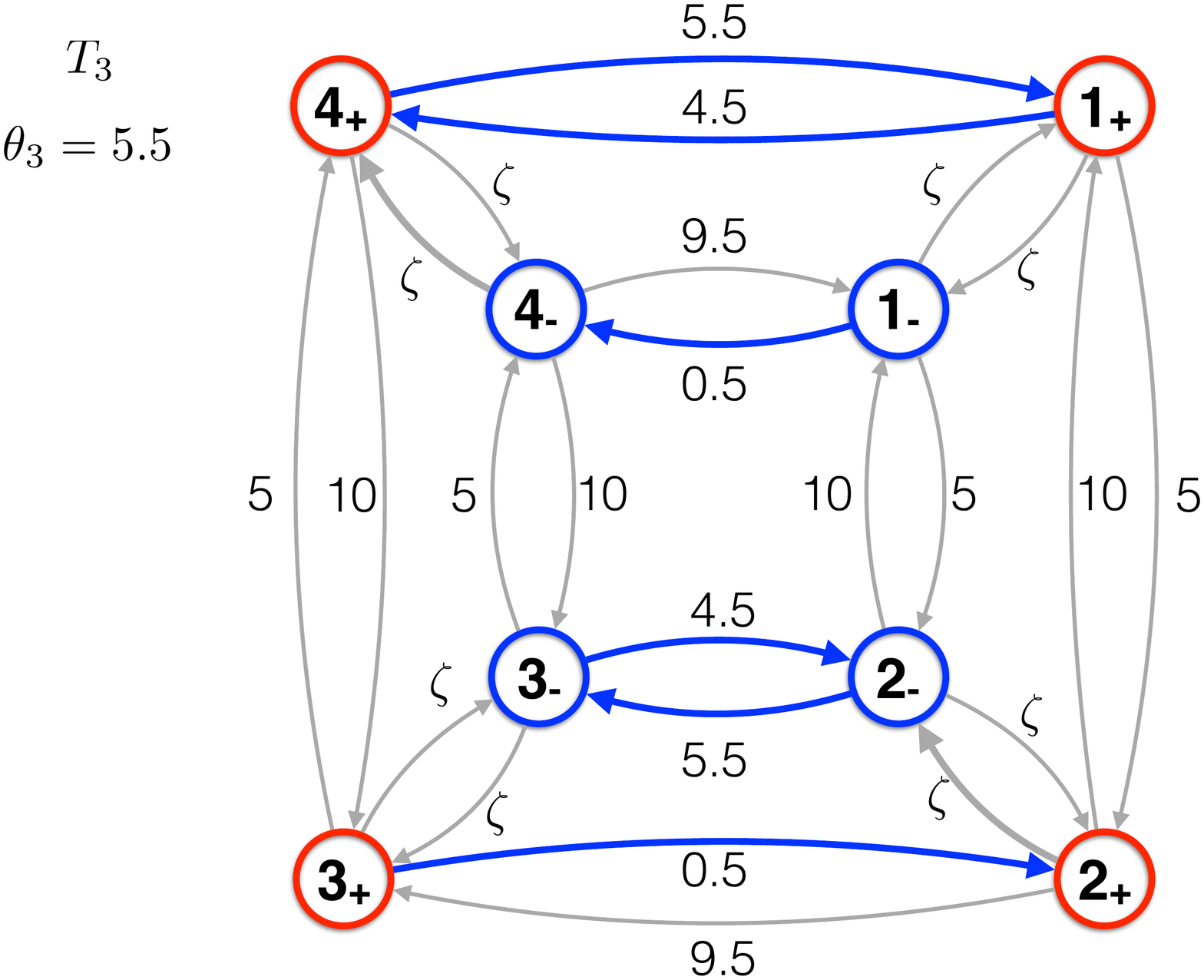}
(f)\includegraphics[width = 0.33\textwidth]{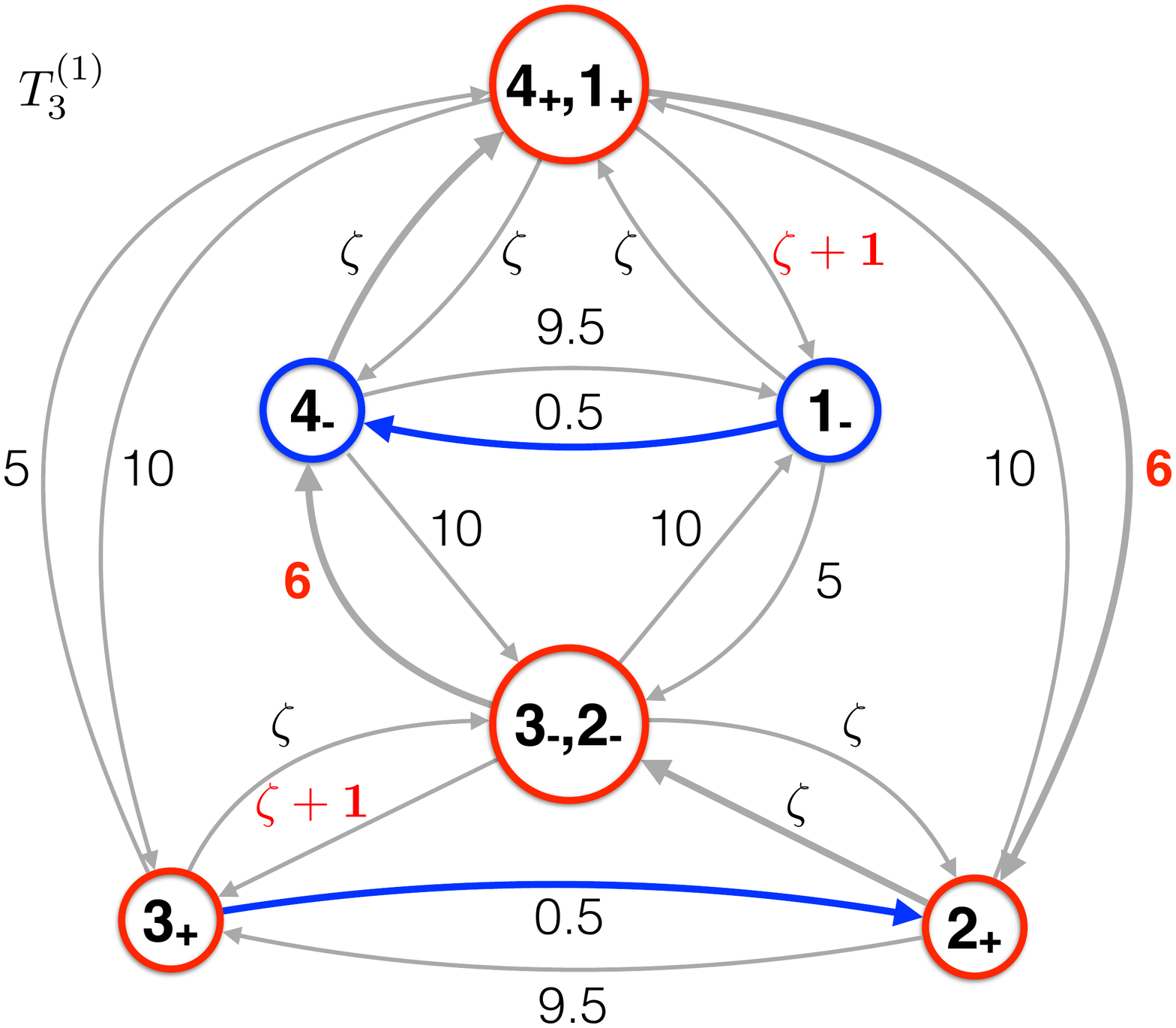}
}
\centerline{
(g)\includegraphics[width = 0.33\textwidth]{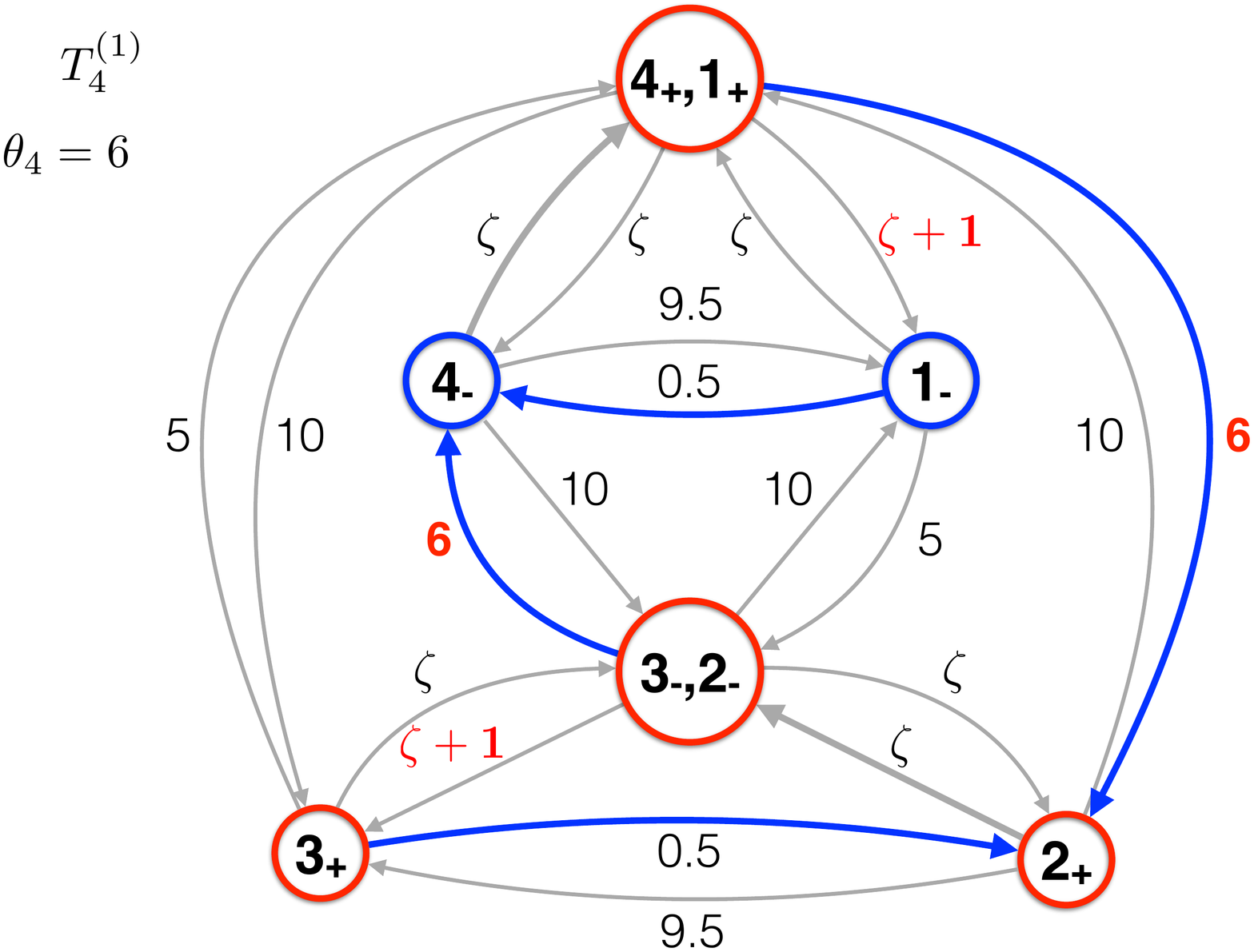}
(h)\includegraphics[width = 0.33\textwidth]{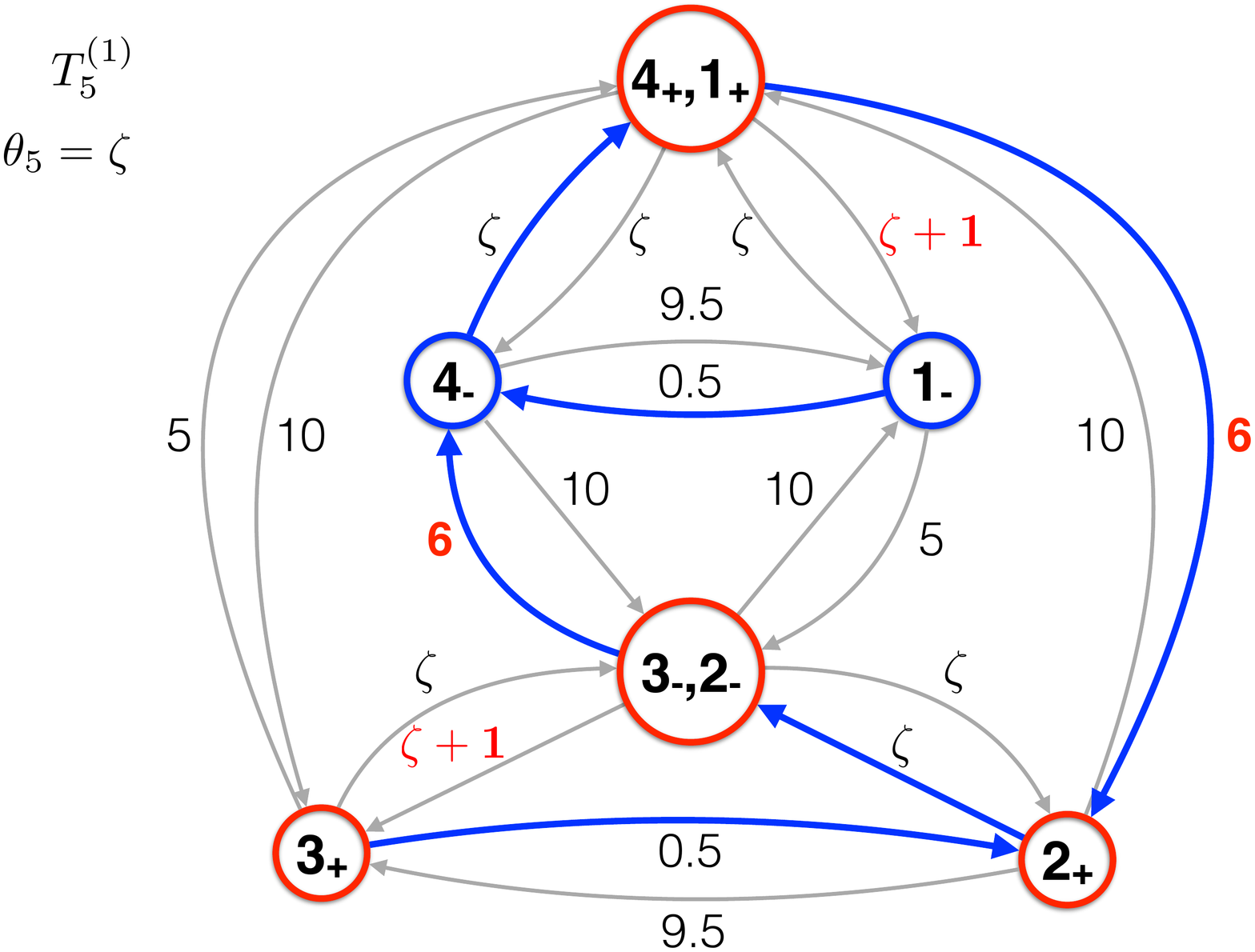}
(i)\includegraphics[width = 0.33\textwidth]{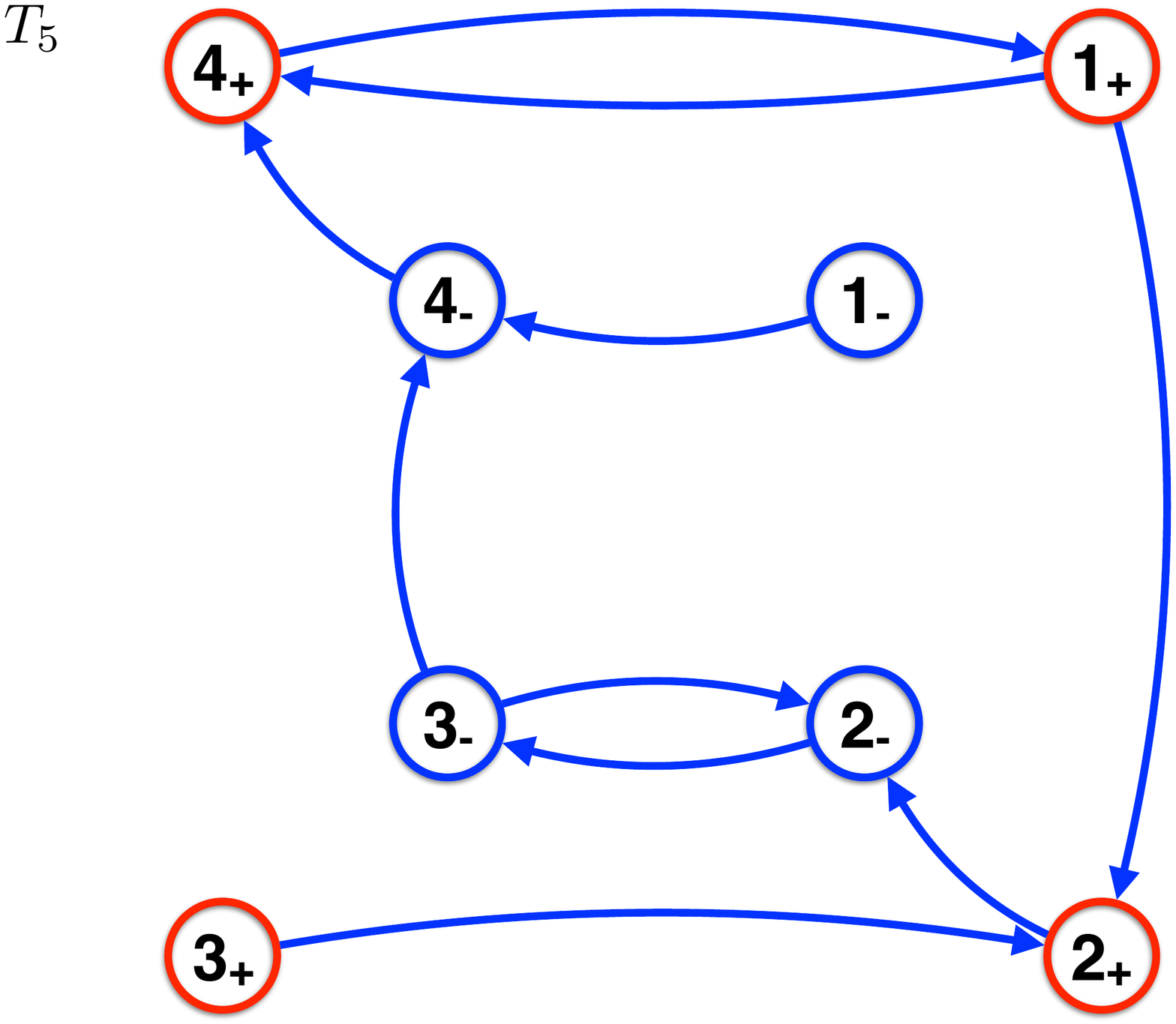}
}
\caption{
An application of Algorithm 2 to the analysis of the dynamics of a two-headed molecular motor with distinguishable ``front" and ``back".
Algorithm 2 terminates as soon as  a closed communicating class 
containing at least one state out of  $\{1_{+},1_{-}\}$ and at least one state out of $\{3_{+},3_{-}\}$ is found.
(a): The Markov chain describing the dynamics of { the} two-headed molecular motor. 
The { parameter} $\zeta$ is assumed to satisfy $6<\zeta<9.5$.
(b): The min-arcs from each vertex are extracted. 
(c): The T-graph $T_1$ contains no closed communicating classes.
(d): The T-graph $T_2$ contains no closed communicating classes.
(e): The T-graph $T_3$ contains two closed communicating classes: $\{4_{+},1_{+}\}$ and $\{3_{-},2_{-}\}$. 
(f): The closed communicating classes are contracted to super-vertices. The modified arc weights are shown in red.
(g):  The T-graph $T_4^{(1)}$ contains no closed communicating classes.
(h):  The T-graph $T_5^{(1)}$ contains a single closed communicating class $\{\{4_{+},1_{+}\},2_{+},\{3_{-},2_{-}\},4_{-}\}$. States
$1_{-}$ and $3_{+}$ are transient.
The stopping criterion is met.
(i): The T-graph $T_5$ corresponding to the forward motion shown in Fig. \ref{MMotor1} by arrows.
}
\label{Mchain}
\end{center}
\end{figure}
\begin{figure}[htbp]
\begin{center}
\includegraphics[width=\textwidth]{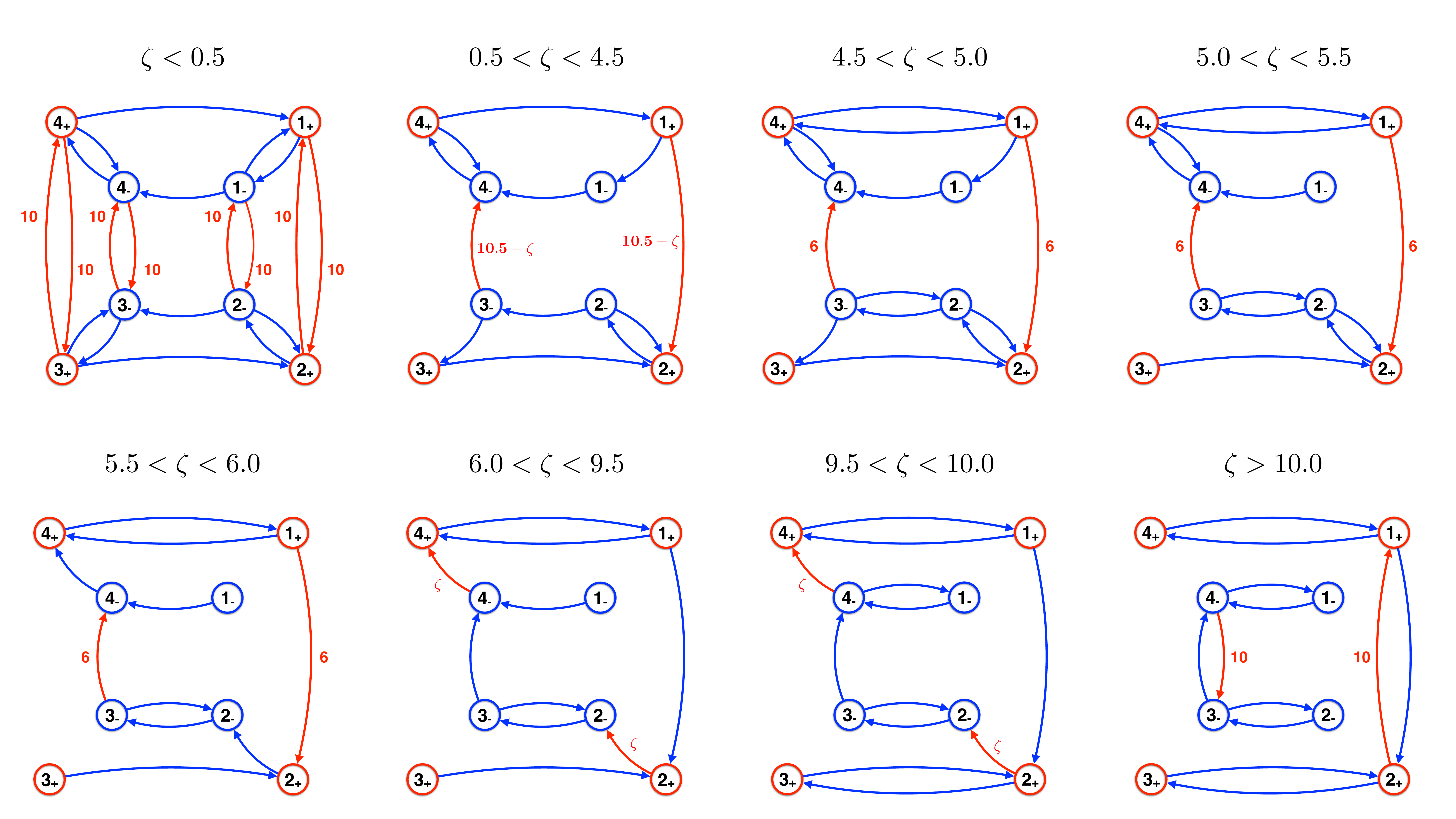}
\caption{The final T-graphs indicating the most probable switching processes between the states 
\{right head front, left head back\} and \{right head back, left head front\} and back
are shown 
for all intervals of $\zeta$ bounded by its critical values. The slowest transitions are indicated by the red arcs, and the exponential factors 
of their rates are {displayed}.
}
\label{fig:step}
\end{center}
\end{figure}

{
\section{Conclusion}
\label{sec:conclusion}
We have introduced the T-graphs indicating the most likely to observe transitions up to the corresponding critical timescales 
in possibly time-irreversible continuous-time Markov chains with 
exponentially small pairwise transition rates. We designed Algorithm 1 and Algorithm 2
to find the sequence of critical timescales and build the hierarchy of T-graphs in the cases without and with symmetry respectively. 
Both Algorithms 1 and 2 
can be used for non-symmetric Markov chains. 
Algorithm 1 is more straightforward for programming \cite{CG}.
Furthermore, in the non-symmetric case,
one can extract asymptotic estimates for eigenvalues and extract the hierarchy of the optimal W-graphs from the output of Algorithm 1.}

{
Algorithm 1 can still run on Markov chains with symmetry.
 The presence of symmetry might not be 
apparent from the input data and the output as it is in the example in Fig. \ref{fig:m1} (a).
The symmetry check should be included into the code implementing Algorithm 1.
If symmetry is detected, the output of Algorithm 1 should be interpreted as follows.
The distinct values of $\gamma_k$'s  are the correct numbers $\theta_p$. 
The graphs built by Algorithm 1
do not necessarily satisfy the definition of the T-graphs: they might miss some arcs.
However, the absorbing states and closed communicating classes are identified correctly, while
some arcs 
might be missing within the closed communicating classes. 
}

{
Algorithm 2 is specially designed to handle Markov chains with symmetry. 
It finds the sequence of distinct critical exponents, and 
accurately builds the hierarchy of T-graphs.
}

Many Markov chains serving as models of natural systems involve some symmetry. 
Here we have investigated a toy model inspired by the Astumian's model of kinesin moving on microtubule
and determined the range of rates of chemical switch enabling forward motion for the assumed free energy landscape.
{ We are planning to consider applications to large stochastic networks in our future work. In particular, 
we will analyze  a time-irreversible network with over 6000 vertices modeling the aggregation 
dynamics of Lennard-Jones particles using Algorithm 2.}

\section*{Acknowledgement}
We would like to thank Mr. Weilin Li for valuable discussions in the early stage of this work.
We are grateful to Prof. C. Jarzynski and Prof. M. Fisher for suggesting us to consider molecular motors as an 
example of a natural time-irreversible system with transition rates in the exponential form.
{ We also thank the anonymous reviewers for their valuable feedback and comments.}
This work was partially supported by NSF grants 1217118 { and 1554907}.


\begin{appendices}

\section{Proof of Theorem \ref{GC_asymeigval} } \label{App_A}

The following notations will be used throughout the proof.
  \begin{itemize}
  \item 
 $\mathcal{G}(n-l)$ is the set of all W-graphs with $n-l$ sinks and $l$ arcs for the graph 
  $G(\mathcal{S},\mathcal{A},\mathcal{U})$. The set of subgraphs of $G$ with exactly $l$ arcs emanated from distinct vertices will be denoted by
  $\mathcal{H}(l)$. Note that graphs in $\mathcal{H}(l)$ might contain directed cycles. Therefore, $\mathcal{G}(n-l) \subsetneq \mathcal{H}(l)$.
  \item 
  $\mathcal{S}^l $ is the set of all ordered selections of $l$ vertices out of $n$  (in the combinatorial sense):
  $$\mathcal{S}^l: = \{ i_1 i_2 \cdots i_l ~|~ i_m \in \{ 1, 2 , \cdots, n-1\}, 1 \leq m \leq l , \text{ distinct} \}.$$
  Note that $|\mathcal{S}^l| = n(n-1)\ldots(n-l+1)$.
  \item 
  $\mathcal{O}^l $ is the set of combinations of $l$ vertices out of $n$  (in the combinatorial sense).
   Each combination in $\mathcal{O}^l$ is ordered 
so that $i_1<i_2<\ldots<i_l$, i.e.,
  $$\mathcal{O}^l := \{  i_1 i_2 \cdots i_l  \in \mathcal{S}^l ~|~ i_1 < i_2 < \cdots < i_l \}.$$
 Note that $|\mathcal{O}^l| = \left(\begin{array}{c}n\\l\end{array}\right)$.
  \item 
 Every sequence $j_1,\ldots j_l$  in $\mathcal{S}^l$ can be permuted to an ordered sequence $\{i_1,\ldots, i_l\}$ in $\mathcal{O}^l$.
 This defines the permutation map $\sigma$: $\sigma(j_1,\ldots,j_l) = (i_1,\ldots,i_l)$. 
 Note that the map $\sigma$ is onto but not one-to-one.
  \item
 We will call sequences in $\mathcal{S}^l$ equivalent if and only if they are mapped to the same sequence in $\mathcal{O}^l$. 
  Therefore, $\mathcal{O}^l=\mathcal{S}^l/\sim$, { meaning that $\mathcal{O}^l$ is the set of equivalence classes in $\mathcal{S}^l$}.
 %
   \item
  For any $i_1 \cdots i_l \in \mathcal{O}^l$, we will denote by $L_{ i_1 \cdots  i_l }$ the $l\times l$ submatrix of $L$ 
  consisting of  the intersection of $ i_1, \cdots, i_l $ rows and columns of $L$.
  \end{itemize}
  
 \begin{proof}
 The proof of Theorem \ref{GC_asymeigval} consists of the following three steps.
  \begin{enumerate} [{\bf Step 1:}]
  \item 
 The characteristic polynomial of the generator matrix $L$ is 
  $$ 
  P_L(t) := \text{det} ( t I - L) = t^n + \sum_{l = 1}^{n-1} C_l t^{n-l}
  $$ 
  whose coefficients $C_l$  are given by
  \begin{align}
  C_l 
  &= \sum_{i_1  \cdots i_l \in \mathcal{O}^l} \text{det} ( - L_{i_1 \cdots i_l} )\notag \\
  &= (-1)^l \sum_{i_1  \cdots i_l \in \mathcal{O}^l}  ~\sum_{ \substack{j_1  \cdots j_l \in\mathcal{S}^l\\ j_1  \cdots j_l\sim i_1  \cdots i_l}} 
  (-1)^{| \sigma(j_1, \cdots, j_l) |} L_{i_1 j_1} \cdots L_{i_l j_l}, \label{C_r}
   \end{align}
   where $| \sigma(j_1, \cdots, j_l) |$ is the number of inversions in $\sigma(j_1, \cdots, j_l) $.
  \item 
  Using the zero row sum property of the generator matrix $L$, Eq. (\ref{C_r}) can be further simplified resulting at the following more compact expression:
   \begin{equation}
   \label{C_r_g}
  C_l = \sum_{g \in \mathcal{G}(n-l)} \Pi(g),  \quad{\rm where}\quad \Pi(g) = \prod_{(i\rightarrow j)\in g}L_{ij}.
  \end{equation}
    The key idea of the  derivation of Eq. (\ref{C_r_g}) is  to show that
   \begin{equation}
    \label{C_r_H(r)}
  C_l = (-1)^l \sum_{g \in \mathcal{H}(l) } a_g \Pi(g), \quad{\rm where}\quad 
  a_g = 
  \begin{cases}
  (-1)^l, & \text{ if } g \in \mathcal{G}(n-l),\\
  0, & \text{ if } g \in \mathcal{H}(l) \setminus \mathcal{G}(n-l).
  \end{cases}
  \end{equation}
  \item  
Comparing the coefficient of the characteristic polynomial  
$$
P_L(t) = \text{det} ( t I - L ) = t^n + \sum_{l = 1}^{n-1} C_l t^{n-l} = t (t + \lambda_1) \cdots (t + \lambda_{n-1}),
$$ 
we obtain the following estimates for eigenvalues:
  \begin{align}
  \lambda_m = \frac{ \Pi (g^{\ast}_{m}) }{ \Pi (g^{\ast}_{m+1}) } (1 + o(1) ). \label{Alambda}
  \end{align}
Eq. \eqref{Alambda} is equivalent to Eq. \eqref{GC_lambda}.
  \end{enumerate}
  
 Now we elaborate each step.
  
   \subsection*{Step 1:}
Consider the following polynomial in $n$ variables $t_1, t_2, \cdots, t_n$:
  \begin{align}
  P_L( t_1, t_2, \cdots, t_n) 
  &:= \text{det} ( \text{diag} \{ t_1, t_2, \cdots, t_n\} - L)\notag \\
  &= t_1 \cdots t_n + \sum_{l = 1}^{n - 1} ~\sum_{i_1 \cdots i_l \in \mathcal{O}^l} C_{i_1 \cdots i_l} t_{i_{l+1}} \cdots t_{i_n} \label{ptt}
  \end{align}
  where $\{ i_{l+1}, \cdots, i_n \} = \{ 1, \cdots, n\} \setminus \{ i_1, \cdots, i_l \}$. Replacing all 
  $t_i$  by $t$, we recover the characteristic polynomial $P_L(t)$ where 
  \begin{align}
  C_l = \sum_{i_1 \cdots i_l \in \mathcal{O}^l} C_{ i_1 \cdots i_{l} }. \label{C_det}
  \end{align}
 The term $C_{i_1 \cdots i_l} t_{i_{l+1}} \cdots t_{i_n}$ in Eq. \eqref{ptt} is obtained by picking $t_{i_{l+1}}$, ..., $t_{i_n}$ from the diagonal entries in rows of 
the matrix $\text{diag} \{ t_1, t_2, \cdots, t_n\} - L$ and multiplying them by the determinant of $-{L_{i_1\ldots i_l}}$. Hence
 \begin{align}
 C_{ i_1 \cdots i_l }  = \text{det} \left( -L_{  i_1 \cdots i_l } \right). \label{C_i_s} 
 \end{align}
  Combining Eqs. \eqref{C_det} and \eqref{C_i_s} and applying the  Leibniz formula for determinants we obtain Eq. \eqref{C_r}.
 
 \subsection*{Step 2: }
Consider the product terms  $L_{i_1 j_1} \cdots L_{i_l j_l}$, $ i_1 \cdots i_l \in \mathcal{O}^l, i_1 \cdots i_l \sim j_1 \cdots j_l$
 in Eq. \eqref{C_r}. 
  Suppose the sequences $i_1 \cdots i_l$ and $j_1 \cdots j_l$ agree on exactly $s$ entries and differ at $l-s$ entries, i.e., 
 \begin{align}
 &i_{m_1} = j_{m_1}, \cdots, i_{m_s} = j_{m_s},\\
 &i_{m_{s+1}} \neq j_{m_{s+1}}, \cdots, i_{m_l} \neq j_{m_l}, \text{ where } i_{m_{s+1}} \cdots i_{m_l} \sim j_{m_{s+1}} \cdots j_{m_l} \in \mathcal{S}^{l-s}.
 \end{align}
  Note that $s$ can be any number between $0$ and $l$ except for $l-1$.
Using the zero sum property of $L$,  we obtain
 \begin{align}
& L_{i_1 j_1} \cdots L_{i_l j_l} = L_{ i_{m_1} i_{m_1} } \cdots L_{ i_{m_s}  i_{m_s} }  L_{ i_{m_{s+1}}  j_{m_{s+1}} } \cdots L_{ i_{m_l}  j_{m_l} }\nonumber \\
 &= \left( - \sum_{d_1 \neq i_{m_1} } L_{i_{m_1} d_1}  \right) \cdots \left( - \sum_{d_s \neq i_{m_s} } L_{i_{m_s} d_s}   \right) 
 L_{ i_{m_{s+1}}  j_{m_{s+1}} } \cdots L_{ i_{m_l}  j_{m_l} } \nonumber \\
 &= (-1)^s \sum_{d_1 \neq i_{m_1}, \cdots, d_s \neq i_{m_s} } L_{ i_{m_1} d_1 } \cdots L_{ i_{m_s} d_s } L_{ i_{m_{s+1}} j_{m_{s+1}} } \cdots L_{ i_{m_l} j_{m_l} }. \label{breakloop}
 \end{align} 
It is helpful to consider the collection of graphs $g$ with $n$ vertices and sets of arcs 
$$
\{ (i_{m_1} \rightarrow d_1) , \cdots, (i_{m_s} \rightarrow d_s), (i_{m_{s+1} } \rightarrow j_{m_{s+1}}), \cdots, ( i_{m_l} \rightarrow  j_{m_l} ) \}
$$
corresponding to the products 
 $$
 L_{ i_{m_1} d_1 } \cdots L_{ i_{m_s} d_s } L_{ i_{m_{s+1}} j_{m_{s+1}} } \cdots L_{ i_{m_l} j_{m_l} }
 $$ 
 in Eq. \eqref{breakloop}.
 Each of the vertices $i_1,\ldots,i_l$ of $g$ has exactly one outgoing arc, while the other vertices have no outgoing arcs.
Hence, the graphs $g$ belong to the set of graphs $\mathcal{H}(l)$.
 If $s \leq l-2$, each of the vertices $\{ i_{m_{s+1}}, \cdots, i_{m_l} \}$  of the graph $g$ has exactly one outgoing arc heading to
 $\{ i_{m_{s+1}}, \cdots, i_{m_l} \}$, and exactly one incoming arc
from $\{ i_{m_{s+1}}, \cdots, i_{m_l} \}$. Therefore, the arcs 
$\{(i_{m_{s+1} } \rightarrow j_{m_{s+1}}), \cdots, ( i_{m_l} \rightarrow  j_{m_l} ) \}$ necessarily form cycles in $g$.
The arcs 
 $\{ (i_{m_1} \rightarrow d_1) , \cdots, (i_{m_s} \rightarrow d_s)\}$ do not necessarily form cycles.
Therefore, Eq.  \eqref{C_r} can be rewritten using Eq. \eqref{breakloop} and the graphs $\mathcal{H}(l)$ as
\begin{align}
  C_l 
  &= \sum_{i_1  \cdots i_l \in \mathcal{O}^l} \text{det} ( - L_{i_1 \cdots i_l} )
  = (-1)^l \sum_{i_1  \cdots i_l \in \mathcal{O}^l}  ~\sum_{ \substack{j_1  \cdots j_l \in\mathcal{S}^l\\ j_1  \cdots j_l\sim i_1  \cdots i_l}} 
  (-1)^{| \sigma(j_1, \cdots, j_l) |} L_{i_1 j_1} \cdots L_{i_l j_l}
  \notag \\
 & = (-1)^l \sum_{i_1  \cdots i_l \in \mathcal{O}^l}  ~\sum_{ \substack{j_1  \cdots j_l \in\mathcal{S}^l\\ j_1  \cdots j_l\sim i_1  \cdots i_l}} 
  (-1)^{| \sigma(j_1, \cdots, j_l) |}
   (-1)^s 
 L_{ i_{m_{s+1}} j_{m_{s+1}} } \cdots L_{ i_{m_l} j_{m_l} }
 \sum_{ \substack{d_1\neq i_{m_1}\\ \cdots\\ d_s \neq i_{m_s}}}  L_{ i_{m_1} d_1 } \cdots L_{ i_{m_s} d_s }  \notag\\
 &=(-1)^l \sum_{g\in\mathcal{H}(l)}a_g\Pi(g), \label{a_g}
   \end{align}
where the factors $a_g$ will be determined below.

  First we assume that the graph $g$ with the set of arcs $\{ (i_1 \rightarrow j_1), \cdots, (i_l \rightarrow j_l) \}$
  contains no cycles, i.e., $ g \in \mathcal{G} (n-l)\subset\mathcal{H}(l)$. 
  This can happen only if $s=l$ in Eq. \eqref{a_g}, i.e., $i_1=j_1$, ...,$i_l=j_l$. Therefore, $(-1)^{| \sigma(j_1, \cdots, j_l) |} =(-1)^0 = 1$, and
  \begin{equation*}
L_{i_1 i_1} \cdots L_{i_l i_l} = \left( - \sum_{d_1 \neq i_1} L_{i_1 d_1} \right) \cdots \left( - \sum_{d_l \neq i_l} L_{i_l d_l}  \right) = 
(-1)^l \sum_{d_1 \neq i_1, \cdots, d_l \neq i_l} L_{i_1 d_1} \cdots L_{i_l d_l}.
  \end{equation*}
Hence,  the  product   $ L_{i_1 d_1} \cdots L_{i_l d_l}$ corresponding to the graph $g$
  enters Eq. \eqref{a_g}  only once.
Therefore, $a_{g} = (-1)^l$ for all $g \in \mathcal{G} (n-l)$.  

Now we assume that the set of arcs 
$\{ (i_1 \rightarrow x_1), \cdots, (i_l \rightarrow x_l) \}$ 
of $g$ contains $N\ge 1$ cycles, i.e.,  
$g \in \mathcal{H}(l) \setminus \mathcal{G}(n-l)$.
In this case, the product $ L_{i_1 x_1} \cdots L_{i_l x_l}$, $i_p\neq x_p$, $1\le p\le l$, enters Eq. \eqref{a_g}
$2^N$ times either with  the  plus or minus sign. The number $2^N$ comes from the fact that each cycle in $g$ 
can be formed in two ways: $(i)$ by arcs 
corresponding to off-diagonal factors $L_{i_pj_p}$, $i_p\neq j_p$, in Eq. \eqref{a_g}, 
or $(ii)$
by arcs originating from the replacement of diagonal factors $L_{i_pi_p}$ in Eq. \eqref{a_g} with $-\sum_{d_p\neq i_p}L_{i_pd_p}$.
We will prove that $ L_{i_1 x_1} \cdots L_{i_l x_l}$, $i_p\neq x_p$, $1\le p\le l$, enters Eq. \eqref{a_g} with sign plus the
same number of times as it does with sign minus. This will imply that   $a_g=0$.
To do so, we show that for each entry of  $ L_{i_1 x_1} \cdots L_{i_l x_l}$
one can uniquely define another entry with an opposite sign.
Let $c$ be a cycle in $g$.
Consider two terms in Eq. \eqref{a_g} containing the product $L_{i_1x_1}\ldots L_{i_lx_l}$ that correspond to 
possibilities $(i)$ and $(ii)$ for the origin of a selected cycle $c$ in $g$, while all other factors corresponding to the arcs not in $c$ originate in the same way.
Let $\sigma_1(j^1_1,\ldots j^1_l)$ and $\sigma_2(j^2_1,\ldots,j^2_l)$ be the permutations in Eq. \eqref{a_g} corresponding to $(i)$ and $(ii)$ respectively,
and $s_1$ and $s_2$ be the corresponding numbers of fixed entries in $\sigma_1$ and $\sigma_2$ respectively.
If the cycle $c$ has length $|c|$, then 
\begin{equation} 
\label{perm}
(-1)^{|\sigma_1|} = (-1)^{|\sigma_2|}(-1)^{|c|-1},\quad (-1)^{s_1} = (-1)^{s_2+|c|}.
\end{equation}
Here we have used the known combinatorial fact that the parity of a permutation consisting of cycles $c_1$, ..., $c_k$ is $(-1)^{\sum_{j=1}^k|c_j|-1}$.
Therefore, the signs $(-1)^l(-1)^{|\sigma_1|}(-1)^{s_1}$ and $(-1)^l(-1)^{|\sigma_2|}(-1)^{s_2}$ preceding the corresponding products in Eq. \eqref{a_g} are 
opposite. This implies that all products corresponding to any graph $g\in \mathcal{H}(l) \setminus \mathcal{G}(n-l)$ cancel out, i.e., $a_g=0$. 

Therefore, $C_l = \sum_{g\in\mathcal{G}(n-l)}\Pi(g)$, i.e., Eq. \eqref{C_r_g} holds.
{ An example illustrating this cancellation is given at the end of Appendix \ref{App_A}.}

 \subsection*{Step 3:}
 Let us write the characteristic polynomial $P_L(t)$ in the form 
 \begin{align}
 P_L(t) = \text{det} ( tI -L ) = t ( t + \lambda_1 ) \cdots ( t + \lambda_{n-1} ),
 \end{align}
where $|\lambda_1|<|\lambda_2|<\ldots<|\lambda_{n-1}|$.
A simple calculation gives
 \begin{align}
 C_l = \sum_{i_1 \cdots i_r \in \mathcal{O}^l } \lambda_{i_1} \cdots \lambda_{i_l}. \label{C_r_lambda}
 \end{align}
Comparing Eqs. \eqref{C_r_g} and \eqref{C_r_lambda} we obtain
 \begin{equation}
 \label{ratio}
  \frac{C_l}{ C_{l-1}}
  =\frac{\sum_{i_1 \cdots i_l \in \mathcal{O}^l } \lambda_{i_1} \ldots \lambda_{i_l}}
  {\sum_{i_1 \cdots i_{l-1} \in \mathcal{O}^{l-1} } \lambda_{i_1} \ldots \lambda_{i_{l-1}}}  
 =\frac{\sum_{g\in\mathcal{G}(n-l)}\Pi(g)}{\sum_{g\in\mathcal{G}(n-l+1)}\Pi(g)},\quad l\le l\le n-1.
 \end{equation}
Since $L_{ij}$ are of the form $\kappa_{ij}\exp(-U_{ij}/\varepsilon)$, 
the sums in the enumerators and the denominators  are dominated by their
largest summands in Eq. \eqref{ratio}.
According to Assumption \ref{A3}, all optimal W-graphs are unique. Therefore,
\begin{equation}
\label{frac}
 \frac{C_l}{ C_{l-1}}  = \lambda_{n-l}(1+o(1)) 
  = \frac{\Pi(g^{\ast}_{n-l}) }{\Pi(g^{\ast}_{n-l+1}) } (1+o(1)), 
   \end{equation}
and Eq. \eqref{Alambda} immediately follows.

\end{proof}

{
\begin{example}
\normalfont
Let us illustrate the cancellation of terms in Eq. \eqref{a_g}.
Let $n=4$, $l=3$, and $\{i_1,i_2,i_3\}=\{1,2,3\}$.
Then the inner sum in Eq. \eqref{a_g} becomes
\begin{align*}
& \sum_{(j_1,j_2,j_3)\sim(1,2,3)}(-1)^{|\sigma(j_1,j_2,j_3)|}L_{1j_1}L_{2j_2}L_{3j_3} = \\
(j_1,j_2,j_3) = (1,2,3):~\sigma = 0,~s = 3:\qquad &-(L_{12}+L_{13}+L_{14})(L_{21}+L_{23}+L_{24})(L_{31}+L_{32}+L_{34}) \\
(j_1,j_2,j_3) = (1,3,2):~\sigma =1,~s = 1:\qquad&+(L_{12}+L_{13}+L_{14})L_{23}L_{32} \\
(j_1,j_2,j_3) = (2,1,3):~\sigma =1,~s = 1:\qquad&+L_{12}L_{21}(L_{31}+L_{32}+L_{34}) \\
(j_1,j_2,j_3) = (2,3,1):~\sigma =2,~s = 0:\qquad&+L_{12}L_{23}L_{31} \\
(j_1,j_2,j_3) = (3,2,1):~\sigma =1,~s = 1:\qquad&+L_{13}(L_{21}+L_{23}+L_{24})L_{31} \\
(j_1,j_2,j_3) = (3,1,2):~\sigma =2,~s = 0:\qquad&+L_{13}L_{21}L_{32}.
\end{align*}
For each product term of the form $L_{ix_1}L_{2x_2}L_{3x_3}$, $x_q\in\{1,2,3,4\}$, one can draw a graph with 
the set of vertices $\mathcal{S} = \{1,2,3,4\}$ and the set of arcs $\mathcal{A} = \{1\rightarrow x_1,2\rightarrow x_2,3\rightarrow x_3,4\rightarrow x_4\}$.
One can check that all terms corresponding to graphs with no cycles are encountered just once and only in the terms originating from $(j_1,j_2,j_3) = (1,2,3)$.
All of them are preceded by the sign ``-".  This corresponds to $a_g = (-1)^l = (-1)^3$. 
On the contrary, each term corresponding to a graph with cycles (in this example, there can be at most $N=1$ cycle), 
is encountered exactly twice ($2^1 = 2$): once it comes from the product corresponding to $(j_1,j_2,j_3) = (1,2,3)$ with sign ``-", 
and once it comes from some non-identical permutation with sign ``+".  Hence, all of such term cancel out.
\end{example}
}

\section{Proof of Theorem \ref{theorem:m1} } \label{App_B}

Prior to start proving Theorem \ref{theorem:m1}, we prove some auxiliary facts and introduce some use useful definitions.
The proof of Theorem \ref{theorem:m1} will exploit the following lemmas. 
\begin{lemma}
\label{lemma:m1}
Suppose the function {\tt FindTgraphs} is run on a graph $G(\mathcal{S},\mathcal{A},\mathcal{U})$ satisfying Assumptions \ref{A1} and \ref{A2}:
{\tt FindTgraphs}$(r,k,G,\Gamma,\mathcal{B}')$.
Let $c$ be a cycle detected in the graph $\Gamma$ at step $k$. Suppose the weights of all outgoing arcs with tails in $c$  
and heads not in $c$ 
are modified according to the update rule
$$
U^{new}_{ij} = U_{ij} -U_{\mu(i)} +\gamma_k.
$$ 
Then 
$U_{ij}^{new} \ge \gamma_k$ for all $i\in c$, $j \notin c$, and 
$U_{ij}^{new}  = \gamma_k$ if and only if
$U_{ij}=U_{\mu(i)}$.
\end{lemma}

\begin{proof}
The fact that $U_{ij}^{new} \ge \gamma_k$ for all $i\in c$ and $j \notin c$ follows from the fact that $U_{ij}\ge U_{\mu(i)}$. The equality takes place 
 if and only if $U_{\mu(i)} = U_{ij}$, i.e., if $i\rightarrow j$ is another min-arc from $i$. 
\end{proof}

\begin{corollary}
\label{cor1}
Suppose the function {\tt FindSymTgraphs} is run on a graph $G(\mathcal{S},\mathcal{A},\mathcal{U})$ satisfying Assumptions \ref{A1} and \ref{A2}:
{\tt FindSymTgraphs}$(r,p,G,T,\mathcal{B})$.
Let $C$ be a closed communicating class detected in the graph $T$ at step $p$. Suppose the weights of all outgoing arcs with tails in $C$  
and heads not in $C$ 
are modified according to the update rule
$$
U^{new}_{ij} = U_{ij} - U_{\min}(i) +\theta_p.
$$ 
Then 
$U_{ij}^{new} > \theta_p$ for all $i\in C$, $j \notin C$.
\end{corollary}

We will denote by $\mathcal{S}(v_{c})$  the subset of vertices of the original graph $G(\mathcal{S},\mathcal{A},\mathcal{U})$
contracted into the super-vertex $v_{c}$.

\begin{corollary}
\label{cor2}
{Suppose the function {\tt FindTgraphs} is run on a graph $G(\mathcal{S},\mathcal{A},\mathcal{U})$ satisfying Assumptions \ref{A1} and \ref{A2}:
{\tt FindTgraphs}$(r,k,G,\Gamma,\mathcal{B}')$. 
Let $c_1$, ..., $c_N$ be a sequence of  cycles created after the addition of arcs of weights
$\gamma_1 \le \gamma_2 \le \ldots \le \gamma_N$  respectively, such that $\mathcal{S}(v_{c_N})\supset\mathcal{S}(v_{c_l})$ for all $l < N$.
Let $i$ and $j$ be vertices of $G$ such that $(i\rightarrow j)\in\mathcal{A}$ and  $i\in \mathcal{S}(v_{c_l})$, $l < N$, and $j\notin \mathcal{S}(v_{c_N})$. 
Let $U^{(l)}$ the (possibly modified) weight of the arc $i\rightarrow j$ after the creation of the cycle $c_l$.
Then  $U^{(N)}_{ij} \ge \gamma_N$ and $U^{(N)}_{ij}   = \gamma_N$ if and only if $(i\rightarrow j)$ is a min-arc from $i$.
}
\end{corollary}

{
\begin{corollary}
\label{cor3}
Suppose the function {\tt FindTgraphs} is run on a graph $G(\mathcal{S},\mathcal{A},\mathcal{U})$ satisfying Assumptions \ref{A1} and \ref{A2}:
{\tt FindTgraphs}$(r,k,G,\Gamma,\mathcal{B}')$. 
Let $c_1$, ..., $c_N$ be a sequence of  nested cycles created after the addition of arcs of weights
$\gamma_1 \le \gamma_2 \le \ldots \le \gamma_N$  respectively, i.e., $\mathcal{S}(v_{c_1})\subset\mathcal{S}(v_{c_2})\subset\ldots\subset\mathcal{S}(v_{c_N})$.
Suppose each set of vertices $\mathcal{S}(v_{c_l})$ contains a vertex $x$ with an outgoing arc $x\rightarrow y$ 
such that $U_{xy} = U_{\min}(x)$, but $y\notin \mathcal{S}(v_{c_l})$, $l = 1,\ldots,N-1$.
Suppose there is an arc $(i\rightarrow j)\in\mathcal{A}$ such that $i\in \mathcal{S}(v_{c_1})$, $j\notin \mathcal{S}(v_{c_N})$.
Then  $U^{(N)}_{ij} = U_{ij}-U_{\min}(i) + \gamma_{N}$.
\end{corollary}
\begin{proof}
For all $1\le l\le N-1$ we have $U_{\min}(v_{c_l}) = \gamma_l$. Therefore, 
\begin{align*}
U^{(1)}_{ij} &= U_{ij}-U_{\min}(i) + \gamma_1,\\
U^{(2)}_{ij} &= U_{ij}^{(1)} - U_{\min}(v_{c_1}) + \gamma_2 = U_{ij}-U_{\min}(i) + \gamma_2,\\
&\ldots\\
U^{(N)}_{ij} &= U_{ij}^{(N-1)} - U_{\min}(v_{c_{N-1}}) + \gamma_N = U_{ij}-U_{\min}(i) + \gamma_N.
\end{align*}
\end{proof}
}
\begin{lemma}
\label{lemma:m2}
Let $\mathcal{S}'\subset \mathcal{S} $ be a subset of vertices of a directed graph $G(\mathcal{S},\mathcal{A})$.
Suppose every vertex in $\mathcal{S}'$ has at least one outgoing arc. Then
\begin{enumerate}[$(i)$]
\item 
if all arcs with tails in $\mathcal{S}'$ have heads also in $\mathcal{S}'$, then there is at least one directed cycle formed 
by the arcs with tails in $\mathcal{S}'$;
\item
if the arcs with tails in $\mathcal{S}'$ form no directed cycles, then at least one of them must head in $\mathcal{S}\backslash\mathcal{S}'$.
\end{enumerate}
\end{lemma}
\begin{proof} 
Let us select one outgoing arc for each vertex in $\mathcal{S}'$ and denote the set of the selected arcs  by $\mathcal{A}'$.
If all arcs in $\mathcal{A}'$ head in $\mathcal{S}'$ then $|\mathcal{A}'| =  |\mathcal{S}'|$ in the graph $G':=G'(\mathcal{S}',\mathcal{A}')$.
Hence $G'$ cannot be a directed forest. Hence it contains  at least one cycle which proves Statement $(i)$. Statement $(ii)$ is  the negation of $(i)$.
\end{proof}


\begin{proof} (Theorem \ref{theorem:m1}.)
We start with Statement 4 because its proof is the shortest. A vertex $i$ is an absorbing state of $T_p$
 if and only if 
the weight of min-arcs from $i$ in the original graph $G$ is greater than $\theta_p$. 
In turn, this happens if and only if $i$ has no outgoing arc in $\Gamma_{K_p}$, i.e., 
$i$ is absorbing in $\Gamma_{K_p}$.

{\bf Auxiliary Statement:} \emph{
At the end of steps $K_p$ and $p$  { of Algorithms 1 and 2 respectively  for all $p\ge 0$} we have:
$\mathcal{B}'\subseteq \mathcal{B}$ and the sets of distinct values in $\mathcal{B}'$ and $\mathcal{B}$ coincide.
}

Statements 1, 2, and 3 
and the auxiliary statement
will be proven
by induction in the recursion level $r$ of Algorithm 2.

{\bf Basis.} The initial graphs $T_0 = T_0(\mathcal{S},\emptyset,\emptyset)$ and $\Gamma_0=\Gamma_0(\mathcal{S},\emptyset,\emptyset)$ in Algorithms 1 and 2 coincide. 
Furthermore, as the initializations in Algorithms 1 and 2 are complete, 
$\mathcal{B}'\subseteq \mathcal{B}$, and the sets of distinct arc weights in the buckets $\mathcal{B}'$ and $\mathcal{B}$ are the same.
This gives us the induction basis. 

{\bf Induction Assumption.} Assume that at step $p_r$ of Algorithm 2 and the corresponding step $K_{p_r}$ of Algorithm 1 we have: 
\begin{itemize}
\item
$\mathcal{B}'\subseteq\mathcal{B}$;
\item
 the set of distinct arc weights in $\mathcal{B}$ and $\mathcal{B}'$ are the same; 
 \item
 all closed communicating classes 
are contracted into single super-vertices in $T^{(r)}_{p_r}$ and $\Gamma^{(r')}_{K_{p_r}}$, where $r'$ is the recursion level in Algorithm 1 at the end of step $K_{p_r}$;
{ furthermore, the set of vertices of $T^{(r)}_{p_r}$ and $\Gamma^{(r')}_{K_{p_r}}$ coincide, and each super-vertex $v$ of $T^{(r)}_{p_r}$ has 
the corresponding super-vertex $v'$ of $\Gamma^{(r')}_{K_{p_r}}$ such that $\mathcal{S}(v) = \mathcal{S}(v')$};
\item
$\Gamma_{K_{p_r}}$ is a subgraph of $T_{p_r}$;
\item
the sets of distinct values in $\{\theta_p\}_{p=1}^{p_r}$ and $\{\gamma_{k}\}_{k=1}^{ K_{p_r}}$ coincide.
\end{itemize}
To prove the induction step, we need to show that Statements 1, 2, 3, and the auxiliary statement hold up to $p=p_{r+1}$.

{\bf Induction Step.} 
\begin{enumerate}
\item
The induction assumptions imply that  
all graphs $\Gamma_{k}$, $K_{p-1}<k\le K_p$,
are subgraphs of $T_p$ for all $p$ such that the recursion levels remain $r'$  and $r$ in Algorithms 1 and 2 respectively. 
For such $p$, the graphs $T_p^{(r)}$ and $\Gamma_{K_p}^{(r')}$ are built by adding  arcs from 
buckets $\mathcal{B}$ and $\mathcal{B}'$ respectively, while no new arcs are added to these buckets and no arc weights are modified.
Hence, Statements 1, 2, 3, and the auxiliary statement hold for all  such $p$. Therefore, if no cycles are encountered by Algorithm 1
at steps $K_{p_r}<k\le K_{p_{r+1}-1}$, Statements 1, 2, 3, and the auxiliary statement hold for $K_{p_r}<k\le K_{p_{r+1}-1}$.

\item
Now we show that Statements 1, 2, 3, and the auxiliary statement hold
independent of whether or not cycles were encountered in Algorithm 1 at some $K_{p_r}<k\le K_{p_{r+1}-1}$.
Note that a cycle can be formed in Algorithm 1  at some step $K_p<k\le K_{p+1}$ where $p_r<p<p_{r+1}$ only if
an open communicating class is formed at step $p$ of Algorithm 2.

Since no arcs are added to the bucket $\mathcal{B}$ by Algorithm 2 at steps $p_r<p<p_{r+1}$, the graphs $T_p^{(r)}$ 
consist of all min-arcs
of the graph $G^{(r)}$ of weights $\le \theta_{p_{r+1}-1}$. 
Now we consider Algorithm 1 for steps $K_{p_r}<k\le K_{p_{r+1}-1}$ and prove Statements 1, 2, 3,
and the auxiliary statement by induction in the number of cycles.

Let $c_1$ be the first cycle created in $\Gamma^{(r')}$ at step $k_1$ after the addition of an arc of weight $\theta_p$.
Since $\Gamma_{k_1}^{(r')}$ is a subgraph of $T_{p}^{(r)}$,
the cycle $c_1$ must be a subclass of an open communicating class $C$ created in $T^{(r)}$ 
after the addition of a set of arcs of weight $\theta_p$.
Therefore, $c_1$ is an open communicating class in $T_p^{(r)}$. 
This means the set of min-arcs with tails in $c_1$ and heads not in $c_1$ in not empty. 
All these min-arcs are in $T_p^{(r)}$, and none of them is in $\Gamma_{k_1}^{(r')}$.
By Lemma \ref{lemma:m1},  the weights of these min-arcs become $\theta_p$ after the modification. 
One of them is picked and added to the bucket $\mathcal{B}'$.
Hence, this arc will be added to $\Gamma^{(r'+1)}$ at some $k_1<k\le K_p$. This allows us to conclude that at least one more arc of weight $\theta_p$
will be removed from $\mathcal{B}'$ after the cycle $c_1$ is formed. Hence, $c_1$ is not a closed communicating class in $\Gamma_{K_p}^{(r')}$.
Therefore, Statements 1, 2, 3 and the auxiliary statement hold for $K_{p_r}<k\le\min\{k_2,K_{p_{r+1}-1}\}$.

Assume that Statements 1, 2, 3,
and the auxiliary statement  hold  up to step $k_{N} < K_{p_{r+1} -1}$ of Algorithm 1.
Let  cycle $c_N$ be encountered at step $k_N$ in Algorithm 1 after the addition of an arc of weight $\theta_p$.
The set of vertices $\mathcal{S}^{(r')}(v_{c_N})$ is a subclass of an open communicating class $C$ of the graph $T_p$ by the induction assumption. 
Hence  the set of the min-arcs
with tails in $\mathcal{S}^{(r')}(v_{c_N})$ and heads not in $\mathcal{S}^{(r')}(v_{c_N})$ is not empty. All of these arcs are in $T_p^{(r)}$
but not in $\Gamma_{k_N}^{(r')}$.
By Corollary \ref{cor2}, their weights will become
$\theta_p$ during step $k_N$ of Algorithm 1. 
One of these min-arcs will be added to the bucket $\mathcal{B}'$ and then removed from it 
at some step $k_N<k\le K_{p}$. Hence the cycle $c_N$ cannot be a closed communicating class in $T_p^{(r)}$.
This proves Statements 1, 2, 3, and the auxiliary statement  for all $K_{p_r}<k\le K_{p_{r+1}-1}$.

\item
Now we show that Statements 1, 2, 3, and the auxiliary statement hold for $K_{p_{r+1}-1}+1\le k\le K_{p_{r+1}}$.
Let $C$ be the closed communicating class formed in the graph $T_{p_{r+1}}^{(r)}$
after the addition of the set of min-arcs of weight $\theta_{p_{r+1}}$.
Therefore, all min-arcs from the vertices in $C$ head in $C$. 
After contracting $C$ into a single super-vertex $v_C$, the weight of min-arcs from it will be
\begin{equation}
\label{te1}
\min_{i\in C,~j\notin C}[U_{ij}-U_{\min}(i) +\theta_{p_{r+1}}] >\theta_{p_{r+1}}.
\end{equation}
(Here $U_{\min}(i)$ is the weight of min-arcs from $i$ in the graph $G^{(r)}$.)
Let the recursion level at the end of step $K_{p_{r+1}-1}$ in Algorithm 1 be $r''$, and $C'$ be the set of vertices in
$\Gamma^{(r'')}$ corresponding to $C$. 
If no cycles were formed in Algorithm 1 with vertices in $C$ then $C'=C$.
Otherwise, some of the vertices of $C$ are contacted into super-vertices in $\Gamma^{(r'')}$. 
In this case, the subset of vertices of $C$ contracted
into super-vertices is an open communicating subclass of $C$.
Let us show that $C'$ is a closed communicating class of $\Gamma_{K_{p_{r+1}}}^{(r'')}$.

Lemma \ref{lemma:m2} implies that the min-arcs from $C'$ all heading in $C'$ form at least one cycle $c$. Consider two cases.

{\bf Case 1: cycle $c$
includes whole  $C'$.}
Then by Corollary \ref{cor3} the weight of min-arcs from the vertex $v_c$ in Algorithm 1 will be given by Eq. \eqref{te1}, 
i.e., the same as it is in Algorithm 2,
and one of those min-arcs will be added to $\mathcal{B}'$. 
Hence Statements 1, 2, 3, and the auxiliary statement hold at $p=p_{r+1}$ and $K_{p_{r+1}-1}<k\le K_{p_{r+1}}$.

{\bf Case 2: the cycle $c$ does not include all vertices from $C'$.} 
Since $C$ does not contain closed communicating subclasses,  
the set of min-arcs in $G^{(r)}$ with tails in $c$ and heads not in $c$ is not empty.
By Corollary \ref{cor3}, these min-arcs will be the min-arcs from $v_c$, and their weights will be $\theta_p$.
One of these min-arcs will be added to $\mathcal{B}'$ and then removed.
Hence  the min-arcs from $(C\backslash\mathcal{S}^{(r)}(v_c))\cup \{v_c\}$ head to $(C\backslash\mathcal{S}^{(r)}(v_c))\cup \{v_c\}$.
By Lemma \ref{lemma:m2}, they form at least one cycle $c$. Again, there are two options: either $c$ includes all vertices of $C$ or not.
In the former case, using the argument from Case 1, we prove the induction step. In the latter case, we use the argument from Case 2.
Repeating this argument at most $|C|-1$ number of times (as each new cycle includes at least one more vertex of $C'$ in comparison with the previous one), 
we obtain a cycle including all vertices of $C$.

Repeating this argument for all closed communicating classes formed in $T^{(r)}$ at step $p_{r+1}$, we conclude that  
all closed communicating classes encountered in Algorithm 2 at steps $p_{r+1}$  will be contracted 
into single super-vertices by both Algorithms 1 and  2, and arcs of the same weight will be added to $\mathcal{B}$ and $\mathcal{B}'$.
Hence the induction step is proven. This completes the proof of Theorem \ref{theorem:m1}.
\end{enumerate}
\end{proof}

\end{appendices}


\end{document}